\documentclass[11pt,a4paper]{article}

\usepackage[T1]{fontenc}
\usepackage[cp1250]{inputenc}
\usepackage{times}

\usepackage[leqno]{amsmath} 
\usepackage{amsthm,amsfonts,amssymb} 
\usepackage{bm} 

\input xy  \xyoption{all}

\usepackage{color}
\usepackage{geometry}
\usepackage{wrapfig}

\usepackage{graphicx} 

\usepackage{enumerate} 
\usepackage{hyperref} 

\newcommand{\change}{}
\newcommand{\newnew}[1]{}

\newtheorem{thm}{Theorem}[section]
\newtheorem{prop}[thm]{Proposition}
\newtheorem{lem}[thm]{Lemma}
\newtheorem{cor}[thm]{Corollary}

\theoremstyle{definition}
\newtheorem{remark}[thm]{Remark}
\newtheorem{definition}[thm]{Definition}
\newtheorem*{conclusions}{Conclusion}

\numberwithin{equation}{section}

\newcommand{\X}{\mathfrak{X}}
\newcommand{\R}{\mathbb{R}}

\newcommand{\del}{\delta}
\newcommand{\g}{\mathfrak{g}} 
\newcommand{\h}{\mathfrak{h}} 

\newcommand{\ra}{\rightarrow}
\newcommand{\lra}{\longrightarrow}

\newcommand{\dd}{\mathrm{d}} 
\newcommand{\T}{\mathrm{T}} 
\newcommand{\V}{\mathrm{V}} 
\newcommand{\HH}{\mathrm{H}} 
\newcommand{\pa}{\partial} 

\newcommand{\Ann}{\operatorname{Ann}} 
\newcommand{\spann}{\operatorname{span}}

\newcommand{\wt}{\widetilde}
\newcommand{\wh}{\widehat}

\newcommand{\TT}{\mathcal {T}} 
\newcommand{\WW}{\mathcal {W}} 
\newcommand{\C}{\mathcal {C}} 
\newcommand{\PP}{\mathcal {P}} 
\newcommand{\jet}{\mathrm{T}} 
\newcommand{\Jet}[1]{\frac{\partial}{\partial #1}\Big|_{0}} 
\newcommand{\jjet}[1]{\frac{\partial}{\partial #1}\big|_{0}} 
\newcommand{\ad}{\operatorname{ad}} 
\newcommand{\A}[2]{A^{#1}_{\ #2}} 
\newcommand{\br}[3]{C^{#1}_{\ #2 #3}} 
\newcommand{\bra}[3]{c^{#1}_{\ #2 #3}} 
\newcommand{\delt}[2]{\del^{#1}_{\ #2}} 
\newcommand{\B}{\mathcal{B}} 
\newcommand{\RR}{\mathcal{R}} 
\newcommand{\BL}{\mathrm{H}L} 
\newcommand{\FL}{\mathrm{F}L} 
\newcommand{\cform}{\bm{\omega}} 
\newcommand{\Cform}{\bm{\dot\omega}} 

\newcommand{\und}{\underline}
\def\<#1>{\left\langle #1\right\rangle} 

\newcounter{example}
\newenvironment{example}[1]{\refstepcounter{thm} \noindent \textbf{Example \arabic{section}.\arabic{thm}. #1}}{\smallskip}

\title{A comparison of vakonomic and nonholonomic dynamics with applications to non-invariant Chaplygin systems\footnote{This research was supported by the National Science Center under the grant DEC-2011/02/A/ST1/00208 ``Solvability, chaos and control in quantum systems''.}}

\author{
Michał Jóźwikowski\footnote{email: \texttt{mjozwikowski@gmail.com}}\\ \emph{Institute of Mathematics, Polish Academy of Sciences}\\[0.5cm]
Witold Respondek\footnote{email: \texttt{witold.respondek@insa-rouen.fr}}\\ \emph{Normandie Universit\'{e}, INSA de Rouen, LMI}}

\begin{document}
\maketitle
\begin{abstract}
We study relations between vakonomically and nonholonomically constrained Lagrangian dynamics for the same set of linear constraints. The basic idea is to compare both situations at the level of variational principles, not equations of motion as has been done so far. The method seems to be quite powerful and effective. In particular, it allows to derive, interpret and generalize many known results on non-Abelian Chaplygin systems. We apply it also to a class of systems on Lie groups with a left-invariant constraints distribution. Concrete examples of the unicycle in a potential field, the two-wheeled carriage and the generalized Heisenberg system are discussed. 
\end{abstract}

\section{Introduction}\label{sec:intro}

\paragraph{State of research.} 
The problem of obtaining the equations of motion of a mechanical system in the presence of constraints has a long history and has gained attention of many prominent researchers (see e.g., \cite{Leon_2012} for a brief historical discussion, compare also \cite{Arnold_Koz_Nies_2010}). In general, constraints are introduced by specifying a submanifold (or simply a subset) $C\subset \T Q$ of the tangent bundle of the configuration manifold $Q$. Typically $C$ is assumed to be a (non-integrable) distribution (we speak of the \emph{linear case} in such a situation). In principle, there are two non-equivalent ways of generating the dynamics under the constraints $C$.\footnote{Throughout this work we will use the term \emph{dynamics} as a synonym of the set of all trajectories of a given system.} They are known as \emph{nonholonomic} and \emph{vakonomic} (i.e., \emph{variational of axiomatic kind} \cite{Arnold_Koz_Nies_2010}) methods. Nonholonomic dynamics are believed to be the ones describing the real physical movement of the constrained system \cite{Lewis_Murray_1995}. They are obtained by means of \emph{Chetaev's principle} (in the linear case we speak rather about \emph{d'Alembert's principle} or the \emph{principle of virtual work}). On the other hand, vakonomic dynamics are related to optimal control theory and are derived as a solution of a constrained variational problem.

The word \emph{nonholonomic} is often used in two contexts. As an adjective describing the method of generating the dynamics by means of Chetaev's (d'Alembert's) principle and as a substitute of the word \emph{non-integrable} in the description of the constraints distribution $C$. Because of the latter, vakonomic dynamics are sometimes called \emph{variational nonholonomic} \cite{Cardin_Favretti_1996,Fernandez_Bloch_2008}. To avoid possible confusions, in this paper we will reserve the name \emph{nonholonomic} for its meaning related with Chetaev's principle. Let us remark that by the \emph{dynamics} we will understand the set of all trajectories of the considered system.

The problem of comparison of the two (i.e., vakonomic and nonholonomic) non-equivalent ways of generating the constrained dynamics has  been addressed since long time by the scientific community (see Sec. 3 in \cite{Leon_2012} and the references therein). Their non-equivalence can be easily observed at the level of the equations of motion -- in many cases vakonomic dynamics are much richer than the nonholonomic ones. Therefore, it is natural to address the following question: 
\begin{equation}\label{main_question}\tag{Q1}
\text{\textbf{Are the nonholonomic dynamics a subset of the vakonomic ones for a given constrained system?}}
\end{equation} 
It is also interesting to formulate this problem for a particular trajectory:
\begin{equation}\label{main_question_1}\tag{Q2}
\text{\textbf{Is a given nonholonomic trajectory a vakonomic one?}}
\end{equation}  
Some authors ask also the inverse of the latter:
\begin{equation}\label{main_question_2}\tag{Q3}
\text{\textbf{Is a given vakonomic trajectory a nonholonomic one?}}
\end{equation}  
In this paper we will be concerned with giving answers to these questions {\change for a relatively wide class of non-invariant Chaplygin systems}.  We will refer to them as to the \emph{comparison problems}.
Let us remark that in general it is easier to find a set of necessary and sufficient conditions for a given trajectory of a system to answer \eqref{main_question_1} or \eqref{main_question_2} than to find such conditions for the whole system to answer \eqref{main_question} (although a positive answer to \eqref{main_question_1} on every nonholonomic trajectory provides an obvious sufficient condition). The reason is basically that restricting attention to a single trajectory helps to avoid certain global problems, such as the existence of solutions, etc.

Although it was observed already at the end of $19^\mathrm{th}$ century that the nonholonomic and vakonomic methods lead, in general, to different trajectories, the problem of their comparison was first stated in \cite{Lewis_Murray_1995} and \cite{Cardin_Favretti_1996} (the former being inspired by an example of a unicycle moving on the plane discussed in \cite{Bloch_Crouch_1993}). In the terminology of \cite{Fernandez_Bloch_2008} nonholonomic systems answering positively question \eqref{main_question} are called \emph{conditionally variational}, whereas systems possessing only some nonholonomic trajectories that answer positively question \eqref{main_question_1} -- \emph{partially conditionally variational}. Crampin and Mestdag \cite{Crampin_Mestdag_2010} used terms \emph{weak} and \emph{strong} \emph{consistency} in a similar, but slightly different context. Some authors (\cite{Cortes_Leon_Inn_2003,Favretti_1998,Fernandez_Bloch_2008}) speak about \emph{equivalence} of nonholonomic and vakonomic dynamics for systems positively answering \eqref{main_question}, although at the level of dynamics there can be at most an inclusion, not equivalence. 

In general the sets of nonholonomic and vakonomic trajectories are not related by the inclusion of type \eqref{main_question}  -- a sphere rolling on a rotating table provides a natural example of a system whose generic nonholonomic trajectory cannot be a vakonomic one (see \cite{Favretti_1998,Lewis_Murray_1995}). According to \cite{Fernandez_Bloch_2008}, Rumianstev \cite{Rumianstev_1978} was the first to answer question \eqref{main_question_1}. His answer, however, requires an explicit knowledge of the Lagrange multipliers of the vakonomic trajectories, hence in fact the solutions of vakonomically constrained problem.
So far all approaches to the comparison problems \eqref{main_question}--\eqref{main_question_2} used the method of comparing the nonholonomic and vakonomic equations of motion. 
Below we list the most important existing contributions to this field. The first three of them concern \emph{Chaplygin systems} i.e., Lagrangian systems defined on a principal $G$-bundle with linear constraints given by a horizontal distribution of a $G$-principal connection. The last one is more general, yet the answer is presented in the form of a non-decisive algorithm (concrete criteria are derived also for the Chaplygin case only). Let us note that most of those results require certain regularity assumptions about the Lagrangian that are needed, in principle, to present implicit equations of motion in an explicit form.
\begin{itemize}
\item In Thm. 3.1 in \cite{Favretti_1998}, Favretti gives an answer to \eqref{main_question_1} and \eqref{main_question_2} (in a slightly more general setting for affine distributions). He constructs explicitly the vakonomic  multiplier and, by comparing nonholonomic and vakonomic equations of motion, gives an answer in terms of the curvature of the constraints distribution. His assumptions are the $G$-invariance and certain regularity of the Lagrangian (the latter being sufficient for the existence of a momentum map and for the nonholonomic multiplier to be given by a time-dependent section over the configuration space). For the constrained geodesic problem Favretti gave, in Thm. 3.2 of \cite{Favretti_1998}, sufficient condition for the positive answer to \eqref{main_question}. His assumptions are, however, very strong: the constraints distribution is totally geodesic and, additionally, the perpendicular distribution is integrable. 
As a particular example he showed that a two-wheeled carriage is a constrained mechanical system for which every nonholonomic trajectory is a vakonomic one.
    
\item Fernandez and Bloch in  \cite{Fernandez_Bloch_2008} were able to explicitly find vakonomic multipliers for a simple (yet relatively wide) class of \emph{Abelian Chaplygin systems} (and under an additional assumption that the Lagrangian is of mechanical type and regular). In consequence, they were able to apply Rumianstev's method to these systems. The resulting answer to \eqref{main_question_1} was given in terms of the geometry of the constraints distribution and vertical derivatives of the Lagrangian. For a more general class of \emph{non-Abelian Chaplygin systems} their approach gives a partial answer to \eqref{main_question_1}. Fernandez and Bloch show, in particular, that examples of a unicycle and a two-wheeled carriage answer \eqref{main_question} positively. They, however, claim incorrectly that the examples of the Heisenberg system, the Chaplygin skate, and an invariant system on $SO(3)$ give a negative answer to \eqref{main_question}. We comment and correct their  statement in Remark~\ref{rem:Bloch_Fernandez}. Let us remark that Fernandez and Bloch study also an interesting question about the relation of problem \eqref{main_question} with that of the existence of an invariant measure. 

\item 
The paper of Crampin and Mestdag \cite{Crampin_Mestdag_2010} attacks the problem of comparison from a slightly different angle. Their main idea is to express the constrained dynamics by means of certain vector fields on $\T Q$. The comparison problem can now be solved by comparing those fields. To do so, they use an ingenious technical tool of \emph{anholonomic frames} (i.e., local frames adapted to the constraints distribution), which considerably simplifies the calculations. As a result they were able to extend the results of \cite{Fernandez_Bloch_2008} to non-Abelian Chaplygin systems (actually regaining some results from \cite{Favretti_1998}). Crampin and Mestdag work under technical assumptions about the regularity of the Lagrangian (required if we want the Lagrangian dynamic to be locally the flow of a vector field) and restrict themselves to specific (yet quite general) classes of vakonomic multipliers (defined by a section over the configuration manifold). Due to these restrictions the relation of their results with the original problem \eqref{main_question} is not obvious. Clearly if the multipliers can be determined (as turns out to be the case for Chaplygin systems) one answers \eqref{main_question}. Nevertheless, they provide criteria for answering \eqref{main_question_1} and \eqref{main_question_2} for particular trajectories.

\item Cortes, de Leon, Martin de Diego and Martinez \cite{Cortes_Leon_Inn_2003}  formulated both vakonomic and nonholonomic mechanics in a presymplectic framework similar to Skinner-Rusk formalism. To compare both dynamics one has to apply a constraint algorithm and compare the resulting final constraints submanifolds. Using that method the authors re-obtained Theorems 3.1 and 3.2 of Favretti \cite{Favretti_1998} under weaker assumptions. They also studied a few well-known examples including the unicycle.   
\end{itemize}
We discuss the relation of the above results with our work in Remark~\ref{rem:comparison}.


\paragraph{Crucial ideas.} 
Our basic ideas are rooted in the conceptual works of Tulczyjew \cite{Tul_2004} on statics of physical systems. According to him, equilibria of such a system are determined by a \emph{variational principle} which involves possible \emph{configurations} of the system, (infinitesimal) \emph{processes} (movements)  that the system can be subject to, and its \emph{reactions} to such movements (i.e., work that must be performed in order to change the configuration). Paper \cite{Tul_2004} deals mainly with statics, yet this limitation should be understood as a simple particular realization of a very general philosophy. For example, to treat Lagrangian mechanics we should translate configurations to \emph{admissible trajectories}, infinitesimal movements to \emph{admissible variations} and reactions to the \emph{change of the action functional} along these variations. The philosophy of Tulczyjew gives also a new insight into the idea of constraints: these are simply \emph{restrictions} of the sets of configurations and/or infinitesimal movements of the system. It also changes the perspective of looking at the equations of motion: we should understand them not as constituting the system, but merely as reflections of the underlying variational principle, which is the basis of every study. Actually this point of view is not new, just ''out of fashion'' at present. It can be traced back in time as far as to Lagrange himself (see the first comment in Koiller's paper \cite{Koiller_1992}).

Describing nonholonomic and vakonomic Lagrangian dynamics in the spirit of Tulczyjew's variational principles is elementary. Given a constraints submanifold $C\subset \T Q$ we define admissible trajectories for both situations as these paths $\gamma(t)\in C$ that are the tangent lifts of the true base paths. The reactions are again common for both situations and given by the changes of the action functional (defined by means of the Lagrangian). The only difference appears at the level of admissible variations: in the nonholonomic case they are described by Chetaev's principle (for linear constraints they are performed in the directions of $C$), whereas in the vakonomic case we consider these variations that respect the constraints. In this way we are able to present both the nonholonomic and the vakonomic dynamics within the common framework of Tulczyjew's variational principles. This observation should be attributed to Gracia, Martin and Munos \cite{Gracia_Martin_Munos_2003}. Similar remarks have been made before (see e.g., \cite{Cardin_Favretti_1996}), yet the authors of \cite{Gracia_Martin_Munos_2003} were, in our opinion, the first to use systematically that common nature of nonholonomic and vakonomic dynamics. In this context one should also mention later works \cite{GG_2008,Guo_Liu_Inn_2007,Leon_2012}.

The main idea of this paper comes directly from \cite{Tul_2004} and \cite{Gracia_Martin_Munos_2003}. Namely, \textbf{we compare nonholonomic and vakonomic dynamics at the level of the corresponding variational principles} (in fact, it is enough to concentrate on admissible variations) not equations of motion, as is usually done. In this way we get to the point where the actual differences between these two dynamics come from. Differences in the equations of motions are just an emanation of these basic differences. And, after all, questions \eqref{main_question}-\eqref{main_question_2} are not about the equality of equations, but the equality of trajectories. 
From the technical side we must admit a strong inspiration from \cite{Crampin_Mestdag_2010}. The idea of adapting the frames to the constraints distribution allowed us to treat Chaplygin systems easily.


\paragraph{Our results.}
In Section~\ref{sec:var_calc} we present the philosophy of Tulczyjew's variational principles \cite{Tul_2004} applied to Lagrangian dynamics. We introduce \emph{restricted variational principles} that correspond to constrained systems and discuss the particular cases of nonholonomically and vakonomically constrained Lagrangian dynamics, proceeding in accordance with \cite{Gracia_Martin_Munos_2003}. Concerning the comparison problems we prove abstract Proposition~\ref{prop:constraints}, which states that \emph{restricting the variational principle} results in extending the set of extremals. A slight generalization in Proposition~\ref{prop:two_var_prob} allows us to incorporate symmetries of the Lagrangian into the game: we can compare the extremals of two different variational principles, provided that we can compare the admissible variations up to the symmetries of the Lagrangian. This possibly gives a new insight into an interesting problem to study relations between the constraints and the symmetries of the system (see \cite{Leon_2012}, Sec. 4.4 and the references therein). 
\medskip

The results of Section~\ref{sec:var_calc} have a very formal and abstract character, and may seem to introduce superfluous formalism or just vainly  reformulate things that are well-known. To show that it is not so, we present, in Section~\ref{sec:chaplygin}, their application to the comparison problems \eqref{main_question_1} and \eqref{main_question_2} for a broad class of systems with linear constraints, namely to \emph{non-invariant} \emph{Chaplygin systems} (i.e., Chaplygin systems without the $G$-invariance assumption). Our biggest gain is simplicity, as admissible variations are much easier objects to work with than the equations of motions, the latter being derived from the former. Therefore the proof of our main result -- Theorem~\ref{thm:chap_syst} -- is straightforward. This result fully characterizes (in terms of the geometry of the (non-invariant) Chaplygin system):
\begin{enumerate}
\item [\eqref{part:nh_st}] those nonholonomic extremals that are simultaneously unconstrained extremals;
\item [\eqref{part:nh_vak}] those nonholonomic extremals that are simultaneously vakonomic extremals (i.e., provides an answer to \eqref{main_question_1});
\item [\eqref{part:vak_nh}] those vakonomic extremals (associated with a given Lagrange multiplier) that are simultaneously nonholonomic extremals (thus provides an answer to \eqref{main_question_2}).
\end{enumerate} 
The known results on Chaplygin systems from \cite{Crampin_Mestdag_2010,Favretti_1998,Fernandez_Bloch_2008} follow easily from Theorem~\ref{thm:chap_syst} as corollaries as pointed in Remark~\ref{rem:comparison} and Corollary~\ref{cor:chaplygin}. Let us mention that we do not just regain these results, but also substantially generalize them, as in our approach  any regularity conditions are superfluous and the role of the symmetry conditions becomes apparent. Actually, it turns out that the symmetry is not as important as the existence of the natural splitting of the configurations space into horizontal and vertical parts. 

In the remaining part of Section~\ref{sec:chaplygin}, we derive the precise form of vakonomic multipliers (Proposition~\ref{prop:chapl_vak}), discuss the relation between questions \eqref{main_question_1} and \eqref{main_question_2} (Lemma~\ref{lem:relation_b_c}) and study (non-invariant) Chaplygin systems subject to additional symmetry conditions in Corollaries~\ref{cor:inv_constr}--\ref{cor:chaplygin}.  
\medskip

Section~\ref{sec:lie_groups} presents an application of our general  methods from Section~\ref{sec:var_calc} to a particular class of systems on Lie groups with linear constraints defined by a left invariant distribution. Such systems, with an additional assumption of the symmetry of the Lagrangian, were considered for instance in \cite{Koiller_1992}. In this case, we prove Theorem~\ref{thm:Lie} which answers the same questions as Theorem~\ref{thm:chap_syst} for the considered class of systems (in fact both theorems are closely related as is explained in Remarks~\ref{rem:Chapl_as_Lie} and~\ref{rem:Chapl_as_lie_1}). In this case we can also derive the precise form of the vakonomic multiplier and, moreover, write explicitly the nonholonomic equations of motion (Lemma~\ref{lem:lie_eqn}). In Corollary \ref{cor:lie} we discuss a special case of a system with an additional symmetry.\medskip

In Section~\ref{sec:examples}, we study concrete examples of nonholonomic systems with linear constraints. These include the unicycle (Example~\ref{ex:unicycle}), the two-wheeled carriage (Example~\ref{ex:2wheels}), and the Heisenberg system (Example~\ref{ex:heis}, and its generalization in Example~\ref{ex:gen_heis}). All these situations position themselves in the common setting of Sections~\ref{sec:chaplygin} and \ref{sec:lie_groups} described in Remarks~\ref{rem:Chapl_as_Lie} and \ref{rem:Chapl_as_lie_1}. Therefore we use them to illustrate the results from both Section~\ref{sec:chaplygin} and \ref{sec:lie_groups}. For all the considered situations, which have been widely studied (for instance in \cite{Bloch_Crouch_1993,Cardin_Favretti_1996,Cortes_Leon_Inn_2003,Crampin_Mestdag_2010,Favretti_1998,Fernandez_Bloch_2008,Guo_Liu_Inn_2007}), 
our methods provide an elegant answer to question \eqref{main_question} (in most cases already known in the literature). Let us note that our Examples~\ref{ex:heis} and \ref{ex:gen_heis} contradict Proposition 3(5) in \cite{Fernandez_Bloch_2008}, which state that a system on a 3-dimensional manifold with a 2-dimensional non-integrable constraints distribution cannot answer question \eqref{main_question} positively. We explain that incorrect statement in Remark~\ref{rem:Bloch_Fernandez}. Of particular interest is Example~\ref{ex:2wheels} where, from purely geometric (Lie algebraic) and relatively simple consideration, we were able to re-obtain an interesting result from \cite{Crampin_Mestdag_2010}: the two-wheeled carriage with a shifted center of mass answers positively \eqref{main_question} if and only if the parameters of the system satisfy a certain algebraic condition. In \cite{Crampin_Mestdag_2010} this case is described by a vakonomic multiplier equal to the momentum shifted by a constant of motion. 

In addition to our main question \eqref{main_question}, in all considered examples we were also able to determine these nonholonomic trajectories which are simultaneously extremals of the unconstrained dynamics. We also derived the general form of the vakonomic multiplier.

\newpage
\section{Preliminaries}\label{sec:prel}

Throughout the paper we work in the $C^\infty$-smooth category. By $Q$ we will denote an $n$-dimensional smooth manifold, by $\tau_Q:\T Q\ra Q$ its tangent bundle, and by $\pi_Q:\T^\ast Q\ra Q$ the cotangent bundle. We will use the symbol $\X(Q)$ to denote the $C^\infty(Q)$-module of vector fields on $Q$. When working with local coordinates the summation convention of Einstein will always be assumed. {\change By $I=[t_0,t_1]\subset\R$ we will denote a fixed real interval.}

\paragraph{{\change The induced} local coordinates.} 
Let us introduce a local coordinate system  $(q^a)$, $a=1,2,\hdots, n$  on a manifold $Q$. Such a system induces a coordinate system $(q^a,\dot q^b)$, $a,b=1,\hdots,n$ on $\T Q$, i.e., $\dot q^b$'s are the coordinates of a vector in $\T Q$ with respect to the local frame $\{\pa_{q^a}\}$ or, equivalently, $\dot q^b:=\<\dd q^b,\cdot>:\T Q\ra\R$. {\change  An iteration of this construction leads to the induced coordinate system $(q^a,\dot q^b, q'^c, \dot q'^d)$ on the second tangent bundle $\T\T Q$. That is, $q'^c:=\<\dd q^c,\cdot>:\T\T Q\ra \R$ and $\dot q'^d:=\<\dd \dot q^d,\cdot>:\T\T Q\ra \R$. Note that now $q^c$ is treated as a function on $\T Q$ (by composing it with $\tau_Q$), contrary to $q^a$ being a function on $Q$.}

\paragraph{The canonical flip.}
It is well known (see e.g., \cite{GGU_2006}) that the second tangent bundle $\T\T Q$ admits an involutive map called \emph{the canonical flip}
$$\kappa_Q:\T\T Q\lra \T\T Q,$$ 
which is defined as
$$\kappa_Q:\Jet s\Jet t q(t,s)\longmapsto\Jet t\Jet sq(t,s)\ ,$$
where {\change a homotopy} $q(t,s)\subset Q$ is any representative of an element $\jjet s\jjet t q(t,s)\in\T\T Q$. {\change In the induced coordinates $(q^a,\dot q^b, q'^c, \dot q'^d)$ on $\T\T Q$, the flip interchanges coordinates $\dot q^a$ and $q'^a$ (which corresponds to differentiation with respect to parameters $t$ and $s$, respectively). That is
$$\kappa_Q:(q^a,\dot q^b, q'^c, \dot q'^d)\longmapsto(q^a,q'^b, \dot q^c, \dot q'^d)\ . $$  
A geometric idea behind $\kappa_Q$ is very simple. Namely, we change our point of view on the homotopy $q(t,s)$ in $Q$: instead of treating it as an $s$-parameter family of curves in $t$, we treat it as a $t$-parameter family of curves in $s$.   
}

\paragraph{Anholonomic frames.}
In our work we shall, however, use also another coordinate system $(q^a,v^b)$ on $\T Q$ associated with a given local frame (local basis) $\{e_a\}$ of $\T Q$. In other words, $v^b:=\<e^b,\cdot>:\T Q\ra\R$, where $\{e^b\}$ is the local coframe dual to $\{e_a\}$, i.e., $\<e^b,e_a>=\delt ba$. Such coordinates were considered  in the context of nonholonomic constraints by Crampin and Mestdag \cite{Crampin_Mestdag_2010}, who call $\{e_a\}$ an \emph{anholonomic frame}. To describe the passage between the two coordinate systems $(q^a,\dot q^b)$ and $(q^a,v^b)$, introduce a family of transition matrices $\A ab(q)$ relating the two local frames 
$$e_b=\A ab(q)\pa_{q^a}\ ,$$  
which implies
\begin{equation}\label{eqn:trans_coord}
\dot q^a=\A ab(q)v^b. 
\end{equation}
The above formula is useful when describing the tangent lift of a curve in $Q$. Let, namely, $q:I\ra Q$ be a smooth curve described locally by $q(t)\sim(q^a(t))$. Its \emph{tangent lift}, which we will denote by $\jet q(t)$, {\change or $\jet_tq(t)$ to emphasize the role of the variable $t$}, is a curve in $\T Q$ described locally by 
\begin{align*}
&\jet_tq(t)\sim\left(q^a(t),\dot q^b(t)\right) 
\intertext{in coordinates $(q^a,\dot q^b)$, or}
&\jet_tq(t)\sim\left(q^a(t),v^b(t)\right) 
\end{align*}
in coordinates $(q^a,v^b)$, where $\dot q^a(t)$ and $v^b(t)$ are related by \eqref{eqn:trans_coord}, i.e., $\dot q^a(t)=\A ab(q(t))v^b(t)$.\medskip 

{\change Coordinates $(q^a,v^b)$ on $\T Q$ induce coordinates $(q^a,v^b, q'^c, v'^d)$ on the second tangent bundle $\T\T Q$ by $q'^c:=\<\dd q^c,\cdot>$ and $v'^d=\<\dd v^d,\cdot>$. In these new coordinates $\kappa_Q$ takes a more general form
\begin{equation}\label{eqn:kappa_coord}
\kappa_Q:(q^a,v^b,q'^c,v'^d)\longmapsto (q^a,\und{v}^b, \A ce(q)v^e,v'^d+\br def(q)\und{v}^e v^f)\ .
\end{equation}
Here, the coordinates $\und{v}^a$  on $TQ$ are related to $q'^a$ via \eqref{eqn:trans_coord}, i.e., $q'^a=\A ab(q)\und{v}^b$, while $\br abc(q)$ satisfy} \newnew{Chyba logiczniej jest zamienić tutaj kolejność: najpierw mówimy, że możemy wyliczyć odpowiedną formułę wpółrzędnościową, a potem cudownie okazuje się, że współczynniki tak naprawdę pochodzą od nawiasu Liego}
\begin{equation}
\label{eqn:def_c_abc}
\A sa(q)\br abc(q)=\A db(q)\frac{\pa}{\pa q^d}\A sc(q)-\A dc(q)\frac{\pa}{\pa q^d}\A sb(q)\ .
\end{equation} 
{\change Actually, $\br abc(q)$} are structure functions expressing the Lie bracket of the vector fields of the frame $\{e_a\}$, that is, 
$$[e_b,e_c]=\br abc(q) e_a.$$
This can be easily deduced form formula \eqref{eqn:trans_coord} and the fact that $[\pa_{q^b},\pa_{q^c}]=0$.

The presence of the coefficients of the Lie bracket in formula \eqref{eqn:kappa_coord} suggests a possible relation between the flip $\kappa_Q$ and the Lie bracket of vector fields on $Q$. This is indeed the case, as explained in Par. 6.13 of \cite{Kolar_Michor_Slovak_1993}. As a consequence, a smooth distribution $D\subset \T Q$ is integrable if and only if $\T D$ is $\kappa_Q$-invariant {\change (cf. Proposition~\ref{prop:int_distr} below).}


\newpage
\section{An abstract approach to constrained Lagrangian dynamics}\label{sec:var_calc}

In this section, we look at the (constrained) Lagrangian dynamics from a slightly more formal and more abstract point of view than it is usually done in the literature. Such an approach will allow us to treat different constrained variational problems in a unified way. Similar ideas have already been presented (see for example \cite{GG_2008, Gracia_Martin_Munos_2003, Tul_2004}).

\paragraph{{\change Admissible paths and admissible variations}.} 
The standard variational problem on $\T Q$ is constituted by a function $L:\T Q\ra\R$ called a \emph{Lagrangian}. Given any smooth path $\gamma:I=[t_0,t_1]\ra \T Q$ we can define the \emph{action} along $\gamma$ by the formula
$$S_L(\gamma):=\int_IL(\gamma(t))\dd t.$$ 
In the standard problem we consider  $S_L$ only for {\change \emph{admissible paths}, i.e.} $\gamma:I\ra \T Q$ which are tangent prolongations of curves in $Q$. That is, $\gamma(t)=\jet_t q(t)=(q(t),\dot q(t))$, where $q(t)=\tau_Q(\gamma(t))$ is the base projection of $\gamma(t)$.

A \emph{variation} along a path $\gamma$ is simply a vector field along $\gamma$, i.e, a curve $\del \gamma:I\ra\T\T Q$ which projects onto $\gamma$ under $\tau_{\T Q}$. Among all variations along a fixed $\gamma$ we can distinguish the class of variations with \emph{vanishing end-points}, i.e., $\del \gamma$ such that $\T\tau_Q(\del\gamma(t))\in\T Q$ vanishes at $t_0$ and $t_1$.

In the standard variational problem we consider \emph{variations generated by homotopies} {\change(see Figure~\ref{fig:homotopy})}. Let namely $q(t,s)\in Q$ be a one-parameter family of base paths. This homotopy defines a natural variation 
$$\del\gamma(t)=\jet_s\big|_{s=0}\jet_t q(t,s)$$
along the path $\gamma(t)=\jet_tq(t,0)=(q(t,0),\dot q(t,0))$, where the dot  stands for the derivative with respect to~$t$. Geometrically such $\del\gamma$ is \emph{generated} by a curve $\xi(t):=\jet_s\big|_{s=0}q(t,s)\in\T_{q(t,0)}Q$ with the help of the canonical flip $\kappa_Q:\T\T Q\ra\T\T Q$:
$$\del\gamma(t)=\jet_s\big|_{s=0}\jet_t q(t,s)=\kappa_Q\left(\jet_t\jet_s\big|_{s=0}q(t,s)\right)=\kappa_Q(\jet_t\xi(t)).$$
To emphasize the role of the \emph{generator} $\xi(t)$ (called also sometimes an \emph{infinitesimal variation} {\change or a \emph{virtual displacement}}) we will denote 
$$\del_\xi\gamma:=\del\gamma\ .$$
{\change Variations of this form will be called \emph{admissible}.} \newnew{Pojęcie admissible variation było używane, ale nie było zdefiniowane}

\begin{figure}[h]
\begin{center}
\includegraphics[width=0.45\textwidth]{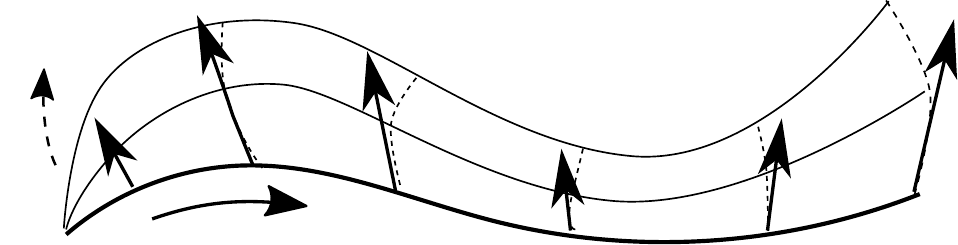}
\put(-150,0){$q(t,0)$}
\put(-115,35){$q(t,s)$}
\put(-3,30){$\xi(t)$}
\put(-187,3){$t$}
\put(-235,23){$s$}
\end{center}
\caption{A homotopy-generated variation and its generator~$\xi$. Note that the homotopy with a fixed end-point(s) corresponds to a generator vanishing at that end-point(s).}
\label{fig:homotopy}
\end{figure}
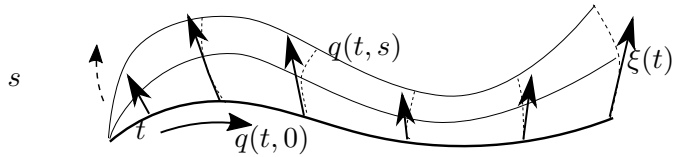

Assume now that, in local coordinates $(q^a,v^b)$ on $\T Q$, the path $\gamma(t)=\jet_tq(t)$ corresponds to a curve $(q^a(t),v^b(t))$ and the generator $\xi(t)$ to $(q^a(t),w^b(t))$.  Due to formula \eqref{eqn:kappa_coord}, the variation of $\gamma(t)$ generated by $\xi(t)$ reads locally, in coordinates  $(q^a,v^b, q'^c, v'^d)$,  as
\begin{equation}\label{eqn:var_coord}
\del_\xi\gamma(t)\sim\left(q^a(t),v^b(t),\A ce(q(t))w^e(t),\dot w^d(t)+\br dab(q(t)) v^a(t)w^b(t)\right).
\end{equation}
We will need the following three facts about admissible variations which follow directly from the above coordinate description. 
\begin{prop}\label{prop:prop_variation}
\begin{enumerate}[(i)] The admissible variation $\del_\xi\gamma$ has the following properties
\item \label{point:prob_var_1} it is linear with respect to $\xi$, i.e., 
\begin{equation}\label{eqn:var_linear}
\del_{\xi+\xi'}\gamma=\del_\xi\gamma+\del_{\xi'}\gamma,
\end{equation}
where the addition is performed with respect to the tangent vector bundle structure on $\tau_{\T Q}:\T\T Q\ra\T Q$. 
\item \label{point:prob_var_2} it projects to $\gamma$ under $\tau_{\T Q}$. 
\item \label{point:prob_var_3} it projects to $\xi$ under $\T\tau_Q$.
\end{enumerate}
\end{prop}

Note also that for any $q(t)\in Q$ and any $\xi(t)\in \T_{q(t)}Q$ there exists a homotopy $q(t,s)$ such that $q(t,0)=q(t)$ which generates  the variation $\del_\xi\gamma$ along $\gamma(t)=\jet_tq(t)=(q(t),\dot q(t))$. Observe that variations with vanishing end-points correspond to $\xi(t)$'s satisfying $\xi(t_0)=\xi(t_1)=0$.

\paragraph{{\change Variational principles.} }
The {\change \emph{standard variational problem}} is to search of all {\change admissible paths $\gamma:I\ra \T Q$} such that, for every {\change admissible} variation $\del\gamma$ with vanishing end-points (as considered above), the associated \emph{variation of the action} $S_L$  at $\gamma$ along $\del\gamma$
$$\<\dd S_L(\gamma),\del\gamma>:=\int_I\<\dd L(\gamma(t)),\del\gamma(t)>\dd t$$
vanishes. 

\begin{remark}
The usage of the symbol $\dd S_L(\gamma)$ can be made rigorous in the framework of analysis on Banach manifolds. The action $S_L$ can be understood as a function on the manifold of paths $\gamma$ of a certain class, whereas $\del\gamma$'s are elements of the tangent space to that manifold. The interested reader may consult \cite{Martinez_2008}.
\end{remark}

Motivated by the standard situation we propose the following general definition in the spirit of \cite{Tul_2004}.
\begin{definition}\label{def:var_princ}
A \emph{variational principle} on $\T Q$ is constituted by a triple 
$$\PP:=(L,\TT,\WW)$$ 
consisting of a Lagrangian function $L:\T Q\ra\R$, a set of \emph{admissible paths} $\TT\subset C^\infty(I,\T Q)$, and a set of \emph{admissible variations} $\WW\subset C^\infty(I,\T\T Q)$ along admissible paths. For a given admissible path $\gamma\in\TT$, we will use symbols $\WW_\gamma$ for admissible variations along $\gamma$, and $\WW_\gamma^0$ for admissible variation along $\gamma$ with vanishing end-points.

An admissible path $\gamma \in\TT$ is called an \emph{extremal} (or a \emph{trajectory}) of the variational principle $\PP$ if and only if 
$$\<\dd S_L(\gamma),\del\gamma>=0\quad\text{for every $\del\gamma\in\WW_\gamma^0$},$$
i.e., $\gamma$ is a critical point of the action $S_L$ (see \cite{Martinez_2008}) relative to admissible variations {\change of $\PP$} with vanishing end-points. The set of all extremals of $\PP$ will be denoted by $\Gamma_\PP\subset\TT$, {\change and called the \emph{dynamics of $\PP$.}} Notice that searching for a critical point of $S_L$ relative to paths $\gamma$ such that $\del\gamma\in \WW^0_\gamma$ need not correspond, in general, to minimization or maximization of $S_L$. 
\newnew{Zdefiniowałem pojęcie dynamiki}
\end{definition} 

\begin{remark}\label{rem:vanishing_end_points}
Despite the fact that in the above definition of an extremal we used  admissible variations with vanishing end-points only, the role of the end-points should not be underestimated. In fact, the full variational principle should describe the reaction of the system to an \emph{arbitrary} admissible variation, i.e., it should contain not only the information about the extremal, but also about the boundary terms which describe the initial and final momenta of the system (cf. Sec. 15 in \cite{Tul_2004}). The need of including the non-vanishing end-points becomes apparent also in some natural situations in  variational calculus and control theory, where more general boundary conditions (like transversality) are needed. 
\end{remark}

\begin{remark} Clearly, the {\change solutions} of the standard variational problem on $\T Q$ (i.e., solutions of the associated Euler-Lagrange equations {\change for a Lagrangian $L:\T Q\ra \R$}) are extremals of the following {\change\emph{standard variational principle}} $\PP^{st}_{\T Q}=(L,\TT^{st}_{\T Q},\WW^{st}_{\T Q})$ where
\begin{align*}
&\TT^{st}_{\T Q}=\{\gamma(t)\ |\ \gamma(t)=\jet_tq(t)=(q(t),\dot q(t))\text{ and } q:I\ra Q\}\quad\text{and}\\
&(\WW^{st}_{\T Q})_{\gamma(t)}=\{\del_{\xi(t)}\gamma(t)=\kappa_Q(\jet_t\xi(t))\ |\ \xi(t)\in\T_{q(t)}Q \}.
\end{align*}
We shall refer to the elements of $\WW^{st}_{\T Q}$ as to \emph{standard admissible variations}. 
\end{remark}


\paragraph{Constraints.} 
The above Definition~\ref{def:var_princ}  turns out to be particularly useful in the context of constrains. We say that $\wh{\PP}=(L,\wh\TT,\wh\WW)$ is a \emph{restricted variational principle}  of $\PP=(L,\TT,\WW)$ if it is obtained from $\PP$ by shrinking the set of admissible trajectories and/or admissible variations, i.e., $\wh\TT\subseteq\TT$ and/or $\wh\WW\subseteq\WW$. One should think that the principle $\PP$ describes an unconstrained system and $\wh\PP$ is the same systems with imputed constraints. Usually these restrictions are somehow related to additional geometric structures on the bundle $\T Q$. 
Two important examples of such a situation are vakonomic and nonholonomic variational principles associated with a submanifold $C\subset \T Q$, being two restrictions of the standard variational principle $\PP^{st}_{\T Q}$. Below we shall show that the extremals of these restricted variational principles constitute the vakonomically and nonholonomically constrained Lagrangian dynamics in the standard sense. \newnew{Wyróżniam to jako definicję - jest kluczowa dla naszych rozważań}

\begin{definition}\label{def:vak_var_princip}
{\change Let $C\subset \T Q$ be a submanifold.}
We define the \emph{vakonomic variational principle} $\PP^{vak}_C=(L,\TT^{vak}_C,\WW^{vak}_C)$ associated with $C$ {\change as a restriction of the standard variational principle $\PP^{st}_{\T Q}=(L,\TT^{st}_{\T Q},\WW^{st}_{\T Q})$, where we consider only those  admissible paths that belong to $C$ and those admissible variations that are tangent to $C$.} That is 
\begin{align*}
&\TT_{C}^{vak}=\TT_C:=\{\gamma\in\TT^{st}_{\T Q}\ |\ \gamma(t)\in C \text{ for every $t\in I$}\}\subset\TT^{st}_{\T Q}
\intertext{and}
&\WW_{C}^{vak}:=\{\del_\xi\gamma\in\WW^{st}_{\T Q}\ |\ \del_{\xi(t)}\gamma(t)\in\T_{\gamma(t)}C \text{ for every $t\in I$}\}\subset \WW^{st}_{\T Q}.
\end{align*}  
\end{definition}

Observe that {\change elements of $\WW^{vak}_C$  are precisely} vakonomic variations present in the literature (see e.g., \cite{Arnold_Koz_Nies_2010, Cardin_Favretti_1996, Leon_2012}). Clearly, any homotopy {\change $q(t,s)\in Q$ such that $\gamma(t,s):=\jet_t q(t,s)\subset C$} produces a variation in $\WW^{vak}_C$. Conversely, every variation in $\WW^{vak}_C$ can be obtained from a homotopy {\change $q(t,s)\in Q$ such that $\gamma(t,s):=\jet_t q(t,s)$} lies in $C$ up to $o(s)$-terms.\footnote{Such relaxation of the condition $\gamma(t,s)\subset C$, allows to exclude the problems of singular trajectories and abnormal extremals (see Ssec. 1.4 in \cite{Arnold_Koz_Nies_2010}).} In light of this observation it is clear that the extremals of {\change the vakonomic variational principle} $\PP^{vak}_C$ {\change are precisely} the extremal points of $S_L$ on the set of admissible paths $\TT^{vak}_C$, i.e., {\change they are trajectories of the vakonomically constrained system on $Q$ (constituted by $L$ and $C$) in the usual sense present in the literature \cite{Arnold_Koz_Nies_2010}. For this reason we will simply speak about \emph{vakonomic dynamics} meaning the dynamics of the vakonomic variational principle (i.e. the set of all exteemals of $\PP^{vak}_C$). We will also use an  abbreviated symbol $\Gamma^{vak}_C$ (instead of $\Gamma_{\PP^{vak}_C}$) to denote these dynamics. }

Observe that although the set of \emph{vakonomic admissible variations} $\WW_{C}^{vak}$ is characterized by the simple condition $\WW_{C}^{vak}=\WW^{st}_{\T Q}\cap\C^\infty(I,\T C)$, in general, it is difficult to specify the generators $\xi(t)$ for which a given admissible variation $\del_{\xi(t)}\gamma(t)$ of the standard variational principle belongs to $\T C$. Note also that since the vakonomic variations are tangent to $C$, the vakonomic dynamics are determined by the restriction of $L$ to $C$.  

\newnew{Myślę, że to jest dobry moment żeby wspomnieć o mnożnikach Lagrange'a dla przypadku wakonomicznego. Formułujemy dla tego przypadku nasze wyniki a nie został on zbyt dobrze wyjaśniony.} 
{\change
\begin{remark}\label{rem:multipliers}
In practice, finding extremals of the vakonomic variational principle can be reduced to finding extremals of the standard variational principle but with a modified Lagrangian. Indeed, observe that since $\del_\xi\gamma\in \T C$, we can add to the Lagrangian $L$ any function $\phi(t,q,v)$ on $\R\times\T Q$ vanishing at $(q,v)\in C\subset\T Q$ without changing the value of the variation $\<\dd S_L(\gamma), \del_\xi\gamma>$. Thus if $\gamma\in\TT_C$ is an extremal of the standard variational principle $\PP^{st}_{\T Q}$ with the new Lagrangian $L+\phi$, then $\gamma$ is also an extremal of the vakonomic variational principle $\PP^{vak}_C$ with the initial Lagrangian $L$. This reasoning gives sufficient (and also necessary -- see e.g., Thm. 4.1 in \cite{Cardin_Favretti_1996} or Lemma 3 in \cite{Gracia_Martin_Munos_2003}) conditions for extremals of $\PP^{vak}_C$. In practice, we can choose the new Lagrangian in the form
$$
\wt{L}(q,v,t):=L(q,v)+\mu_\alpha(t)\Phi^\alpha(q,v)\ ,
$$
where $C$ is locally described by equations $\Phi^\alpha(q,v)=0$, for $\alpha=1,\hdots, k$, and $\mu_\alpha(t)$ are arbitrary functions, known usually as \emph{multipliers}. 
\end{remark}
}

\medskip

With the same submanifold $C\subset\T Q$ one can associate a different construction of a nonholonomic variational principle. \newnew{formalizuję tę definicję}
\begin{definition}\label{def:nh_var_princip}
{\change Let $C\subset \T Q$ be a submanifold. A }
\emph{nonholonomic variational principle} $\PP^{nh}_C=(L,\TT^{nh}_C,\WW^{nh}_C)$  {\change associated with $C$ is a restriction of the standard variational principle
$\PP^{st}_{\T Q}=(L,\TT^{st}_{\T Q},\WW^{st}_{\T Q})$, where we consider only these  admissible paths that belong to $C$ }
$$\TT_{C}^{nh}:=\TT_{C},$$
{\change and the set $\WW^{nh}_C$ is defined by means of the \emph{Chetaev's principle}. More precisely, $\WW^{nh}_C$ consists of these admissible variations $\del_\xi\gamma\in\WW_{\T Q}^{st}$ that are generated by an infinitesimal variation $\xi(t)$ whose} vertical lift $\V_{\gamma(t)}\xi(t):=\jet_{s=0}(\gamma(t)+s\xi(t))$ takes values in $\T C$:
$$\WW_{C}^{nh}=\{\del_\xi\gamma\in\WW_{\T M}^{st}\ |\ \V_{\gamma(t)}\xi(t)\in\T_{\gamma(t)} C \text{ for every $t\in[t_0,t_1]$}\}.$$
Notice that $\WW^{nh}_C\subset \T C\cap \V\T Q$, where $\V\T Q$ stands for the vertical distribution on $\T Q$ defined as the kernel of $\T\tau_Q:\T\T Q\ra\T Q$. We shall refer to the elements of $\WW^{nh}_{C}$ as to \emph{nonholonomic admissible variations}.
\end{definition}

 By the very definition of Chetaev's principle it is clear that {\change extremals of the nonholonomic variational principle $\PP^{nh}_C$ are precisely trajectories of the nonholonomically constrained system on $Q$ (constituted by $L$ and $C$) in the standard sense present in the literature \cite{Arnold_Koz_Nies_2010}. For this reason we will simply speak about \emph{nonholonomic dynamics} meaning the dynamics of the nonholonomic variational principle (i.e. the set of all extemals of $\PP^{nh}_C$). We will also use the abbreviated symbol $\Gamma^{nh}_C$ (instead of $\Gamma_{\PP^{nh}_C}$) to denote these dynamics. }

It is known that the extremals of $\PP^{nh}_C$ do not correspond to minimization (maximization) of $S_L$. In fact, they are not ''the shortest'' but ''the straightest'' paths as noticed by Hertz (see \cite{Leon_2012} and the references therein).

In the special case when the constraints are linear (resp. affine), meaning that $C=D$, where $D$ is a linear distribution (resp. $C=X+D$, where $X$ is a vector field, and $D$ a linear distribution), Chetaev's principle becomes the well-known \emph{d'Alembert's principle}: we consider admissible variations $\del_\xi\gamma\in\WW_{\T Q}^{st}$ that are generated by infinitesimal variations $\xi(t)$ with values in $D$ (in both, linear and affine cases):
$$\WW_{D}^{nh}=\{\del_\xi\gamma\in\WW_{\T Q}^{st}:\xi(t)\in D_{q(t)} \text{ for every $t\in I$}\}.$$
In this case we have an explicit information  about the infinitesimal variations, i.e., $\xi(t)\in D_{q(t)}$. Note, however, that the variations $\del_\xi\gamma$ will, in general, not be tangent to $C=D$ (resp., to $X+D$). For this reason the knowledge of $L|_D$ is not sufficient to  study nonholonomically constrained dynamics. For a deeper discussion of the constrained dynamics in a more general setting of algebroids consult \cite{GG_2008,GLMM_2009}. 

{\change Tu summarize the above considerations on restricted variational principles (cf. \cite{Gracia_Martin_Munos_2003}):
\begin{itemize}
	\item extremals $\Gamma^{vak}_C$ of the vakonomic variational principle $\PP^{vak}_C$ are the trajectories of the vakonomically constrained system on $Q$ (associated with $C$) in the usual sense, and 
	\item extremals $\Gamma^{nh}_C$ of the nonholonomic variational principle $\PP^{vak}_C$ are the trajectories of the nonholonomically constrained system on $Q$ (associated with $C$) in the usual sense. 
\end{itemize}
}


\paragraph{A comparison of variational principles.}
Looking at admissible variations rather than equations of motion will allow us to compare extremals of different variational principles in a simple manner. Recall that $\Gamma_{\PP}$ denotes the set of extremals of a given variational principle $\PP$.  

\begin{prop}\label{prop:constraints} Assume that $\wh\PP =(L,\wh\TT ,\wh\WW)$ is a restricted variational principle of $\PP=(L,\TT,\WW)$, that is, $\wh\TT \subseteq\TT$ and $\wh\WW\subseteq\WW$ (i.e., for any $\gamma\in\wh\TT$ we have $\wh\WW_\gamma\subseteq\WW_\gamma$). Then
$$\Gamma_{\PP}\cap\wh\TT \subseteq\Gamma_{\wh\PP}.$$
\end{prop}
\begin{proof}Indeed, if $\gamma\in\wh\TT $ satisfies $\<\dd S_L(\gamma),\del\gamma>=0$ for every $\del\gamma\in \WW_\gamma^0$, then also $\<\dd S_L(\gamma),\wh\del \gamma>=0$ for any $\wh\del \gamma\in\wh\WW_\gamma ^0\subseteq\WW_\gamma^0$. 
\end{proof}

\begin{remark}\label{rem:not_trivial}
The above proposition looks trivial it allows, however, for an immediate derivation of some classical results such as:
\begin{itemize}
\item Proposition 6.2 in \cite{Cortes_Leon_Inn_2003},  which states that {\change every solution of the standard variational problem that respects the constraints $C$ is simultaneously} a trajectory of a nonholonomically, as well as, vakonomically constrained system associated with $C$. This is obvious in light of Proposition~\ref{prop:constraints} as vakonomic and nonholonomic variational principles are restrictions of the standard variational principle.
\item Remark following Theorem 2  in \cite{Crampin_Mestdag_2010}, which states that vakonomic trajectories with trivial multipliers $\mu_a(t)=0$ (cf. Remark~\ref{rem:multipliers}) are also nonholonomic trajectories. This again is clear as any vakonomic extremal with trivial multipliers is, in fact, an extremal of the standard variational problem and we can repeat the above reasoning.
\item Theorem 3.2 (i) of \cite{Favretti_1998}, which states that for every sub-Riemannian geodesic problem with a totally geodesic constraints distribution (i.e., such that every unconstrained geodesic tangent to the constraints at a point remains tangent at all its points) the nonholonomic geodesics are precisely the unconstrained geodesic respecting the constraints. This fact follows again from Proposition~\ref{prop:constraints} implying that the unconstrained geodesics respecting the constraints are also the nonholonomic ones. Moreover, by the assumptions and by the uniqueness of (nonholonomic) geodesics with a given initial velocity, we get the equality of these two sets.  Theorem 3.2 (ii) of \cite{Favretti_1998}, stating that in this case every nonholonomic geodesic is also a vakonomic one is again clear in the light of Proposition~\ref{prop:constraints}. From this simple reasoning we see that the additional assumption present in Theorem 3.2 (that the distribution perpendicular to the constraints is integrable) is superfluous.
\end{itemize}  
\end{remark}

More generally, we can compare two variational principles $\PP=(L,\TT,\WW)$ and $\wh\PP=(L,\wh\TT,\wh\WW)$ defined on the same manifold $\T Q$ and with the same Lagrangian $L$, even if $\WW$, $\TT$ and $\wh\WW$, $\wh\TT $ are not so directly related, provided that we have some information about infinitesimal symmetries of $L$. 

\begin{prop}\label{prop:two_var_prob}
Consider an extremal $\gamma\in\Gamma_{\PP}\cap\wh\TT $. Assume that for every variation $\wh\del \gamma\in\wh{\WW}_\gamma^0$ there exists a variation $\del\gamma\in\WW^0_\gamma$ such that $\wh\del\gamma(t)-\del\gamma(t)\in(\ker\dd L)_{\gamma(t)}$ for every $t\in[t_0,t_1]$. Then
$$\gamma\in\Gamma_{\wh\PP }.$$
\end{prop}
\begin{proof}
Take an extremal $\gamma\in\Gamma_{\PP}\cap\wh\TT $. We would like to show that $\<\dd S_L(\gamma),\wh\del \gamma>=0$ for every $\wh\del \gamma\in\wh\WW ^0_\gamma$. Now for any $\del\gamma$ satisfying the assumptions, we have
$$\<\dd S_L(\gamma),\wh\del \gamma>-\Big\langle\dd S_L(\gamma),\del\gamma\Big\rangle=\int_I\<\dd L(\gamma(t)),\wh\del\gamma(t)-\del\gamma(t)>\dd t=0.$$
Hence $\langle\dd S_L(\gamma),\wh\del \gamma\rangle=\<\dd S_L(\gamma),\del\gamma>$ which equals 0 as $\gamma$ is an extremal of $\PP$ and $\del \gamma\in\WW_\gamma^0$.
\end{proof}
The condition $\wh\del\gamma-\del\gamma\in (\ker \dd L)_\gamma$ from the above proposition may be understood as a symmetry condition. Indeed, it means that the set of admissible variations $\wh{\WW}_\gamma^0$ is contained in $\WW_\gamma^0$ up to infinitesimal symmetries of $L$. 

\paragraph{\change Example -- holonomic constraints.} 
As a simple concrete illustration of our approach to the question of comparing variational principles we can easily prove the following well-known fact (compare e.g., Prop. 2.8 in \cite{Lewis_Murray_1995}). 

\begin{prop}\label{prop:holonomic}. Let $D\subset\T Q$ be a smooth distribution on a manifold $Q$. Then vakonomic and nonholonomic variational principles associated to $D$ (for the same Lagrangian) coincide, that is,  
$$\PP^{vak}_D=\PP^{nh}_D$$
if and only if $D$ is integrable. 
\end{prop}

Constraints discussed in the above proposition are known as \emph{holonomic constraints}. Notice that integrability of $D$ is a necessary and sufficient condition for the principles $\PP^{vak}_D$ and $\PP^{nh}_D$ to coincide and thus it implies that the sets of extremals $\Gamma_D^{vak}$ and $\Gamma_D^{nh}$ coincide as well, but it is not necessary for the latter. For example, if $L$ is constant, then $\Gamma_D^{vak}=\Gamma_D^{nh}=\TT_D$ (the set of all admissible paths), independently of $D$.

{\change To prove Proposition~\ref{prop:holonomic} we will use the following fact that relates the canonical flip $\kappa_Q$ with integrability of distributions. 

}

\begin{prop}\label{prop:int_distr}
Let $D\subset \T Q$ be a smooth distribution on $Q$. Let $\T D$ denote the tangent bundle of the distribution $D$ considered as a submanifold of $\T Q$. Then $D$ is integrable if and only if $\kappa_Q$ maps $\T D$ into $\T D$. 
\end{prop}

\begin{proof} Consider any two $D$-valued vector fields $X,Y:Q\ra D$, and chose a point $p\in Q$. The vectors $A=\T X(Y(p))$ and $B=\T Y(X(p))$ clearly belong to $\T D$. Note that vectors $A$ and $\kappa_Q(B)$ project to the same vector $X(p)\in\T_pQ$ via $\tau_{T Q}$ and to the same vector $Y(p)\in\T_p Q$ via $\T\tau_Q$, hence their difference $A-\kappa_Q(B)$ is vertical. In fact,  
$$
A-\kappa_Q(B)=\V_{X(p)}[X,Y](p):=\jet_{s=0}\big[X(p)+s\cdot[X,Y](p)\big]\ .
$$
{\change The above equality can be easily checked by a direct calculation using \eqref{eqn:kappa_coord}.}  A proof can be also found in paragraph 6.13 in \cite{Kolar_Michor_Slovak_1993}.\newnew{Zdecydowałem się nie wyróżniać powyższej równości jako osobnego wyniku, ale raczej ukryć ją wewnątrz innego wyniku, który jest znacznie prostszy do zrozumienia.}

It follows that $\kappa_Q(B)$ belongs to $\T D$ if and only if $\V_{X(p)}[X,Y](p)$ belongs to $\T_{X(p)} D$. The latter is equivalent to $[X,Y](p)\in D_p$ (since the vertical part of $\T D$ {\change can be canonically identified with} $D$).
\end{proof}

Now we are ready to prove Proposition~\ref{prop:holonomic}.

\begin{proof}[Proof of Proposition~\ref{prop:holonomic}.]
Assume that $D$ is integrable. It is enough to check that $\WW^{vak}_D=\WW^{nh}_D$. For a given admissible path $\gamma(t)$ take a generator of a nonholonomic variation $\xi(t)\in D$. Now since $\jet_t\xi(t)\in \T D$, {\change from Proposition~\ref{prop:int_distr} it follows that} $\del_\xi\gamma(t)=\kappa_Q(\jet_t\xi(t))\in\T D$, and thus $\del_\xi\gamma$ is a vakonomic variation. Conversely, given a vakonomic variation $\del_\xi\gamma(t)=\kappa_Q(\jet_t\xi(t))\in\T D$, we have $\jet_t\xi(t)=\kappa_Q(\del_\xi\gamma(t))\in \T D$ due to the fact that  $\kappa_Q$ is an involution. Hence $\xi(t)\in D$ is a generator of a nonholonomic variation. 

If $D$ is not integrable then $\kappa_Q(V)\notin \T D$ for some $V\in \T_{\gamma_0} D$. Now choose a point $t'\in(t_0,t_1)$ and consider an admissible path $\gamma\in \TT_D$ and a generator $\xi(t)\in D$ of a nonholonomic variation $\del_\xi\gamma$ such that $\gamma(t')=\gamma_0$ and $\jet_{t=t'}\xi(t)=V$. Clearly $\del_\xi\gamma(t')=\kappa_Q(\jet_{t=t'}\xi(t))=\kappa_Q(V)\notin \T D$, and hence the nonholonomic variation $\del_\xi\gamma$ cannot belong to $\WW^{vak}_D$.  
\end{proof}

\newpage
\section{Non-invariant Chaplygin systems} \label{sec:chaplygin}

In this section we shall apply Proposition~\ref{prop:two_var_prob} to solve the comparison problems \eqref{main_question_1} and \eqref{main_question_2} for a particular class of systems with linear constraints, namely for (non-invariant) Chaplygin systems. In particular, we will be able to recover (and generalize) some results from \cite{Crampin_Mestdag_2010,Favretti_1998,Fernandez_Bloch_2008}. To demonstrate the usefulness of our approach we shall omit the usual assumptions of the $G$-invariance of both: the constraints distribution and the Lagrangian.

\subsection{\change The geometry of Chaplygin systems}
\newnew{W tym długim rozdziale wyróżniłem dwa podrozdziały - przygotowawczy i wynikowy}
\paragraph{Chaplygin systems.}
Consider a right \emph{principal $G$-bundle} $\pi:Q\ra M=Q/G$. By a \emph{vertical distribution} on $Q$ we shall understand the distribution $\V Q:=\ker\pi_\ast\subset\T Q$ consisting of all vectors tangent to the fibres of $\pi$. By $R_g(q)$ or simply $q\cdot g$ we shall denote the action of an element $g\in G$ on a point $p\in Q$. Note that the induced action $(R_g)_\ast$ preserves $\V Q$, i.e., $(R_g)_\ast \V_q Q=\V_{q\cdot g}Q$. 

\begin{definition}\label{def:hor_dist}
A \emph{horizontal distribution} on $Q$ is any smooth distribution $\HH Q\subset \T Q$ such that at each $q\in Q$ we have $\T_qQ=\HH_qQ\oplus\V_qQ$. (Note that we do not assume that $\HH Q$ is $G$-invariant, i.e., that it is a horizontal distribution of a principal $G$-connection.) A curve $q(t)\in Q$ is called \emph{horizontal} if its tangent lift $\jet_t q(t)$ belongs to $\HH Q$. 

{\change Clearly, $\HH Q$ is a horizontal bundle of an \emph{Ehresmann connection} on $Q$. }
\end{definition}

\begin{definition}
\label{def:ch_syst} By a (\emph{non-invariant}) \emph{Chaplygin system} we shall understand a principal $G$-bundle $\pi:Q\ra M$ equipped with a horizontal distribution $\HH Q\subset \T Q$ and a smooth Lagrangian function $L:\T Q\ra\R$.
\end{definition}

\newnew{Rysunek poglądowy.}

\begin{figure}[h]
\begin{center}
\includegraphics[width=0.45\textwidth]{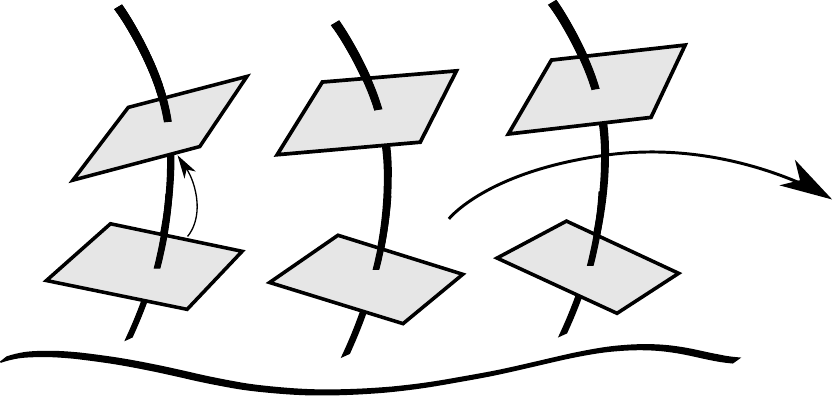}
\put(-100,15){$\HH Q$}
\put(2,47){$\R$}
\put(-50,51){$L$}
\put(-195,-2){$M$}
\put(-195,90){$Q$}
\put(-160,50){$G$}

\end{center}
\caption{A (non-invariant) Chaplygin system is a principal $G$-bundle $\pi:Q\ra M$, equipped with a horizontal distribution $\HH Q$  and a Lagrangian $L:\T Q\ra\R$, which not necessarily need to be $G$-invariant.}
\label{fig:chaplygin}
\end{figure}
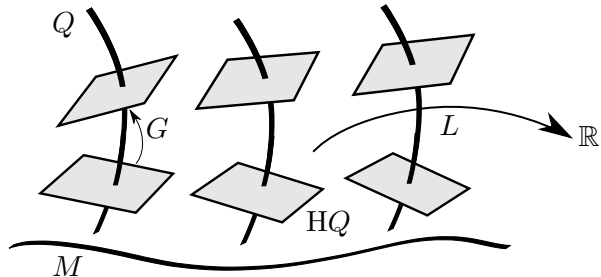

Usually in the literature the $G$-invariance of the Lagrangian $L$ and of the horizontal distribution $\HH Q$ is assumed. Such systems are called \emph{Chaplygin systems} \cite{Crampin_Mestdag_2010,Fernandez_Bloch_2008}, which term was coined by Koiller \cite{Koiller_1992}. Sometimes Chaplygin systems are described as \emph{Abelian} or \emph{non-Abelian}, depending on the commutativity of the structural Lie group $G$. Cantrijn et al. \cite{Cantrijn_Cortez_Inn_2002} use the adjective \emph{generalized} Chaplygin system in the same sense as other authors \cite{Crampin_Mestdag_2010,Fernandez_Bloch_2008, Koiller_1992} use the word \emph{non-Abelian} (to emphasize that the Lie group $G$ is \emph{general}). Our Definition~\ref{def:ch_syst} describes  a more general situation with no invariance conditions assumed.  To distinguish it from the standard setting we added the adjective \emph{non-invariant}. Clearly, the standard Chaplygin system are a special case of the non-invariant Chaplygin system with additional symmetry assumptions. Thus all our considerations about non-invariant systems will hold also in the standard case.

At this point it is worthy to remark about the side convention. We speak about systems with the right action of the structural group following the classical textbook \cite{Kobayashi_Nomizu_1963}. However, all our results remain valid also for systems with the left group action, provided that we carefully  substitute the right action with the left action, change $\ad_{g^{-1}}$ to $\ad_{g}$, etc.

With a given (non-invariant) Chaplygin system one can naturally associate nonholonomically and vakonomically constrained dynamics, taking $\HH Q$ to be the constraints distribution $C$. Note that the \emph{admissible paths} $\TT_{\HH Q}$ in the corresponding variational principles are precisely the tangent lifts of the horizontal curves in $Q$ (we shall therefore refer to the elements of $\TT_{\HH Q}$ as to \emph{horizontal admissible paths}).

\paragraph{\change A brief overview.} \newnew{Dodałem ten fragment, żeby wyjaśnić co się będzie działo.}
{\change Throughout the remaining part of this section we shall be working with a given (non-invariant) Chaplygin system constituted by a Lagrangian $L:\T Q\ra \R$ and a horizontal distribution $\HH Q\subset \T Q$ on a right principal $G$-bundle $\pi:Q\ra M=Q/G$. Our ultimate goal is to solve the comparison problems \eqref{main_question_1} and \eqref{main_question_2} for such a system (taking $\HH Q$ to be the constraints distribution). This will be formulated as Theorem~\ref{thm:chap_syst}.

From our considerations in the previous Section~\ref{sec:var_calc}, it should be clear that to address the comparison problems it is essential to understand the geometry of the admissible variations of the system in question. This will be the content of Lemma~\ref{lem:var_structure}, where an admissible variation $\del_{\xi(t)}\wt X(t)$ along a horizontal admissible path\footnote{Below we will denote base vector fields by $X\in\X(M)$ and their horizontal lifts by $\wt X\in\X(Q)$.} $\wt X(t)$ is described in terms of the splitting $\T Q=\HH Q\oplus\V  Q$. More precisely, this splitting induces the splitting of the second tangent bundle $\T\T Q=\T\HH Q\oplus\T\V  Q$, and the main result of Lemma~\ref{lem:var_structure} is a description of the $\T\V Q$-component of $\del_{\xi(t)}\wt X(t)$. This, in turn, allows to determine these generators $\xi(t)$ for which this component vanishes, and consequently the admissible variation $\del_{\xi(t)}\wt X(t)$ is vakonomic (i.e. tangent to $\HH Q$).

Observe that the presence of the splitting $\T Q=\HH Q\oplus\V Q$ allows to decompose the generator $\xi(t)$ into its horizontal and vertical parts, and due to the linearity of $\del_{\xi(t)}\wt X(t)$ with respect to $\xi(t)$ (Proposition~\ref{prop:prop_variation} point \eqref{point:prob_var_1}), we can restrict our attention to two distinct cases: $\xi(t)$ being horizontal, and $\xi(t)$ being vertical. Further, due to the Lie group action on $Q$, vertical objects have a canonical description it terms of the Lie algebra of the structural group, and thus Lemma~\ref{lem:var_structure} is formulated in the language of Lie algebra valued objects. \medskip

The execution of the program sketched above requires, however, some technical preparation. This will be the content of the few following subsections, where we shall introduce technical tools needed to formulate and prove Lemma~\ref{lem:var_structure}. Most of them are standard notions from the theory of connections and $G$-bundles such as: fundamental vector fields, connection forms, curvature, vertical derivatives, etc.    

}


\paragraph{Fundamental vector fields.}
 Denote by $\g$ the tangent space of the Lie group $G$ at the identity $e$ equipped with the left Lie algebra structure $[\cdot,\cdot]_\g:\g\times\g\ra\g$. Note that for any $q\in Q$, since the pointed fibre $(Q_{\pi(q)},q)$ can be canonically identified with $(G,e)$ via the $G$-action, we can identify $\V_qQ$ with $\g=\T_eG$, and therefore there exists a vector bundle isomorphism $\phi_Q:\V Q\ra Q\times\g$. Now for each $a\in \g$ we can construct a \emph{fundamental vector field} $\wt a$ defined by $\wt a=(\phi_Q)^{-1}(Q\times\{a\})$. It is well-known \cite{Kobayashi_Nomizu_1963} that the flow of $\wt a$  is $t\mapsto R_{\exp(t\cdot a)}$, that $(R_g)_\ast \wt a=\wt{\ad_{g^{-1}}a}$, and that the association $a\mapsto \wt a$ is a Lie algebra homomorphism, i.e., $[\wt a,\wt b]=\wt{[a,b]_\g}$ for each $a,b\in\g$. 
  

\paragraph{The canonical splitting and connection 1-forms.} 

The (non-invariant) Chaplygin system on $\pi:Q\ra M$ provides us with a canonical splitting $\T Q=\HH Q\oplus_Q \V Q$. Combining this with the canonical isomorphism $\phi_Q:\V Q\approx Q\times\g$ one gets 
$$\T Q=\HH Q\oplus_Q\left(Q\times \g\right).$$
Using the above identification we can project every vector in $\T Q$ to its $\g$-part. We shall denote this projection by 
$$\T Q\ni Z\longmapsto \cform(Z)\in\g.$$
Usually $\cform$ is called the \emph{1-form of the Ehresmann connection} associated with $\HH Q$.

Note also that the canonical splitting $\T Q=\HH Q\oplus_Q \V Q$ induces the splitting $\T\T Q=\T\HH Q\oplus_{\T Q} \T\V Q$. Again we can combine the latter with the tangent map of the canonical isomorphism $\T\phi_Q:\T\V Q\approx \T Q\times\T\g\approx \T Q\times\g\times\g$ and get
$$\T\T Q=\T\HH Q\oplus_{\T Q}\left(\T Q\times\g\times\g\right).$$
It follows that every vector in $\T\T Q$ can be projected to its $\T\g=\g\times\g$-part:
$$\T_Z\T Q\ni\mathcal{A}\longmapsto \left(\cform(Z),\Cform(\mathcal{A})\right)\in \g\times\g\ ,$$
where by $\Cform(\cdot)$ we denoted the projection to the second copy of $\g$ in $\T\g=\g\times\g$. 
Clearly, this map is simply the 1-form of the lifted {\change Ehresmann connection associated with $\T\HH Q$. }

\paragraph{Horizontal lifts.}
At each point $q\in Q$ the tangent map $\pi_\ast$ is an isomorphism between $\HH_qQ$ and $\T_{\pi(q)}M$. Therefore, given a vector $X\in\T_x M$ and a point $q\in Q$ such that $\pi(q)=x$, we can lift $X$ to a unique horizontal vector $\wt X_q\in \HH_qQ$ such that $\pi_\ast \wt X_q=X$. In other words, we have a canonical vector bundle isomorphism $h:Q\times_M\T M\ra \HH Q$ such that $\wt X_q=h(q,X)$. Applying the lifting procedure point-wise to a base vector field $X\in\X(M)$ we obtain its \emph{horizontal lift} $\wt X\in\X(Q)$. 
\medskip

The construction of the horizontal lift allows us to introduce several interesting geometric structures associated with the structure of a (non-invariant) Chaplygin system on $\pi:Q\ra M$ such as the curvature of $\HH Q$, the map $\B$ (which measures the rate of $G$-invariance of $\HH Q$) and two particular derivatives of the Lagrangian (we will call them the \emph{horizontal} and the \emph{vertical} derivative). We shall describe these in the remaining part of this section. 
\medskip

\paragraph{{\change The curvature of $\HH Q$.}} 
It is well known that for any two base vector fields $X,Y\in \X(M)$, the vector $[\wt X,\wt Y]-\wt{[X,Y]}$ at $q\in Q$ belongs to $\V_qQ$ (i.e., is vertical). Moreover, the association $(X,Y)\mapsto [\wt X,\wt Y]-\wt{[X,Y]}$ is $C^\infty(M)$-linear with respect to both $X$ and $Y$ (i.e., has tensorial character). Therefore it defines a bilinear and skew-symmetric map
$$\wt\RR: Q\times_M\wedge^2\T M\lra \V Q.$$
Combining $\wt \RR$ with the $\g$-projection $\cform:\V Q\subset\T Q\ra\g$ we obtain a bilinear skew-symmetric $\g$-valued map
\begin{equation}\label{eqn:curv}
\RR:=\cform\circ\wt\RR:Q\times_M\wedge^2\T M\lra\g\ ,
\end{equation}
called the \emph{curvature} of the horizontal distribution $\HH Q$. Clearly, the curvature measures the rate of non-integrability of $\HH Q$ at a given point $q\in Q$.
\medskip

\paragraph{\change The measure of the $G$-invariance of $\HH Q$.} 
Similarly as above, observe that for any base vector field $X\in\X(M)$ and for any $a\in\g$, the Lie bracket $[\wt X,\wt a]$ at $q\in Q$ is vertical. Moreover, the association $(X,a)\mapsto[\wt X,\wt a]$ is $\C^\infty (M)$-linear (tensorial) with respect to $X$ and $\R$-linear with respect to $a$. Therefore it defines a bilinear map
$$\wt\B:Q\times_M\T M\times \g\ra VQ.$$
Combining $\wt \B$ with  the $\g$-projection $\cform:\V Q\subset\T Q\ra\g$ gives us a $\g$-valued bilinear map 
\begin{equation}\label{eqn:B}
\B:=\cform\circ\wt\B:Q\times_M\T M\times\g\ra \g.
\end{equation} 
The map $\B$ measures the non-invariance of the horizontal distribution with respect to the $G$-action as the following remark explains.

\begin{remark}[{\change The case of a $G$-invariant horizontal distribution}]\label{rem:curv_B_inv} For a given $a\in\g$ consider the curve $t\mapsto\exp(t\cdot a)\in G$ (i.e.,  the flow of the fundamental vector field $\wt a$). \newnew{Zrezygnowałem z oznaczania $q\cdot \exp(t\cdot a)$ przez $q_a(t)$. To raczej zaciemnia obraz zamiast go upraszczać.} Now for a given $X\in\T_{\pi(q)}M$ consider a curve 
$$\left(R_{\exp(t\cdot a)}\right)_\ast \wt X_q-\wt X_{q\cdot \exp(t\cdot a)}\in\T_{q\cdot \exp(t\cdot a)}Q.$$
Clearly the 1-jet of this curve at $t=0$ is a vector in $\T_{0_q}\T Q$. Due to the canonical identification $\T_0\T Q\approx\T Q\times_Q\T Q$, we may represent this vector as a pair of vectors in $\T_q Q$ (in fact, it turns out that both vectors are elements of $\V_q Q\subset\T_qQ$). The first of these vectors is represented by the curve $q\cdot \exp(t\cdot a)$, thus it is the fundamental field $\wt a_q$. By the definition of the Lie derivative, the second is $$\mathcal{L}_{\wt a} \wt X_q=-[\wt a,\wt X]_q=\wt B(q)(X,a).$$
We conclude that if $\HH Q$ is $G$-invariant, then $(R_g)_\ast \wt X_q=\wt X_{q\cdot g}$, and thus $\wt\B(q)(X,a)=0$. 
\medskip

Further, if $\HH Q$ is $G$-invariant, then
$$\RR(q\cdot g)(X,Y)=\ad_{g^{-1}}\RR(q)(X,Y).$$ 
Indeed, from the $G$-invariance of $\HH Q$, we conclude that 
$$\wt \RR(q\cdot g)=(R_g)_\ast\wt \RR(q)(X,Y)=\wt{\ad_{g^{-1}}\RR}(q)(X,Y).$$
\end{remark}

\paragraph{{\change The vertical and the horizontal derivative of the Lagrangian.}} 

{\change Let $L:\T Q\ra \R$ be a Lagrangian. Its} \emph{vertical derivative} $\FL:\HH Q\ra\g^\ast$ is defied 
by the formula\footnote{\change Our notation convention for the vertical derivative $\FL$ follows the literature (e.g. \cite{Favretti_1998,Fernandez_Bloch_2008}). For the notion of the horizontal derivative that will be introduced below (and is not present in the literature) we propose the symbol $\BL$.}
\begin{equation}\label{eqn:FL}
\<\FL(\wt X_q),b>:=\<\dd L, \V_{\wt X_q}\wt b_q>=\frac{\dd}{\dd t}\bigg|_{0}L(\wt X_q+t\cdot \wt b_q)\ ,
\end{equation}
{\change where $b\in \g$ is any element of the Lie algebra.} In other words, $\langle\FL(\wt X_q),b\rangle$ is just the usual fiber-wise derivative of $L$ in the direction of {\change the fundamental vector field} $\wt b$. In the case of the standard ($G$-invariant) Chaplygin system (with a hyper-regular Lagrangian, i.e. the related Legendre map is a global diffeomorphism between $\T Q$ and $\T^\ast Q$), the map $\FL(\wt X)$ coincides with the notion of the  momentum map  (restricted to $\HH Q\subset \T Q$) along a trajectory of the system (cf. Sec. 3 of \cite{Favretti_1998}).\medskip

\newnew{Zmieniłem oznaczenie $\mathrm{B}L\mapsto \BL$.} {\change Now we will introduce a similar notion of a horizontal derivative.} 
Recall the lifting isomorphism $h:Q\times_M\T M\ra\HH Q$.  For a given $X\in\T_x M$ consider the map $h(\cdot,X):Q_x\ra \HH Q$. Now for a given $b\in\g$, by $h_b\wt X_q\in \T_{\wt X_q}\HH Q$ we shall denote the tangent map of $h(\cdot, X)$ evaluated on the fundamental vector field $\wt b_q$. In other words, 
$h_b\wt X_q$ is the 1-jet at $t=0$ of a curve {\change $\wt X_{q\cdot \exp(t\cdot b)}$.} We define the \emph{horizontal derivative} of the Lagrangian 
$\BL:\HH Q\ra\g^\ast$ 
by the formula
\begin{equation}\label{eqn:BL}
\<\BL(\wt X_q),b>:=\<\dd L,h_b\wt X_q>=\frac{\dd}{\dd t}\bigg|_{0}L(\wt X_{q\cdot \exp(t\cdot b)}).
\end{equation}
 {\change In other words, $\BL$ measures how the Lagrangian evaluated on a horizontal vector behaves under the action of the structural group $G$. Contrary to the notion of the horizontal derivative, $\BL$ is not present in the literature as it vanishes under the assumptions of the $G$-invariance of both $L$ and $\HH Q$ (cf. Remark~\ref{rem:L_inv} below and Remark~\ref{rem:curv_B_inv}). Together} the derivatives $\BL$ and $\FL$ allow one to express easily the condition of the symmetry of the Lagrangian.

\begin{remark}[{\change The case of a $G$-invatiant Lagrangian}]\label{rem:L_inv}
Assume that the Lagrangian is invariant with respect to the action of the structural group $G$. Then at every $\wt X\in\HH_q Q$ 
$$\<\BL(\wt X),b>+\<\FL(\wt X),\B(q)(\wt X,b)>=0$$
for every $b\in\g$.

Indeed, since $L$ is $G$-invariant, then for any $g(t)\in G$
$$0=\frac{\dd }{\dd t}\bigg|_{t=0}L\left((R_{g(t)})_\ast \wt X_q\right)=\<\dd L(\wt X),\jet_{t=0}(R_{g(t)})_\ast \wt X_g>.$$
Take now $g(t)=\exp(t\cdot b)$, for $b\in\g$. We can decompose the 1-jet of $(R_{g(t)})_\ast \wt X_q$ into the sum (with respect to $\T\tau:\T\T Q\ra\T Q$) of the 1-jets of
$$(R_{g(t)})_\ast \wt X_q-\wt X_{q\cdot g(t)}\quad \text{and}\quad\wt X_{q\cdot g(t)}$$
By Remark~\ref{rem:curv_B_inv}, the first of these curves corresponds to the vector $(\wt b_q,\wt\B(q)(X,b))\in\T_qQ\oplus\T_qQ\approx\T_{0_q}\T Q$. The second is simply $h_b\wt X_b$. Now the sum of these two vectors is equal to the sum 
$$\V_{\wt X_q}\wt \B(q)(X,b)+h_b\wt X_q$$ 
taken with respect to the vector bundle structure in $\tau_{\T Q}:\T\T Q\ra\T Q$. It follows that 
\begin{align*}
0=\<\dd L(\wt X),\jet_{t=0}(R_{g(t)})_\ast \wt X_g>=&\<\dd L(\wt X),\V_{\wt X_q}\wt \B(q)(X,b)+h_b\wt X_q>\overset{\eqref{eqn:BL}-\eqref{eqn:FL}}=\\
&\<\FL(\wt X),\B(q)(X,b)>+\<\BL(\wt X),b>.
\end{align*}
Note that above we had to apply the addition in $\tau_{\T Q}$ (not $\T\tau_Q$), since $\dd L$ in not linear with respect to the latter vector bundle structure. Note also that we actually used only the invariance of $L$ on horizontal vectors. 
\end{remark}

\paragraph{Local description.}
In order to describe the structure of admissible variations in Chaplygin systems, we need to introduce local coordinates adopted to the structure of the (non-invariant) Chaplygin system on $\pi:Q\ra M=Q/G$. 

Consider any local trivialization $Q\approx M\times G$ of the $G$-bundle $\pi$ and local coordinates $(q^a)=(x^i,g^\alpha)$ adopted to this trivialization (i.e., $(x^i)$ with $i=1,\hdots,m$ are coordinates on $M$ and $(g^\alpha)$ with $\alpha=1,\hdots k$ coordinates on $G$). Choose a local frame $\{e_i\}$ on $\T M$ and a basis $\{e_\alpha\}$ of $\T_eG$. The set of vector fields $\{\wt e_i,\wt e_\alpha\}$ with $i=1,\hdots,m$ and $\alpha=1,\hdots,k$, consisting of the horizontal lifts of fields $e_i$ and the fundamental vector fields associated with elements $e_\alpha\in\g$, is a local frame on $Q$. We introduce a coordinate system $(q^a,v^b)=(x^i,g^\alpha, y^j,a^\beta)$  on $\T Q$ associated with this particular frame (recall our considerations from the first subsection of Section~\ref{sec:prel}). By its very definition these coordinates  are naturally adopted to the splitting $\T Q=\HH Q\oplus_Q\V Q$, i.e., for a vector $Z\in\T Q$ represented by $(x^i,g^\alpha, y^j,a^\beta)$ its $\HH Q$-projection is simply $(x^i,g^\alpha, y^j,0)$ and its $\V Q$-projection is $(x^i,g^\alpha,0,a^\beta)$. Moreover the $\g$-projection of $Z$ is $\cform(Z)=a^\alpha e_\alpha\in\g$, i.e., the considered coordinate system is also naturally compatible with the identification $\phi_Q:\V Q\approx Q\times\g$. 

Consequently, also the induced coordinate system $(q^a,v^b, q'^c,v'^d)=(x^i,g^\alpha, y^j,a^\beta,x'^l, g'^\gamma, y'^l, a'^\delta)$ on $\T\T Q$ is naturally compatible with the induced splitting $\T\T Q=\T\HH Q\oplus_{\T Q} \T\V Q$ and the canonical identification $\T\V Q\approx\T Q\times\g\times\g$. Hence for $\mathcal{A}\in\T\T Q$ represented by $(x^i,g^\alpha, y^j,a^\beta, x'^l, g'^\gamma, y'^l a'^\delta)$ its $\T \HH Q$-projection reads as $(x^i,g^\alpha, y^j,0, x'^l, g'^\gamma, y'^l,0)$, its $\T\V Q$-projection is $(x^i,g^\alpha,0,a^\beta, x'^l, g'^\gamma,0, a'^\delta)$ and, moreover, $\Cform(\mathcal{A})= a'^\alpha e_\alpha\in\g$.  

In the considered situation, rule \eqref{eqn:trans_coord} which relates the induced coordinates $\dot q^a=(\dot x^i,\dot g^\alpha)$ with $v^a=(y^i,a^\alpha)$ takes a special form 
\begin{equation}\label{eqn:coord_chapl}
\begin{split}
\dot x^i=&\A ij(q) y^j,\\
\dot g^i=&\A \alpha j(q) y^j+\A \alpha\beta(g) a^\beta,
\end{split}
\end{equation}
with vanishing entries $\A i\alpha(q)$ of the transition matrix, and entries $\A \alpha\beta(q)$ depending on $G$ only. 

Let $\RR^\alpha_{ij}(q)$ and $\B^\alpha_{i\beta}(q)$ be the coefficients of the maps $\RR$ and $\B$ in the chosen $\g$-basis, i.e.,
\begin{align*}
\RR(q)(X,Y)=& e_\alpha \RR^\alpha_{ij}(q) X^iY^j\\
\B(q)(X,a)=& e_\alpha\B^\alpha_{i\beta}(q)X^ia^\beta,
\end{align*}
where $X=X^i e_i$, $Y=Y^je_j$ and $a=a^\beta e_\beta$. Clearly, the above coefficients are also coefficients of $\wt R$ and $\wt B$ in the basis $\wt e_\alpha$, i.e., $\wt\RR(q)(X,Y)=\wt e_\alpha \RR^\alpha_{ij}(q) X^iY^j$ and
$\wt\B(q)(X,a)=\wt e_\alpha\B^\alpha_{i\beta}(q)X^ia^\beta$.   

From our previous considerations and from the definition of $\wt\RR$ and $\wt\B$ it follows that
\begin{prop}\label{prop:coeff_bracket}
The structure functions of the Lie bracket on $\X(Q)$ with respect to the frame $\{\wt e_i,\wt e_\alpha\}$ are given by
\begin{align*}
&[\wt e_i,\wt e_j]=\wt e_k \br kij(x)+\wt e_\alpha \RR^\alpha_{ij}(q),\\
&[\wt e_i,\wt e_\beta]=\wt e_\alpha B^\alpha_{i\beta}(q),\\
&[\wt e_\beta,\wt e_\gamma]=\wt e_\alpha\bra \alpha\beta\gamma,
\end{align*}
where $\br kij(x)$ are the structure functions of the Lie bracket on $\X(M)$ with respect to the frame $\{e_i\}$ (i.e., $[e_i,e_j]=e_k\br kij(x)$), and $\bra \alpha\beta\gamma$ are the structure constants of the Lie algebra $\g$ in the basis $\{e_\alpha\}$ (i.e., $[e_\beta,e_\gamma]_\g=e_\alpha \bra \alpha\beta\gamma$).   
\end{prop}

Using the local coordinates $(x^i,g^\alpha, y^j,a^\beta)$ we can easily describe the derivatives $\BL$ and $\FL$ introduced before. Namely, for a horizontal vector $\wt X_q\sim(x^i,g^\alpha, y^j,0)$ and $b=b^\alpha e_\alpha\in\g$, the curve $\wt X_q$ corresponds to $(x^i,g^\alpha(t), y^j,0)$, where $g^\alpha(t)$ is a local form of the flow of $b$. Thus the vector $h_b\wt X_q$ is represented by $(\dot x^i=0,\dot g^\alpha=\A \alpha\beta(g)b^\beta, \dot y^j=0,\dot a^\beta=0)$. Similarly, a curve $\wt X_q+t\cdot \wt b_q$ corresponds $(x^i,g^\alpha, y^j,t\cdot b^\beta)$ and thus the vector $\V_{\wt X_q}\wt b_q$ is represented by $(\dot x^i=0,\dot g^\alpha=0, \dot y^j=0,\dot a^\beta=b^\beta)$. We conclude that
\begin{align*}
\<\BL(\wt X),b>&=\frac{\pa L}{\pa g^\alpha}(x,g,y,0)\A \alpha\beta(g)b^\beta,\\
\<\FL(\wt X),b>&=\frac{\pa L}{\pa a^\alpha}(X,g,y,0)b^\alpha.
\end{align*}  


\paragraph{The geometry of admissible variations.} 
In this part we shall study the standard admissible variations (i.e., elements of $\WW^{st}_{\T Q}$) along a given horizontal admissible path in $\TT_{\HH Q}$. Our crucial tool in this study will be the splitting $\T Q=\HH Q\oplus_Q \V Q$ introduced above. 

Consider a horizontal admissible path $\wt X(t)=\jet_tq(t)\in\HH_{q(t)}Q$, being the tangent lift of a horizontal curve $q(t)$. Denote by $x(t)=\pi(q(t))$ the base projection of $q(t)$, and by $X(t)=\pi_\ast \wt X(t)\in \T_{x(t)} M$ the base projection of $\wt X(t)$. Take a generator $\xi(t)\in\T_{q(t)}Q$ of the standard admissible variation $\del_{\xi(t)} \wt X(t)$. 
According to Proposition~\ref{prop:prop_variation} part \eqref{point:prob_var_2} this variation is an element of $\T_{\wt X(t)}\T Q$. Taking into account the induced splitting $\T\T Q=\T\HH Q\oplus_{\T Q}\T\V Q$ and the fact that $\wt X(t)$ is horizontal, we have $\del_{\xi(t)} \wt X(t)\in \T_{\wt X(t)}\HH Q\oplus\T_{\theta_{q(t)}}\V Q$, where $\theta_{q(t)}$ stands for the null vector in $\V_{q(t)} Q\subset\T_{q(t)} Q$. Our goal now will be to describe the $\T\V Q$-part of this variation.

Observe that using the splitting $\T Q=\HH Q\oplus_Q\V Q$, we can decompose the generator $\xi(t)$ itself into its horizontal and vertical parts $\xi(t)=\wt Y(t)+\wt b(t)$, where $\wt Y(t)$ is a horizontal lift of some base curve $Y(t)\in \T_{x(t)} M$ to $q(t)$ and $\wt b(t)$ is a fundamental vector field associated with $b(t)\in \g$ taken at a point $q(t)$. Clearly, due to \eqref{eqn:var_linear}, we have $\del_{\xi(t)}\wt X(t)=\del_{\wt Y(t)}\wt X(t)+\del_{\wt b(t)}\wt X(t)$. In the result below we describe the $\T\V Q$-parts of these two components.

\begin{lem}[{\change The structure of an admissible variation}]\label{lem:var_structure} \newnew{To jest kluczowy wynik, więc chyba lepiej wypisać nasze dane explicite}.
{\change Let $\wt X(t)\in \HH_{q(t)}Q$ be a horizontal curve and $\del_{\xi(t)}\wt X(t)=\del_{\wt Y(t)}\wt X(t)+\del_{\wt b(t)}\wt X(t)$ an admissible variation generated by $\xi(t)=\wt Y(t)+\wt b(t)\in\HH_{q(t)} Q\oplus \V_{q(t)} Q=\T_{q(t)}Q$. Then}
\begin{enumerate}[(i)]
\item \label{point:lem_var_1} In the canonical identification $\T\V Q\approx \T Q\times \g\times\g$, the $\T\V Q$-part of the admissible variation $\del_{\wt b(t)}\wt X(t)$ corresponds to 
\begin{equation}\label{eqn:var_b}
\left(\wt b(t),0,\dot b(t)+\B(q(t))(X(t),b(t))\right)\ .
\end{equation}

\item \label{point:lem_var_2}  The $\T\V Q\approx\T Q\times\g\times\g$-part of the nonholonomic admissible variation $\del_{\wt Y(t)}\wt X(t)$ corresponds to 
\begin{equation}\label{eqn:var_Y}
\left(\wt Y(t),0,\RR(q(t))(X(t),Y(t))\right).
\end{equation}

\item \label{point:lem_var_3} Consequently, the standard admissible variation $\del_{\xi(t)}\wt X(t)$ is tangent to $\HH Q$ (i.e., it is a vakonomic admissible variation) if and only if $b(t)$ satisfies the following linear ODE
\begin{equation}\label{eqn:b_tang_HQ}
\dot b(t)+\B(q(t))(X(t),b(t))+\RR(q(t))(X(t),Y(t))=0. 
\end{equation}
\end{enumerate}
\end{lem}
\begin{proof} Recall the local coordinates $(x^i,g^\alpha,y^j,a^\beta)$ on $\T Q$ and $=(x^i,g^\alpha, y^j,a^\beta,x'^l,g'^\gamma,y'^l,a'^\delta)$ on $\T\T Q$ introduced above. 
The horizontal admissible path $\wt X(t)$ corresponds to a curve $(x^i(t),g^\alpha(t),y^j(t),0)$ (such that $\dot x^i(t)=\A ij(q(t)) y^j(t)$ and $\dot g^\alpha =\A \alpha j(q(t))y^j(t)$), while the generator $\xi(t)$ corresponds to  $(x^i(t),g^\alpha(t),z^j(t),b^\beta(t))$ with the same $x^i(t)$ and $g^\alpha(t)$. Clearly, the horizontal part of $\xi(t)$ is  $\wt Y(t)\sim(x^i(t),g^\alpha(t),z^j(t),0)$ and its vertical part is $\wt b(t)\sim(x^i(t),g^\alpha(t),0,b^\beta(t))$.

In this setting the assertion can be proved by a direct coordinate calculation.  
Applying formula \eqref{eqn:var_coord},
describing the coordinate form of the admissible variation (taking into account the coefficients $\br abc(q)$ of the Lie brackets, and transition matrices $\A ab(q)$ described in Proposition~\ref{prop:coeff_bracket} and in equation \eqref{eqn:coord_chapl}) one easily checks that $\del_{\wt b(t)}\wt X(t)$ corresponds to $\dot y^j=0$, as well as to $\dot x^k=0$, $\dot g^\gamma=\A \gamma\alpha(q)b^\alpha$, $a^\beta=0$ and $\dot a^\delta=\dot b^\delta(t)+\B^\delta_{i\alpha}(q(t))y^i(t) b^\alpha(t)$. The last three of these equations mean that the $\T Q\times\g\times\g$-part of $\del_{\wt b(t)}\wt X(t)$ is precisely \eqref{eqn:var_b}. The fact that the $\T Q$-component is $\wt b(t)$ follows also directly from Proposition~\ref{prop:prop_variation} part \eqref{point:prob_var_3}. This proves part \eqref{point:lem_var_1}.
\medskip

Again by Proposition~\ref{prop:prop_variation} part \eqref{point:prob_var_3} the $\T Q$-component of $\del_{\wt Y(t)}\wt X(t)$ is simply $\wt Y(t)$. 
A similar calculation as before shows that for this variation $a^\gamma=0$ and  $\dot a^\beta=\RR^\beta_{i j}(q(t))y^i(t) z^j(t)$, which proves part \eqref{point:lem_var_2}.
\medskip

 From the linearity of the variation \eqref{eqn:var_linear}, we conclude that \eqref{eqn:b_tang_HQ} is satisfied if and only the $\g\times\g$-component in the $\T\V Q\approx \T Q\times\g\times \g$-part of the variation $\del_{\xi(t)}\wt X(t)$ vanishes. But this, in turn, means that the $\T\V Q$-part of this variation is trivial, and hence that the variation belongs to $\T\HH Q$. This proves part \eqref{point:lem_var_3}.
\end{proof}

\begin{remark}
It follows from the above proof and from local forms of vectors $h_b\wt X$ and $\V_{\wt X}\wt b$ (considered at the end of the previous subsection) that we can decompose the variation $\del_{\wt b}\wt X$ into the following sum (with respect to the vector bundle structure $\tau_{\T Q}:\T\T Q\ra\T Q$)  
$$\del_{\wt b}\wt X= h_b\wt X+\V_{\wt X}\left(\dot b+\B(q)(X,b)\right).$$
Hence, in the light of \eqref{eqn:BL} and \eqref{eqn:FL}, the derivative  of $L$ at $\wt X$ in the direction of $\del_{\wt b}{\wt X}$ reads as
\begin{equation}\label{eqn:var_L_b}\begin{split} 
&\<\dd L(\wt X(t)), \del_{\wt b(t)}\wt X(t)>=\<\BL(\wt X(t)),b(t)>+\<\FL(\wt X(t)),\dot b(t)+\B(q(t))(X(t),b(t))>=\\
&\phantom{XXX}\<\BL(\wt X(t))-\frac{\dd}{\dd t}\FL(\wt X(t)),b(t)>+\<\FL(\wt X(t)),\B(q(t))(X(t),b(t))>-\frac{\dd}{\dd t}\<\FL(\wt X(t)),b(t)>.
\end{split}\end{equation}   
\end{remark}

\subsection{\change The comparison problems on Chaplygin systems}

\paragraph{The comparison problem.}
Now we are ready to formulate our main result. Its part \eqref{part:nh_vak} completely solves the comparison problem \eqref{main_question_1} for the (non-invariant) Chaplygin systems. Part \eqref{part:vak_nh} solves completely a variant of an inverse problem \eqref{main_question_2} when a vakonomic extremal corresponds to a particular choice of a Lagrange multiplier, whereas part characterizes \eqref{part:nh_st} these nonholonomic trajectories which are simultaneously extremals of an unconstrained dynamics. 

\begin{thm}\label{thm:chap_syst}
For the (non-invariant) Chaplygin system introduced above:
\begin{enumerate}[(a)]
\item \label{part:nh_st} A nonholonomic extremal $\wt X(t)\in \HH_{q(t)}Q$ is an unconstrained one if and only if
\begin{equation}\label{eqn:nh_st}
\<\BL(\wt X(t))-\frac{\dd}{\dd t}\FL(\wt X(t)),b>+\<\FL(\wt X(t)),\B(q(t))(X(t),b)>=0
\end{equation}
for every $t\in[t_0,t_1]$ and every vector $b\in\g$. 

\item \label{part:nh_vak} A nonholonomic extremal $\wt X(t)\in \HH_{q(t)}Q$  is a vakonomic one if and only if
\begin{equation}\label{eqn:nh_vak}
\int_I\<\BL(\wt X(t)),b(t)>\dd t=\int_I\<\FL(\wt X(t)),\RR(q(t))(X(t),Y(t))>\dd t
\end{equation}
for each pair $\wt Y(t)\in\HH_{q(t)}Q$ and $b(t)\in\g$ vanishing at the end-points and related by equation \eqref{eqn:b_tang_HQ}.

\item \label{part:vak_nh} A vakonomic extremal $\wt X(t)\in \HH_{q(t)}Q$ being a solution of an unconstrained problem with the modified Lagrangian $\wt L(Z,t)=L(Z)-\<\lambda(t),\cform(Z)>$, {\change for some multiplier $\lambda:I\ra\g^\ast$,} is a nonholonomic extremal if and only if
\begin{equation}\label{eqn:vak_nh}
\<\lambda(t),\RR(q(t))(X(t),Y)>=0
\end{equation} 
for every $t\in[t_0,t_1]$ and every vector $Y\in\T_{\pi(q(t))}M$. 
\end{enumerate}
\end{thm}
\begin{proof}
We shall first prove item \eqref{part:nh_st}. Our idea is very simple. Let us take a nonholonomic extremal $\wt X(t)\in\HH_{q(t)}Q$ and a generator $\xi(t)=\wt Y(t)+\wt b(t)$ (vanishing at the end-points) and consider the associated standard admissible variation $\del_{\xi(t)}\wt X(t)$. In accordance with the spirit of Proposition~\ref{prop:two_var_prob}, we would like to compare this variation with some nonholonomic admissible variation (with vanishing end-points). The splitting of the generator $\xi(t)=\wt Y(t)+\wt b(t)$ provides a natural candidate for such a variation, namely $\del_{\wt Y(t)}\wt X(t)$. From the linearity of the variation with respect to the generator \eqref{eqn:var_linear}, we have
\begin{equation}\label{eqn:comp_var_chapl}
\<\dd L(\wt X(t)),\del_{\xi(t)}\wt X(t)>=\<\dd L(\wt X(t)),\del_{\wt Y(t)}\wt X(t)>+\<\dd L(\wt X(t)),\del_{\wt b(t)}\wt X(t)>
\end{equation}
or, in the integrated version,
$$\<\dd S_L(\wt X),\del_{\xi}\wt X>=\<\dd S_L(\wt X),\del_{\wt Y}\wt X>+\<\dd S_L(\wt X),\del_{\wt b}\wt X>.$$
Since $\wt X(t)$ is a nonholonomic extremal, ot follows $\<\dd S_L(\wt X),\del_{\wt Y}\wt X>=0$, and thus 
\begin{equation}\label{eqn:proof_thm_1}
\<\dd S_L(\wt X),\del_{\xi}\wt X>=\<\dd S_L(\wt X),\del_{\wt b}\wt X>.
\end{equation}
Integrating \eqref{eqn:var_L_b} we get
\begin{align*}\<\dd S_L(\wt X),\del_{\wt b}\wt X>=&\int_I \<\BL(\wt X(t))-\frac{\dd}{\dd t}\FL(\wt X(t)),b(t)>+\<\FL(\wt X(t)),\B(\wt X(t),b(t))>\dd t+\\
&\<FL(\wt X(t)),b(t)>\bigg|^{t_1}_{t_0}.
\end{align*}
By \eqref{eqn:proof_thm_1}, vanishing of $\<\dd S_L(\wt X),\del_{\wt b}\wt X>$ ,for every $b(t)\in\g$ vanishing at the end-points, is a necessary and sufficient condition for $\wt X(t)$ to be an unconstrained extremal. In the light of the above equation, it is equivalent to the vanishing of the integrand for every such $b(t)$. This proves item \eqref{part:nh_st}. 
\medskip 

To prove \eqref{part:nh_vak} we shall proceed analogously with a modification that $\xi(t)$ should now be a generator of a vakonomic admissible variation (still vanishing at the end-points). By Lemma~\ref{lem:var_structure}  such generators are characterized by equation \eqref{eqn:b_tang_HQ}. Therefore we can modify \eqref{eqn:var_L_b} to the following form 
$$\<\dd L(\wt X(t)),\del_{\wt b(t)}\wt X(t)>=\<\BL(\wt X(t)),b(t)>-\<\FL(\wt X(t)),\RR(q(t))(X(t),Y(t))>.$$
Thus $\<\dd S_L(\wt X),\del_{\wt b}\wt X>=0$ (and hence $\wt X(t)$ is a vakonomic extremal) if and only if
$$\int_I\<\BL(\wt X(t)),b(t)>-\<\FL(\wt X(t)),\RR(q(t))(X(t),Y(t))>\dd t=0$$
for every $\wt Y(t)$ and $b(t)$ as considered above. 
\medskip 

The proof of item \eqref{part:vak_nh} is conceptually not much different from the proofs of the two previous parts. Let us start with explaining why the Lagrangian modified by a multiplier takes the form $\wt L(Z,t)=L(Z)-\<\lambda(t), \cform(Z)>$ for some $\lambda(t)\in\g^\ast$. This becomes clear, in the light of Remark~\ref{rem:multipliers}, if one observes that the horizontal distribution $\HH Q\subset \T Q$ is characterized by the equation $\cform(Z)=0$, where $Z\in\T Q$. 

Take now any nonholonomic admissible variation $\del_{\wt Y(t)}\wt X(t)$ with vanishing end-points. Since this variation is, in particular, also a standard admissible variation with vanishing end-points, and from the fact that $\wt X(t)$ is a solution of an unconstrained problem with the modified Lagrangian, we know that
$$\<\dd S_{\wt L}(\wt X),\del_{\wt Y}\wt X>=0.$$

Now observe that
\begin{align*}\<\dd {\wt L}(\wt X(t)),\del_{\wt Y(t)}\wt X(t)>=&\<\dd  L(\wt X(t)),\del_{\wt Y(t)}\wt X(t)>-\<\lambda(t), \Cform\left(\del_{\wt Y(t)}\wt X(t)\right)>\overset{\eqref{eqn:var_Y}}=\\
&\<\dd  L(\wt X(t)),\del_{\wt Y(t)}\wt X(t)>-\<\lambda(t), \RR(q(t))(X(t),Y(t))>
\end{align*}
and thus, after integrating, 
$$0=\<\dd S_{\wt L}(\wt X),\del_{\wt Y}\wt X>=\<\dd S_{ L}(\wt X),\del_{\wt Y}\wt X>-\int_I\<\lambda(t), \RR(q(t))(X(t),Y(t))>\dd t.$$
We conclude that 
$$\<\dd S_{ L}(\wt X),\del_{\wt Y}\wt X>=\int_I\<\lambda(t), \RR(q(t))(X(t),Y(t))>\dd t,$$
and hence $\<\dd S_{ L}(\wt X),\del_{\wt Y}\wt X>=0$ (i.e., $\wt X(t)$ is a nonholonomic extremal) if and only if the above integral vanishes for every $Y(t)$. By the standard argument this implies \eqref{eqn:vak_nh}.
\end{proof}

\paragraph{\change Determination of the vakononomic multiplier.}

Let us now explore some natural questions related with our results.
First of all, as simple consequence of our consideration from the proof of Theorem~\ref{thm:chap_syst}, we get the following characterization of vakonomic extremals corresponding to a prescribed multiplier $\lambda(t)\in\g^\ast$.

\begin{prop}\label{prop:chapl_vak} A horizontal curve $\wt X(t)\in\HH_{q(t)}Q$ is a vakonomic extremal associated with a modified Lagrangian $\wt L(Z,t)=L(Z)-\<\lambda(t),\cform(Z)>$ if and only if it satisfies the following two conditions:
\begin{equation}\label{eqn:vak_chapl_1}
\<\dd S_{ L}(\wt X),\del_{\wt Y}\wt X>=\int_I\<\lambda(t), \RR(q(t))(X(t),Y(t))>\dd t,
\end{equation}
for every $Y(t)\in\T_{\pi(q(t))}M$ vanishing at the end-points, and
\begin{equation}\label{eqn:vak_chapl_lambda}
\<\BL(\wt X(t))-\frac{\dd}{\dd t}\left(\FL(\wt X(t))-\lambda(t)\right),b>+\<\FL(\wt X(t))-\lambda(t),\B(q(t))(X(t),b)>=0
\end{equation}
for every $b\in\g$. 
\end{prop}
 Note that formula \eqref{eqn:vak_chapl_lambda} can be understood as a linear \emph{equation defining the vakonomic multiplier}. Observe also that \eqref{eqn:vak_chapl_lambda} for $\lambda(t)=0$ gives condition \eqref{eqn:nh_st} from Theorem~\ref{thm:chap_syst}. This is not just a coincidence: an unconstrained extremal is a vakonomic extremal with trivial vakonomic multiplier. 

\begin{proof}
The horizontal curve $\wt X(t)$ is a vakonomic extremal associated with the multiplier $\lambda(t)$ if and only if 
$\<\dd S_{\wt L},\del_{\xi}\wt X>$ vanishes on every
admissible variation $\del_{\xi(t)}\wt X(t)$ with vanishing end-points. Splitting the generator into its horizontal and vertical parts $\xi(t)=\wt Y(t)+\wt b(t)$, and using the linearity \eqref{eqn:var_linear} it amounts to check conditions $\<\dd S_{\wt L},\del_{\wt Y}\wt X>=0$ and $\<\dd S_{\wt L},\del_{\wt b}\wt X>=0$ separately. 

In light of the proof of Theorem~\ref{thm:chap_syst} \eqref{part:vak_nh}, condition $\<\dd S_{\wt L},\del_{\wt Y}\wt X>=0$ is equivalent to \eqref{eqn:vak_chapl_1}.
Now, 
\begin{align*}\<\dd {\wt L}(\wt X(t)),\del_{\wt b(t)}\wt X(t)>=&\<\dd  L(\wt X(t)),\del_{\wt b(t)}\wt X(t)>-\<\lambda(t), \Cform\left(\del_{\wt b(t)}\wt X(t)\right)>\overset{\eqref{eqn:var_b}}=\\
&\<\dd  L(\wt X(t)),\del_{\wt b(t)}\wt X(t)>-\<\lambda(t), \dot b(t)+\B(q(t))(X(t),b(t))>\overset{\eqref{eqn:var_L_b}}=\\
&\<\BL(\wt X(t)),b(t)>+\<\FL(\wt X(t))-\lambda(t),\dot b(t)+\B(q(t))(X(t),b(t))>
\end{align*}
Integrating the above equality by parts (and using the fact that $b(t)$ vanishes at the end-points) we get
\begin{align*}
&\<\dd S_{\wt L}(\wt X),\del_{\wt b}\wt X>=\\
& =\int_I\left[\<\BL(\wt X(t))-\frac{\dd}{\dd t}\left(\FL(\wt X(t))-\lambda(t)\right),b(t)>+\<\FL(\wt X(t))-\lambda(t),\B(q(t))(X(t),b(t))>\right]\dd t.
\end{align*}
By the standard reasoning, the vanishing of this integral for every $b(t)\in\g$ vanishing at the end points is equivalent to \eqref{eqn:vak_chapl_lambda}.
\end{proof}

\paragraph{\change Extremals that are simultaneously vakonomic and nonholonomic.}
Another interesting issue is the relation between parts \eqref{part:nh_vak} and \eqref{part:vak_nh} in Theorem~\ref{thm:chap_syst}. Both parts give necessary and sufficient conditions for a trajectory to be simultaneously a vakonomic and a nonholonomic extremal. Therefore we should expect that conditions \eqref{eqn:nh_vak} and \eqref{eqn:vak_nh} are equivalent. This is indeed the case, when the form of the vakonomic multiplier \eqref{eqn:vak_chapl_lambda} is taken into account. Therefore 
condition \eqref{eqn:nh_vak} can be viewed as a version of \eqref{eqn:vak_nh} when we have no explicit knowledge of the vakonomic multiplier. Note, however, that  the equivalence of these two conditions is a non-trivial statement. 

\begin{lem}\label{lem:relation_b_c}
Conditions \eqref{eqn:nh_vak} and \eqref{eqn:vak_nh} are equivalent. More precisely, \eqref{eqn:vak_nh} for some multiplier $\lambda(t)$ satisfying \eqref{eqn:vak_chapl_lambda} implies \eqref{eqn:nh_vak}. Conversely, if \eqref{eqn:nh_vak} holds, then there exists a multiplier $\lambda(t)$ satisfying \eqref{eqn:vak_chapl_lambda} such that \eqref{eqn:vak_nh} holds. 
\end{lem}
\begin{proof}
Choose and admissible path $\wt X(t)\in\HH_{q(t)}Q$. 
Consider any generator of a vakonomic variation (i.e., a pair $Y(t)\in\T_{\pi(q(t))}M$ and $b(t)\in\g$ satisfying \eqref{eqn:b_tang_HQ}) and a multiplier $\lambda(t)\in\g^\ast$ satisfying \eqref{eqn:vak_chapl_lambda}. We shall show that 
\begin{equation}\label{eqn:b_c}
-\<\lambda(t),\RR(q)(X,Y)>=\frac{\dd}{\dd t}\<\lambda(t)-\FL(\wt X),b>+\<\BL(\wt X),b>-\<\FL(\wt X),\RR(q)(X,Y)>.
\end{equation}
Indeed, the above formula can be justified by the following calculation (for the simplicity of notation we do not write the time dependence explicitly):
\begin{align*}
-\<\lambda,\RR(q)(X,Y)>\overset{\eqref{eqn:b_tang_HQ}}=&\<\lambda,\dot b+\B(q)(X,b)>=\frac{\dd}{\dd t}\<\lambda,b>-\<\dot\lambda,b>+\<\lambda,\B(q)(X,b)>\overset{\eqref{eqn:vak_chapl_lambda}}=\\
&\frac{\dd}{\dd t}\<\lambda,b>+\<\BL(\wt X),b>-\<\frac{\dd}{\dd t}\FL(\wt X),b>+\<\FL(\wt X),\B(q)(X,b)>=\\
&\frac{\dd}{\dd t}\<\lambda-\FL(\wt X),b>+\<\BL(\wt X),b>+\<\FL(\wt X),\dot b+\B(q)(X,b)>\overset{\eqref{eqn:b_tang_HQ}}=\\
&\frac{\dd}{\dd t}\<\lambda-\FL(\wt X),b>+\<\BL(\wt X),b>-\<\FL(\wt X),\RR(q)(X,Y)>.
\end{align*}  
Assume now that \eqref{eqn:vak_nh} holds, i.e., the left-hand side of \eqref{eqn:b_c} vanishes. Restrict our attention to those $b(t)$'s that are the solutions of \eqref{eqn:b_tang_HQ} and additionally vanish at the end-points. Integrating \eqref{eqn:b_c} for such $b(t)$'s we get \eqref{eqn:nh_vak}.

The passage from \eqref{eqn:nh_vak} to \eqref{eqn:vak_nh} requires a more attention. Consider a class of solutions $b(t)=b_{Y}(t)$ of \eqref{eqn:b_tang_HQ} (for all possible $Y(t)$'s) with the initial condition $b(t_0)=0$. Of course we have no guarantee that $b(t_1)=0$. The crucial observation is that, if \eqref{eqn:nh_vak} holds, then the value of the linear functional
$$b(t) \longmapsto \<\dd S_L(\wt X),\del_{\wt b}\wt X>$$
(for the considered class of $b(t)$'s) depends on $b(t_1)$ only. Indeed, if $b(t)$ and $b'(t)$ are two solutions (for $Y(t)$ and $Y'(t)$, respectively) such that $b(t_1)=b'(t_1)$ then, since \eqref{eqn:b_tang_HQ} is linear, the difference $\Delta b(t)=b(t)-b'(t)$ is another solution (corresponding to $\Delta Y(t)=Y(t)-Y'(t)$), but now vanishing at the end points. Now from the proof of Theorem~\ref{thm:chap_syst} \eqref{part:nh_vak} we know that
$$\<\dd S_L(\wt X),\del_{\wt{\Delta b}}\wt X>=\int_{I}\<\BL(\wt X),b>-\<\FL(\wt X),\RR(q)(X,\Delta Y)>\dd t\overset{\eqref{eqn:nh_vak}}=0.$$
We conclude that 
$$\<\dd S_L(\wt X),\del_{\wt\Delta b}\wt X>= \<\dd S_L(\wt X),\del_{\wt b}\wt X>- \<\dd S_L(\wt X),\del_{\wt b'}\wt X>=0.$$
Hence, since any linear function on $b(t_1)\in\g$ is determined by an element of $\g^\ast$, there exists $\alpha\in\g^\ast$ such that
$$\int_{I}\<\BL(\wt X),b>-\<\FL(\wt X),\RR(q)(X,Y)>=\<\dd S_L(\wt X),\del_{\wt b}\wt X>=\<\alpha,b(t_1)>,$$
for any $b(t)$ from the considered class. Taking this into account and integrating \eqref{eqn:b_c} we get:
$$-\int_I\<\lambda,\RR(q)(X,Y)>\dd t=\<\lambda(t_1)-\FL(\wt X(t_1))+\alpha,b_Y(t_1)>.$$
Now it is enough to choose $\lambda(t)$ satisfying \eqref{eqn:vak_chapl_lambda} such that $\lambda(t_1)=\FL(\wt X(t_1))-\alpha$ to guarantee that $\int_I\<\lambda,\RR(q)(X,Y)>\dd t=0$
for any $Y(t)$. This implies \eqref{eqn:vak_nh}.  
\end{proof}


\paragraph{The symmetric case, relation with the classical results.}
Let us now see how our results look in the special cases of a (non-invariant) Chaplygin system subject to some symmetry conditions. We shall distinguish three particular situations: when $\HH Q$ is $G$-invariant, when $L$ is $G$-invariant, and the (standard) Chaplygin case (i.e., both the constraints and the Lagrangian are $G$-invariant). In the first case

\begin{cor}[invariant constraints]\label{cor:inv_constr} Assume that the constraints $\HH Q\subset\T Q$ are $G$-invariant (i.e., $(R_g)_\ast \wt X_q=\wt X_{q\cdot g}$ for any $g\in G$ and $X\in\T_{\pi(q)}M$). Then, by Remark~\ref{rem:curv_B_inv}, $\B\equiv 0$ and thus
\begin{itemize}
\item Equation \eqref{eqn:b_tang_HQ} defining the vakonomic variation reads as
$\dot b(t)+\RR(q(t))(X(t),Y(t))=0.$
\item Condition \eqref{eqn:nh_st} in Theorem~\ref{thm:chap_syst} \eqref{part:nh_st} reduces to $\frac{\dd}{\dd t}\FL(\wt X(t))=\BL(\wt X(t))$. 
\item the vakonomic Lagrange multiplier $\lambda(t)$ is a solution of the equation  $\frac{\dd}{\dd t}\left(\FL(\wt X(t))-\lambda(t)\right)=\BL(\wt X(t))$
\end{itemize}  
\end{cor}  

For the invariant Lagrangian we have
\begin{cor}[invariant Lagrangian]\label{cor:inv_lagr}
Assume that the Lagrangian $L:\T Q\ra\R$ is $G$-invariant (i.e., $L\left((R_g)_\ast Z\right)= L(Z)$ for any $g\in G$ and $Z\in\T_q Q$). Then, by Remark~\ref{rem:L_inv}, $\<\BL(\wt X),b>+\<\FL(\wt X),\B(q)(\wt X,b)>$ vanishes for every $b\in\g$, and thus
\begin{itemize}
\item Condition \eqref{eqn:nh_st} in Theorem~\ref{thm:chap_syst} \eqref{part:nh_st} reduces to $\FL(\wt X(t))=\mathrm{const}$. 
\item the vakonomic Lagrange multiplier $\lambda(t)$ is a solution of the equation  $\<\frac{\dd}{\dd t}\left(\FL(\wt X(t))-\lambda(t)\right),b>=\<\lambda(t),\B(q(t))(X(t),b)>$
for every $b\in\g$. 
\end{itemize}
\end{cor}  

Finally, if both the constraints and the Lagrangian are $G$-invariant then
\begin{cor}[the Chaplygin case]\label{cor:chaplygin}
For the standard Chaplygin system we have $\B\equiv 0$ and $\BL(\wt X)=0$, and consequently 
\begin{itemize}
\item Equation \eqref{eqn:b_tang_HQ} defining the vakonomic variation reads
$\dot b(t)+\RR(q(t))(X(t),Y(t))=0.$
\item Condition \eqref{eqn:nh_st} in Theorem~\ref{thm:chap_syst} \eqref{part:nh_st} reduces to $\FL(\wt X(t))=\mathrm{const}$. 
\item Condition \eqref{eqn:nh_vak} in Theorem~\ref{thm:chap_syst} \eqref{part:nh_vak} reduces to $\int_I\<\FL(\wt X(t)),\RR(q(t))(X(t),Y(t))>=0$.
\item the vakonomic multiplier takes the form $\lambda(t)=\FL(\wt X(t))+\textrm{const}$.
\end{itemize}  
\end{cor}  

Now we can relate our results to the classical results from \cite{Crampin_Mestdag_2010,Favretti_1998,Fernandez_Bloch_2008}.

\begin{remark}\label{rem:comparison}
\begin{enumerate}

\item In the (standard) Chaplygin case the explicit construction of the vakonomic multiplier (being the momentum map $\FL(\wt X(t))$ shifted by a constant) appeared in \cite{Favretti_1998} in Thm. 3.1 (ii) and Prop. 4, in \cite{Fernandez_Bloch_2008} in Prop. 3 (1) and Cor. 4, as well as along the lines of Sec. 6 in \cite{Crampin_Mestdag_2010}. Our formula \eqref{eqn:vak_chapl_lambda} is much more general as it allows to find the vakonomic multiplier also for systems without symmetry. A similar equation, in the coordinate form, can be found in Prop. 4 in \cite{Crampin_Mestdag_2010}, yet the relation of the multiplier and the momentum $\FL(\wt X(t))$ is not so obvious and the requirement that the multiplier is defined by section over $\HH Q$ is needed. Clearly, in the general Chaplygin case this last requirement may be too strong.    

\item Theorem~\ref{thm:chap_syst} \eqref{part:vak_nh} can be found in \cite{Fernandez_Bloch_2008} Prop. 2 (for systems of mechanical type with regular Lagrangians) and in Theorem 2 in \cite{Crampin_Mestdag_2010} (if the multiplier is defined by a section over $\HH Q$). Actually, it is a consequence of the general formula of Rumianstev \cite{Rumianstev_1978}, which requires the explicit knowledge of the vakonomic multiplier. 

As we can derive the explicit value of the vakonomic multiplier in the (standard) Chaplygin case (given in general by \eqref{eqn:vak_chapl_lambda}), one can formulate Theorem~\ref{thm:chap_syst} \eqref{part:vak_nh} in this specific setting. This is exactly Thm. 3.1 in \cite{Favretti_1998} (where the regularity of the Lagrangian is assumed), Prop 3. (2) in \cite{Fernandez_Bloch_2008} (for Abelian Chaplygin systems with regular mechanical Lagrangians), as well as Cor. 1 and Prop. 6 in \cite{Crampin_Mestdag_2010} (the Lagrangian has to be sufficiently regular).
Note that our result works in a more general geometric situation (no symmetry) and without any regularity assumptions.

Actually Favretti in \cite{Favretti_1998} formulates Theorem~\ref{thm:chap_syst} \eqref{part:vak_nh} for invariant affine constraints. In this paper we concentrated solely on the linear case, yet extending Theorem~\ref{thm:chap_syst} to the affine setting does not require much effort.   

\item The criterion from Theorem~\ref{thm:chap_syst} \eqref{part:nh_vak} was so far completely absent in the literature. According to Lemma~\ref{lem:relation_b_c} it can be understood as a version of Theorem~\ref{thm:chap_syst} \eqref{part:vak_nh} in the case that the explicit value of the vakonomic multiplier is unknown (or impossible to derive). Again no regularity condition for the Lagrangian, nor symmetry requirements are needed in this case. 

\item Observations similar to Theorem~\ref{thm:chap_syst} \eqref{part:nh_st} have been considered in the literature \cite{Favretti_1998,Fernandez_Bloch_2008} as the special cases of the more general result for vakonomic systems assuming the vanishing of the multiplier. Concrete examples were already discussed in Remark~\ref{rem:not_trivial}.

\item Fernandez and Bloch made, in \cite{Fernandez_Bloch_2008} Prop. 3 (3), a remark that the property of being conditionally variational (i.e. answering \eqref{main_question} positively) is not affected by adding a base dependent potential to the Lagrangian. In fact, it is obvious from our formula \eqref{eqn:nh_vak} that any change of the Lagrangian which does not change $\BL(\wt X)$ and $\<\FL(\wt X),\RR(q)(X,Y)>$ preserves this property.  
\end{enumerate}

\end{remark}

\newpage

\section{Left invariant systems on Lie groups}\label{sec:lie_groups}

In this section we shall solve the comparison problems \eqref{main_question_1} and \eqref{main_question_2} for a class of systems on Lie groups with left-invariant constraints. Such situations were considered for instance by Koiller \cite{Koiller_1992} under the name \emph{generalized rigid body with constrains} (which term he attributes to Arnold \cite{Arnold_Koz_Nies_2010}). In contrast to the standard treatment, here the invariance of the Lagrangian will not be assumed. These results are related to the (non-invariant) Chaplygin systems considered in the previous Section~\ref{sec:chaplygin} (see Remarks~\ref{rem:Chapl_as_Lie} and \ref{rem:Chapl_as_lie_1}), yet in some cases extend these as explained at the end of Remark~\ref{rem:Chapl_as_Lie}.


\paragraph{Geometric setting.}

Consider a Lie group $H$ and denote by $\h$ its tangent space at the identity, equipped with the canonical left Lie algebra structure $[\cdot,\cdot]_\h:\h\times\h\ra\h$. In the remaining part of this section we shall extensively use the canonical trivialization of $\T H$:
$$H\times\h\approx \T H,\qquad (h,\eta)\longmapsto h_\ast \eta,$$ 
induced by the left action of $H$ on itself. 

Now we shall describe the sets of standard admissible trajectories $\TT^{st}_{\T H}$ and admissible variations $\WW^{st}_{\T H}$ within this trivialization. We claim that
\begin{prop}\label{prop:Lie_triv_var} A curve  $(h(t),\eta(t))\in H\times\h$ corresponds to an admissible curve in $\T H$ if and only if 
\begin{equation}\label{eqn:adm_path_g}
h(t)_\ast \eta(t)=\jet_{t}h(t).
\end{equation}
In the induced trivialization $\T\T H\approx\T(H\times\h)\approx\T H\times\T\h\approx\T H\times\h\times\h$, the admissible variation along $(h(t),\eta(t))$ generated by $(h(t),\xi(t))\in\T_{h(t)}H$ corresponds to 
\begin{equation}\label{eqn:adm_var_g}
\left(h(t)_\ast \xi(t),\eta(t),\dot \xi(t)+[\eta(t),\xi(t)]_\h\right).
\end{equation}

\end{prop}
\begin{proof} The justification of formula \eqref{eqn:adm_path_g} is straightforward: standard admissible curves in $\T H$ are simply the tangent lifts of base curves in $H$. Thus a curve $(h(t),\eta(t))$ corresponds to an admissible curve if and only if its image $h(t)_\ast\eta(t)\in\T H$ is the tangent lift of the base projection $h(t)$.  
\medskip

To prove \eqref{eqn:adm_var_g} consider an admissible curve $(h(t),\eta(t))\in H\times\h$ and a generator $(h(t),\xi(t))\in H\times\h$ of the admissible variation $\del_{\xi(t)}\eta(t)\in \T\T H$. By Proposition~\ref{prop:prop_variation} this variation projects to the admissible curve under $\tau_{\T H}$  and to the generator under $\T\tau_{H}$. This explains the first two entries in the triple \eqref{eqn:adm_var_g}. 

To justify the last entry choose a {\change left-invariant frame} $\{e_\alpha\}$ on $H$. Clearly, in the induced coordinates $(h^\alpha, a^\beta)$ on $\T H$ adapted to this frame, the projection to the $\h$-factor in $\T H\approx H\times\h$ reads simply $(h^\alpha, a^\beta)\mapsto a^\beta e_\beta\in \h$.  Moreover, the Lie bracket of left-invariant vector fields $[e_\beta,e_\gamma]=e_\alpha\bra \alpha\beta\gamma$ has constant coefficients $\bra \alpha\beta\gamma$ being the constants of the Lie algebra $\h$. 

Now for $\eta(t)=\eta^\alpha(t) e_\alpha$ and $\xi(t)=\xi^\alpha(t) e_\alpha$, formula \eqref{eqn:var_coord} shows that, in the induced coordinates on $\T\T H$, the $\dot a^\alpha$-entry in the coordinate formula for the admissible variation reads as $\dot \xi^\alpha(t)+\bra \alpha \beta\gamma \eta^\beta(t)\xi^\gamma(t)$. This corresponds precisely to an element $\dot \xi(t)+[\eta(t),\xi(t)]_{\h}$ in the induced projection of $\T\T H\approx \T H\times\h\times\h$ to its second $\h$-factor. 
\end{proof}

Observe that for a given $\eta(t)\in\h$ and an initial point $h(0)=h_0\in H$ equation \eqref{eqn:adm_path_g} determines a unique solution $h(t)$. Therefore,  with some abuse of notation, we shall sometimes refer to $\eta(t)$ itself as to an admissible trajectory (keeping a fixed initial point $h_0$ in mind). Similarly, we shall denote by $\del_{\xi(t)}\eta(t)$ the admissible variation of the form \eqref{eqn:adm_var_g}. 

In the remaining part of this section we shall compare the nonholonomically and vakonomically constrained dynamics associated with a Lagrangian function $L:\T H\approx H\times\h\ra\R$ and a left-invariant distribution $D\subset\T H$. It is easy to see, that such a distribution, in the canonical trivialization $\T H=H\times\h$, corresponds to $H\times d$, where $d\subset\h$ is a linear subspace. We shall denote the restricted variational principles associated with $D$ by $\PP^{nh}_d=(L,\TT_d,\WW^{nh}_d)$ and  $\PP^{vak}_d=(L,\TT_d,\WW^{vak}_d)$. Clearly, an admissible trajectory $\eta(t)\in\h$ belongs to $\TT_d$ if and only if $\eta(t)\in d$ for every $t\in[t_0,t_1]$. 

Consider now any splitting $\h=d\oplus d'$ into the direct sum of linear subspaces. Denote by $P:\h\ra d$ and $P':\h\ra d'$ the canonical projections of $\h$ into the factors of this splitting. Note that every generator $\xi(t)\in \h$ of an admissible variation $\del_{\xi(t)} \eta(t)$ can be decomposed as $\xi(t)=a(t)+b(t)$, where $a(t)=P(\xi(t))\in d$ and $b(t)=P'(\xi(t))\in d'$.
Decomposing the $\g$-part $\dot \xi(t)+[\eta(t),\xi(t)]_\h$ of an admissible variation $\del_{\xi(t)}\eta(t)$ into $d$ and $d'$-components allows to characterize vakonomic admissible variations among all admissible variations. 
 
\begin{prop}\label{prop:Lie_adm_var} In the above setting, an admissible variation $\del_{\xi(t)}\eta(t)$ along an admissible trajectory $\eta(t)\in d$ generated by $\xi(t)=a(t)+b(t)$ belongs to $\WW^{vak}_d$ if and only if
\begin{equation}\label{eqn:Lie_vak_var}
\dot b(t)+P'[\eta(t),a(t)]_\h+P'[\eta(t),b(t)]_\h=0.
\end{equation}

Every nonholonomic admissible variation along $\eta(t)\in d$ is generated by $\xi(t)=a(t)\in d$ and thus the decomposition of its $\h$-part into $d$ and $d'$-components reads as
$$\dot a(t)+[\eta(t),a(t)]_\h=\left(\dot a(t)+P[\eta(t),a(t)]_\h\right)+\left(P'[\eta(t),a(t)]_\h\right)\in d\oplus d'\ .$$
\end{prop}


\paragraph{The comparison problem.}

In this subsection we shall present solutions of the comparison problems \eqref{main_question_1} and \eqref{main_question_2} for systems introduced above. We also answer the question when a nonholonomic extremal is an unconstrained one. In general we follow the line sketched in the previous Section~\ref{sec:chaplygin}, but now the splitting $\h=d\oplus d'$ will play the role of the splitting $\T Q=\HH Q\oplus_Q\V Q$. 

Before formulating our main result note that, due to the canonical decomposition $\T H\approx H\times\h$, we can treat the Lagrangian $L:\T H\ra\R$ as defined on the product $H\times\h$. Therefore we can differentiate $L(h,\eta)$ with respect to $h$ and $\eta$ separately. Now these differentials allow to express nicely the differential of $L$ in the direction of an admissible variation $\del_\xi \eta$:
\begin{equation}\label{eqn:Lie_diff_lagr}
\begin{split}
\<\dd L,\del_{\xi(t)} \eta(t)>\overset{\eqref{eqn:adm_var_g}}=&\<\frac{\pa L}{\pa h},h(t)_\ast\xi(t)>+\<\frac{\pa L}{\pa \eta},\dot\xi(t)+[\eta(t),\xi(t)]_\h>=\\
&\<h(t)^\ast\left(\frac{\pa L}{\pa h}\right)-\frac{\dd}{\dd t}\left(\frac{\pa L}{\pa \eta}\right)+\ad_{\eta(t)}^\ast\left(\frac{\pa L}{\pa \eta}\right),\xi(t)>+\frac{\dd}{\dd t}\<\frac{\pa L}{\pa \eta},\xi(t)>.
\end{split}
\end{equation}  
Here for any $\phi\in\h^\ast$ and any $\eta\in\h$, $\<\ad_\eta^\ast\phi,\cdot>=\<\phi,[\eta,\cdot]_\h>$. The relation of the differentials $h^\ast \frac{\pa L}{\pa h}$ and $\frac{\pa L}{\pa \eta}$ with the differentials $\BL$ and $\FL$ introduced in the previous Section~\ref{sec:chaplygin} will be explained in Remark~\ref{rem:Chapl_as_Lie}. 

Now we are ready to state the main result of this section.

\begin{thm}\label{thm:Lie} For the systems described above:
\begin{enumerate}[(a)]
\item \label{part:lie_nh_st} A nonholonomic extremal $(h(t),\eta(t))\in H\times d$ is an unconstrained extremal if and only if the covector
\begin{equation}\label{eqn:lie_nh_st}
h(t)^\ast\left(\frac{\pa L}{\pa h}\right)-\frac{\dd}{\dd t}\left(\frac{\pa L}{\pa \eta}\right)+\ad_{\eta(t)}^\ast\left(\frac{\pa L}{\pa \eta}\right)\in\h^\ast
\end{equation} 
for every $t\in[t_0,t_1]$ annihilates every $b\in d'\subset\h$. 

\item \label{part:lie_nh_vak} A nonholonomic extremal $(h(t),\eta(t))\in H\times d$ is a vakonomic extremal if and only if
\begin{equation}\label{eqn:lie_nh_vak}
\int_I\<h(t)^\ast\left(\frac{\pa L}{\pa h}\right)+\ad^\ast_{\eta(t)}P^\ast\left(\frac{\pa L}{\pa \eta}\right),b(t)>\dd t=\int_I\<\ad^\ast_{\eta(t)}(P')^\ast\left(\frac{\pa L}{\pa \eta}\right),a(t)> \dd t
\end{equation}
for every pair $a(t)\in d$ and $b(t)\in d'$ vanishing at the end-points and related by equation \eqref{eqn:Lie_vak_var}. 

\item \label{part:lie_vak_nh} A vakonomic extremal $(h(t),\eta(t))\in H\times d$ corresponding to a multiplier $\lambda(t)=(P')^\ast\lambda(t)\in \Ann(d)\approx (d')^\ast\subset\h^\ast$ is a nonholonomic extremal if and only if
\begin{equation}\label{eqn:lie_vak_nh}
\<\lambda(t),P'[\eta(t),a]_\h>=0
\end{equation}
for every $t\in[t_0,t_1]$ and every $a\in d$. 
\end{enumerate}
\end{thm}
\begin{proof} We proceed analogously to the proof of Theorem~\ref{thm:chap_syst}. To prove \eqref{part:lie_nh_st} consider a nonholonomic extremal $(h(t),\eta(t))\in H\times d$ and consider any generator $\xi(t)=a(t)+b(t)$ of the standard admissible variation $\del_{\xi(t)} \eta(t)$ vanishing at the end-points. From \eqref{eqn:var_linear} we have
$$\<\dd S_L,\del_\xi \eta>=\<\dd S_L,\del_a\eta>+\<\dd S_L,\del_{b}\eta>.$$
Now $\<\dd S_L,\del_a\eta>$ vanishes since $\eta(t)$ is a nonholonomic extremal and $\del_a \eta$ is a nonholonomic admissible variation vanishing at the end-points. Clearly $\eta(t)$ is an unconstrained variation if and only if $\<\dd S_L,\del_{b}\eta>$ vanishes for every $b(t)\in d'$ vanishing at the end-points. In light of \eqref{eqn:Lie_diff_lagr} we see that $\<\dd S_L,\del_{b}\eta>$ vanishes if and only if 
$$\int_I\<h(t)^\ast\left(\frac{\pa L}{\pa h}\right)-\frac{\dd}{\dd t}\left(\frac{\pa L}{\pa \eta}\right)+\ad_{\eta(t)}^\ast\left(\frac{\pa L}{\pa \eta}\right),b(t)>\dd t=0.$$
By the standard reasoning we get condition \eqref{eqn:lie_nh_st}.  
\medskip

To prove \eqref{part:lie_nh_vak}, take $(h(t),\eta(t))$ as above and consider a vakonomic admissible variation $\del_{\xi(t)}\eta(t)$ generated by $\xi(t)=a(t)+b(t)$ vanishing at the end-points. From Proposition~\ref{prop:Lie_adm_var} we conclude that curves $a(t)$ and $b(t)$ are related by equation \eqref{eqn:Lie_vak_var}. By assumption that $\eta(t)$ is a nonholonomic extremal, again $\<\dd S_L,\del_a\eta>=0$ and thus $\eta(t)$ is a vakonomic extremal if and only if $\<\dd S_L,\del_{b}\eta>=0$ for every such $b(t)$. Now we can write
\begin{align*}
\<\dd L,\del_{b(t)}\eta(t)>\overset{\eqref{eqn:Lie_diff_lagr}}=&\<\frac{\pa L}{\pa h},h(t)_\ast b(t)>+\<\frac{\pa L}{\pa \eta},\dot b(t)+[\eta(t),b(t)]_\h>\overset{\eqref{eqn:Lie_vak_var}}=\\
&\<h(t)^\ast\left(\frac{\pa L}{\pa h}\right),b(t)>+\<\frac{\pa L}{\pa \eta},P[\eta(t),b(t)]_\h-P'[\eta(t),a(t)]_\h>=\\
&\<h(t)^\ast\left(\frac{\pa L}{\pa \eta}\right)+\ad_{\eta(t)}^\ast P^\ast\left(\frac{\pa L}{\pa \eta}\right),b(t)>-\<\ad_{\eta(t)}^\ast(P')^\ast\left(\frac{\pa L}{\pa \eta}\right),a(t)>.
\end{align*}
Integrating the above equation over $I=[t_0,t_1]$ we get condition \eqref{eqn:lie_nh_vak}. 
\medskip 

To prove \eqref{part:lie_vak_nh}, observe first that the constraints distribution is characterized in $\h$ by the equation $P'(\eta)=0$. Thus the general form of the modified vakonomic Lagrangian is
$$\wt L(h,\eta,t)=L(h,\eta)-\<\lambda(t),P'(\eta)>,$$
where $\lambda(t)\in\h^\ast$. Clearly, since the additional factor in the Lagrangian vanishes for every $\eta\in d$, we can take $\lambda(t)=(P')^\ast \lambda(t)$, thus restricting our attention to $\lambda(t)\in\Ann(d)\approx(d')^\ast$.  

Now take a vakonomic extremal $(h(t),\eta(t))\in H\times d$ associated with such a $\lambda(t)$ and consider any nonholonomic admissible variation $\del_{a(t)}\eta(t)$ with vanishing end-points. Now the second part of Proposition~\ref{prop:Lie_adm_var} implies that
$$\langle\dd \wt L,\del_{a(t)}\eta(t)\rangle=\<\dd L,\del_{a(t)}\eta(t)>-\<\lambda(t),P'[\eta(t),a(t)]_\h>.$$ 
Integrating the above equality over $I$, and taking into account that $\<\dd S_{\wt L},\del_a \eta>=0$ since $\eta(t)$ is an unconstrained extremal of $\wt L$, we get
$$\<\dd S_{ L},\del_a \eta>=\int_I\<\lambda(t),P'[\eta(t),a(t)]_\h>.$$
Clearly $\<\dd S_{L},\del_a \eta>=0$ if and only if condition \eqref{eqn:lie_vak_nh} holds. 
\end{proof}

\paragraph{\change Determining the vakonomic multiplier.}
 Similarly to the Chaplygin case, in the setting considered in this section, we can deduce the equation defining the vakonomic multiplier. Moreover, due to a simple structure of $\T H$, we are able to derive the nonholonomic equations of motion. 

\begin{lem}\label{lem:lie_eqn} An admissible curve $(h(t),\eta(t))\in H\times d$ is:
\begin{itemize}
\item a nonholonomic extremal if and only if the covector \eqref{eqn:lie_nh_st} for every $t\in[t_0,t_1]$ annihilates every $a\in d\subset\h$:
\begin{equation}\label{eqn:lie_nh}
h(t)^\ast\left(\frac{\pa L}{\pa h}\right)-\frac{\dd}{\dd t}\left(\frac{\pa L}{\pa \eta}\right)+\ad_{\eta(t)}^\ast\left(\frac{\pa L}{\pa \eta}\right)\in\Ann(d)\subset\h^\ast.
\end{equation} 
\item a vakonomic extremal associated with a modified Lagrangian $\wt L(h,\eta,t)=L(h,\eta)-\<\lambda(t),P'(\eta)>$ (where $\lambda(t)\in (d')^\ast\approx\Ann(d)\subset\h^\ast$) if and only if 
\begin{align}
\label{eqn:lie_vak_1}
&\<h(t)^\ast\left(\frac{\pa L}{\pa h}\right)-\frac{\dd}{\dd t}\left(\frac{\pa L}{\pa \eta}\right)+\ad_{\eta(t)}^\ast\left(\frac{\pa L}{\pa\eta}-\lambda(t)\right),a>=0
\intertext{for every $t\in[t_0,t_1]$ and every $a\in d$ , and} 
\label{eqn:lie_vak_lambda}
&\<h(t)^\ast\left(\frac{\pa L}{\pa h}\right)-\frac{\dd}{\dd t}\left(\frac{\pa L}{\pa \eta}-\lambda(t)\right)+\ad_{\eta(t)}^\ast\left(\frac{\pa L}{\pa \eta}-\lambda(t)\right),b>=0
\end{align}
for every $b\in d'\subset\h$. 
\end{itemize}
\end{lem}
\begin{proof} The characterization of the nonholonomic extremals follows directly from formula \eqref{eqn:Lie_diff_lagr} taken for $\xi(t)=a(t)\in d\subset \h$.
\medskip

Now, by the linearity of the variation with respect to the generator \eqref{eqn:var_linear}, $(h(t),\eta(t))$ is a vakonomic extremal associated with $\wt L$ if and only if $\<\dd S_{\wt L},\del_{a}\eta>=0$ and $\<\dd S_{\wt L},\del_{b}\eta>=0$ for all generators $a(t)\in d$ and $b(t)\in d'$ vanishing at the end-points. It follows from the proof of Theorem~\ref{thm:Lie} \eqref{part:lie_vak_nh} that the condition $\<\dd S_{\wt L},\del_{a}\eta>=0$ is equivalent to 
$$\<\dd S_L,\del_a \eta>=\int_I\<\lambda(t),P'[\eta(t),a(t)]_\h>\dd t.$$
Now using formula \eqref{eqn:Lie_diff_lagr} and the restriction $\lambda(t)=(P')^\ast\lambda(t)$, we transform the above equality into  
$$\int_I\<h(t)^\ast\left(\frac{\pa L}{\pa h}\right)-\frac{\dd}{\dd t}\left(\frac{\pa L}{\pa \eta}\right)+\ad_{\eta(t)}^\ast\left(\frac{\pa L}{\pa\eta}\right),a(t)>\dd t=\int_I\<\ad_{\eta(t)}^\ast\lambda(t),a(t)>\dd t.$$
The integrands are equal for every $a(t)\in d$ vanishing at the end-points if and only if \eqref{eqn:lie_vak_1} holds. 
\medskip

To justify \eqref{eqn:lie_vak_lambda} let us calculate,
\begin{align*}
\<\dd \wt L,\del_{b(t)}\eta(t)>\overset{\eqref{eqn:adm_var_g}}=&\<\dd L,\del_{b(t)}\eta(t)>-\<\lambda(t),P'\left(\dot b(t)+[\eta(t),b(t)]_\h\right)>=\\
&\<\dd L,\del_{b(t)}\eta(t)>-\<\lambda(t),\dot b(t)+P'[\eta(t),b(t)]_\h>\overset{\eqref{eqn:Lie_diff_lagr}}=\\
&\<h(t)^\ast\left(\frac{\pa L}{\pa h}\right)-\frac{\dd}{\dd t}\left(\frac{\pa L}{\pa \eta}-\lambda(t)\right)+\ad_{\eta(t)}^\ast\left(\frac{\pa L}{\pa \eta}-\lambda(t)\right),b(t)>+\frac{\dd}{\dd t}\<\frac{\pa L}{\pa \eta}-\lambda(t),b(t)>.
\end{align*}  
In the above calculation we used the fact that $\lambda(t)=(P')^\ast\lambda(t)$. Integrating the above equality over $I$ one gets that  $\<\dd S_{\wt L},\del_{b}\eta>=0$ if and only if
$$\int_I\<h(t)^\ast\left(\frac{\pa L}{\pa h}\right)-\frac{\dd}{\dd t}\left(\frac{\pa L}{\pa \eta}-\lambda(t)\right)+\ad_{\eta(t)}^\ast\left(\frac{\pa L}{\pa \eta}-\lambda(t)\right),b(t)>\dd t=0$$
for every $b(t)\in d'\subset\g$ vanishing at the end-points. By the standard reasoning this is equivalent to condition \eqref{eqn:lie_vak_lambda}.  
\end{proof}


\paragraph{\change The role of the splitting $\h=d\oplus d'$.}
Let us now discuss some aspects of our results from the previous subsection. The first matter is the role of the choice of the completing factor $d'\subset\h$. Obviously our characterizations are "if and only if", which suggests that they should not depend on this choice (which, let us remind, was arbitrary). This is indeed the case as we explain below in detail.

\begin{remark}\label{rem:d'_arbitrary}
Consider two splittings $\h=d\oplus {d'}_1$ and $\h=d\oplus {d'}_2$ with corresponding projections $P_1:\h\ra d$, ${P'}_1:\h\ra {d'}_1$ and $P_2:\h\ra d$, ${P'}_2:\h\ra {d'}_2$, respectively. Define $\Delta P:=P_1-P_2:\h\ra d$. Observe that $d\subset\ker \Delta P$  and that $P_2'=P_1'+\Delta P$. In other words the linear $d$-valued map $\Delta P$ describes the passage between both splittings. Our goal now is to show that all $d'$-dependent conditions from our previous considerations are preserved under the substitution of an element $b\in d'_1$ by an element $b+\Delta P(b)\in d_2'$, $P_1$ by $P_2$, etc.

For notation simplicity denote by $\psi(t)$ the covector \eqref{eqn:lie_nh_st} for a given admissible trajectory $\eta(t)$. Clearly, the nonholonomic  equation \eqref{eqn:lie_nh} reads simply $\<\psi(t),a>=0$ for any $a\in d$. Similarly, the vakonomic equation \eqref{eqn:lie_vak_1} reads as $\<\psi(t),a>- \<\lambda(t),[\eta(t),a]_\h>=0$ for every $a\in d$.  

\begin{itemize}
\item Condition \eqref{part:lie_nh_st} from Theorem~\ref{thm:Lie} reads as $\<\phi(t),b>=0$ for every $b\in d'$. Now for $b_2=b_1+\Delta P(b_1)$
$$\<\phi(t),b_2>=\<\phi(t),b_1>+\<\phi(t),\Delta P(b_1)>.$$
By definition $\Delta P(b_1)\in d$, hence if $\eta(t)$ is a nonholonomic trajectory then $\<\phi(t),\Delta P(b_1)>=0$, and thus
$$\<\phi(t),b_2>=\<\phi(t),b_1>.$$

\item Equation \eqref{eqn:Lie_vak_var} for a generator of a vakonomic variation $\xi=a+b$ is, in fact, the equation $P'\left(\dot\xi(t)+[\eta(t),\xi(t)]\right)=0$ or, equivalently, $\dot\xi(t)+[\eta(t),\xi(t)]\in d$. Clearly the latter is independent of the choice of splitting. 

To see it differently, if $P_1'\left(\dot\xi(t)+[\eta(t),\xi(t)]\right)=0$ then also $P_2'\left(\dot\xi(t)+[\eta(t),\xi(t)]\right)=0$ since both differ by $\Delta P\left(\dot\xi(t)+[\eta(t),\xi(t)]\right)=0$ as $\dot\xi(t)+[\eta(t),\xi(t)]\in d$. 

\item The above point will be helpful in showing the splitting-independence of condition \eqref{part:lie_nh_vak} from Theorem~\ref{thm:Lie}. Namely, the above considerations guarantee that if $(a_1,b_1)$, the $d\oplus d_1'$-factors of a generator $\xi$ 
satisfy \eqref{part:lie_nh_vak} with $P'=P_1'$, then $(a_2=a_1-\Delta P(\xi),b_2=b_1+\Delta P(\xi))$ - its $d\oplus d_1'$-factors -satisfy \eqref{part:lie_nh_vak} with $P'=P_2'$. Note also that both pairs simultaneously vanish at the end-points. Now
\begin{align*}
I_2:=&\<h^\ast\left(\frac{\pa L}{\pa h}\right),b_2>+\<\frac{\pa L}{\pa \eta},P_2[\eta,b_2]>-\<\frac{\pa L}{\pa \eta},P_2'[\eta,a_2]>=\\
&\<h^\ast\left(\frac{\pa L}{\pa h}\right),b_1+\Delta P(\xi)>+\\
&\<\frac{\pa L}{\pa \eta},(P_1-\Delta P)[\eta,b_1+\Delta P(\xi)]>-\<\frac{\pa L}{\pa \eta},(P_1'+\Delta P)[\eta,a_1-\Delta P(\xi)]>=\\
&\left(\<h^\ast\left(\frac{\pa L}{\pa h}\right),b_1>+\<\frac{\pa L}{\pa \eta},P_1[\eta,b_1]>-\<\frac{\pa L}{\pa \eta},P_1'[\eta,a_1]>\right)+\\
&+\left(\<h^\ast\left(\frac{\pa L}{\pa h}\right),\Delta P(\xi)>+\<\frac{\pa L}{\pa \eta},[\eta,\Delta P(\xi)]>-\<\frac{\pa L}{\pa \eta},\Delta P[\eta,\xi]>\right)=:I_1+I_0.
\end{align*} 
Now, since $\Delta P(\dot\xi +[\eta,\xi])=0$ we have $\Delta P([\eta,\xi])=-\Delta P(\dot\xi)$, so  we can write
\begin{align*}
I_0=&\<h^\ast\left(\frac{\pa L}{\pa h}\right),\Delta P(\xi)>+\<\frac{\pa L}{\pa \eta},[\eta,\Delta P(\xi)]>+\<\frac{\pa L}{\pa \eta},\Delta P(\dot\xi)>=\\
&\left(\<h^\ast\left(\frac{\pa L}{\pa h}\right),\Delta P(\xi)>+\<\frac{\pa L}{\pa \eta},[\eta,\Delta P(\xi)]>-\<\frac{\dd}{\dd t}\left(\frac{\pa L}{\pa\eta}\right),\Delta P(\xi)>\right)+\\
&+\<\frac{\dd}{\dd t}\left(\frac{\pa L}{\pa\eta}\right),\Delta P(\xi)>+\<\frac{\pa L}{\pa \eta},\Delta P(\dot\xi)>=\<\psi(
t),\Delta P(\xi)>+\frac{\dd}{\dd t}\<\frac{\pa L}{\pa \eta},\Delta P(\xi)>.
\end{align*}
We see that if $\eta(t)$ is a nonholonomic trajectory and $\xi$ vanishes at the end-points then $\int_{I} I_0(t)\dd t=0$ and thus
$$\int_I I_1(t)\dd t=\int_I I_2(t)\dd t,$$
that is,  condition \eqref{part:lie_nh_vak} from Theorem~\ref{thm:Lie} does not depend on the choice of the splitting.

\item  Since we may assume that the vakonomic multiplier $\lambda(t)$ belongs to $\Ann(d)$, condition \eqref{eqn:lie_nh_vak} from Theorem~\ref{thm:Lie} can be rewritten in the splitting-independent form 
$$\<\lambda(t),[\eta(t),a]_\h>=0.$$

\item Finally, observe that the left hand side of \eqref{eqn:lie_vak_lambda} for $b_2=b_1+\Delta P(b_1)$ reads as
\begin{align*}
\<\psi(t),b_2>+\<\frac{\dd }{\dd t}\lambda(t), b_2>-\<\lambda(t),[\eta(t),b_2]>=
\left(\<\psi(t),b_1>+\<\frac{\dd }{\dd t}\lambda(t), b_1>-\<\lambda(t),[\eta(t),b_1]>\right)+\\\left(\<\psi(t),\Delta P(b_1)>+
\<\frac{\dd }{\dd t}\lambda(t), \Delta P(b_1)>-\<\lambda(t),[\eta(t),\Delta P(b_1)]>\right)\ .
\end{align*}
\end{itemize}
We know that $\Delta P(b_1)\in d$, hence $\<\frac{\dd }{\dd t}\lambda(t), \Delta P(b_1)>=0$ as $\lambda(t)\in\Ann(d)$. 
Now if $\eta(t)$ satisfies \eqref{eqn:lie_vak_1}, then 
$\<\psi(t),\Delta P(b_1)>-\<\lambda(t),[\eta(t),\Delta P(b_1)]_\h>=0$. We conclude that 
$$\<\psi(t),b_2>+\<\frac{\dd }{\dd t}\lambda(t), b_2>-\<\lambda(t),[\eta(t),b_2]>=
\<\psi(t),b_1>+\<\frac{\dd }{\dd t}\lambda(t), b_1>-\<\lambda(t),[\eta(t),b_1]>,$$
i.e., equation \eqref{eqn:lie_vak_lambda} is splitting-independent. 
\end{remark}

\paragraph{\change Results for invariant Lagrangians.}
Now we shall discuss Theorem~\ref{thm:Lie} in the special case of systems with an invariant Lagrangian.

\begin{cor}[invariant Lagrangian] \label{cor:lie} Assume that the Lagrangian $L:H\times\h\ra\R$ is $H$-invariant. Then $\frac{\pa L}{\pa h}=0$ and thus:
\begin{itemize}
\item Condition \eqref{eqn:lie_nh_st} in Theorem~\ref{thm:Lie} reduces to 
$$\frac{\dd }{\dd t}\left(\frac{\pa L}{\pa \eta}\right)-\ad_{\eta(t)}^\ast\left(\frac{\pa L}{\pa \eta}\right)\in\Ann(d')\subset\h^\ast.$$ 
\item Condition \eqref{eqn:lie_nh_vak}  in Theorem~\ref{thm:Lie} reduces to 
$$\int_I\<\ad^\ast_{\eta(t)}P^\ast\left(\frac{\pa L}{\pa \eta}\right),b(t)>\dd t=\int_I\<\ad^\ast_{\eta(t)}(P')^\ast\left(\frac{\pa L}{\pa \eta}\right),a(t)> \dd t,$$
where $a(t)$ and $b(t)$ are related by \eqref{eqn:Lie_vak_var} and vanish at the end-points. 
\item The nonholonomic equation of motion \eqref{eqn:lie_nh} reduces to 
$$\frac{\dd }{\dd t}\left(\frac{\pa L}{\pa \eta}\right)-\ad_{\eta(t)}^\ast\left(\frac{\pa L}{\pa \eta}\right)\in\Ann(d)\subset\h^\ast.$$ 
\item Equation \eqref{eqn:lie_vak_lambda} defining the vakonomic multiplier reads as
$$\frac{\dd}{\dd t}\left(\frac{\pa L}{\pa \eta}-\lambda(t)\right)-\ad_{\eta(t)}^\ast\left(\frac{\pa L}{\pa \eta}-\lambda(t)\right)\in \Ann(d')\subset \h^\ast.$$
\end{itemize}  
\end{cor}

\paragraph{\change Relation with Chaplygin systems.} 
It is also interesting to link our results on left-invariant systems on Lie groups with the results on Chaplygin systems from the previous Section~\ref{sec:chaplygin}. Establishing such a connection, however, requires a careful analysis of the relations between the left and right actions of Lie groups (note that in this section we consider left-invariant distributions, whereas in Section~\ref{sec:chaplygin} we dealt with the right action).    

\begin{remark}[System on Lie groups as (non-invariant) Chaplygin systems]\label{rem:Chapl_as_Lie}
Consider now a special case of a left-invariant system on a Lie group $H$ such that the completing linear subspace $d'\subset\h$ is, in fact, a Lie subalgebra $d'=\g\subset\h$ corresponding to a closed Lie subgroup $G\subset H$. 

Now we are in the setting of a generalized Chaplygin system on the principal $G$-bundle $Q=H$ over the homogeneous space $M=H/G$ of right quotients equipped with the canonical right $G$-action (see \cite{Kobayashi_Nomizu_1963}). The splitting $d\oplus\g=\h$ defines the canonical splitting of $\T Q$ into its horizontal and vertical parts: $\HH Q=H\times d\subset H\times\h\approx\T H$ and $\V Q=H\times \g\subset H\times\h\approx \T H$. The latter identification is precisely the canonical trivialization $\V Q\approx Q\times\g$ considered in the previous Section~\ref{sec:chaplygin}. Clearly, since $(R_g)_\ast\eta=(L_g)_\ast\ad_{g^{-1}}\eta$, the right-$G$-action in the trivialization $\T H\approx H\times\h$ reads as $(R_g)_\ast(h,\eta)=(hg,\ad_{g^{-1}}\eta)$. We clearly see that the horizontal distribution is $G$-invariant if and only if $\ad_G d\subset d$ (and thus, in particular $[\g,d]\subset d$). Consequently, unless the latter is satisfied we deal with a truly non-invariant Chaplygin system.

 It is not difficult to translate our considerations from Section~\ref{sec:chaplygin} into the Language of the present section. The canonical identification of the pointed fibre $(hG,h)$ with $(G,e)$ is given by $hg\mapsto g$, and thus the fundamental vector field associated with $b\in \g$ is just the left-invariant vector field $h\mapsto h_\ast b\approx (h,b)$. Given a vector $X\in T M$ and its horizontal lift $\wt X_h= h_\ast\eta\approx (h,\eta)$ at $h\in H$, we conclude that its lift to $hg\in H$ is $\wt X_{hg}\approx (hg,P\left(\ad_{g^{-1}}\eta\right))$. This follows directly from the fact that $\wt X_{qg}$ is the horizontal projection of $(R_q)_\ast \wt X_q$. This can be seen also from a different perspective. Following \cite{Kobayashi_Nomizu_1963} we can identify $\T M$ with the product bundle $H\times_G (\h/\g)$, where $G$ acts on $\h/\g$ by $\ad_{g^{-1}}$. Identifying $\h/\g$ with $d$ we get that $\T M\approx H\times_G d$, where the $G$ action is given by $P\circ \ad_{g^{-1}}$ (the fact that this is indeed a $G$-action follows from the fact that $\ad_G\g\subset\g=\ker P$).  

Now taking horizontal vectors $\wt X\approx(h,\eta) $ and $\wt Y\approx (h, a)$  at $q=h\in H=Q$ (here of course $\eta, a\in d$), and an element $b\in\g$ we easily get
\begin{subequations}
\begin{align}\label{eqn:lie_chapl_1}
\RR(q)(X,Y)=&P'[\eta,a]_\h\\
\B(q)(X,b)=&P'[\eta,b]_\h\\
\<\FL(\wt X),b>=&\<\frac{\pa L}{\pa \eta}, b>
\intertext{and}
\<\BL(\wt X),b>=&\<h^\ast\left(\frac{\pa L}{\pa h}\right)+\ad_{\eta}^\ast P^\ast\left(\frac{\pa L}{\pa \eta}\right),b>.\label{eqn:lie_chapl_4} 
\end{align}
\end{subequations}
The first three formulae follow easily from the respective definitions in Section~\ref{sec:chaplygin}. We will prove the last, not so obvious, formula. Recall that $\<\BL(\wt X),b>=\frac{\dd}{\dd t}\big|_{t=0}L(\wt X_{q\cdot g(t)})$, where $g(t)$ is the flow of $b$. Now, since $\wt X_{q\cdot g(t)}=P(\ad_{g(t)^{-1}}\eta)$, we get  
\begin{align*}
\frac{\dd}{\dd t}\bigg|_{t=0}L(\wt X_{q\cdot g(t)})=&\frac{\dd}{\dd t}\bigg|_{t=0}L\left(h\cdot g(t),P(\ad_{g(t)^{-1}}\eta)\right)=\\
&\<\frac{\pa L}{\pa h},h_\ast b>+\<\frac{\pa L}{\pa \eta},P(-[b,\eta]_\h)>=
\<h^\ast\left(\frac{\pa L}{\pa h}\right)+\ad_{\eta}^\ast P^\ast\left(\frac{\pa L}{\pa \eta}\right),b>.
\end{align*}

Now, one can easily check that, under identifications  \eqref{eqn:lie_chapl_1}--\eqref{eqn:lie_chapl_4}, formula \eqref{eqn:b_tang_HQ} becomes \eqref{eqn:Lie_vak_var}, condition \eqref{eqn:nh_st} becomes \eqref{eqn:lie_nh_st},
condition \eqref{eqn:nh_vak} becomes \eqref{eqn:lie_nh_vak}, and condition \eqref{eqn:vak_nh} becomes \eqref{eqn:lie_vak_nh}. 

In the light of the above considerations, we may understand Theorem~\ref{thm:chap_syst} as a generalization of Theorem~\ref{thm:Lie} -- we can  substitute the Lie group $H$ with a general manifold $Q$. On the other hand, the results of Theorem~\ref{thm:Lie} apply even if the system on $H$ is not a subject of the action of a closed Lie subgroup $G\subset H$ (the completing subspace $d'$ does not have to be a subalgebra). Thus it covers also situations beyond the reach of Theorem~\ref{thm:chap_syst}. 

From a different perspective, we can treat left-invariant systems discussed in this remark as a particular class of examples of generalized Chaplygin systems, and thus understand Theorem~\ref{thm:Lie} restricted to the systems discussed in this remark as a general example of the usage of Theorem~\ref{thm:chap_syst}. 
\end{remark}

\begin{remark}[System on Lie groups as (non-invariant) Chaplygin systems -- another viewpoint]\label{rem:Chapl_as_lie_1}

There is also another way of understanding the geometric situation described in the previous Remark~\ref{rem:Chapl_as_Lie}. Namely, we can treat this system as a Chaplygin system on the left-principal $G$-bundle $H=Q$ over the homogeneous space $G\backslash H$ of  left quotients, equipped with the canonical left $G$-action   
(note that our considerations from Section~\ref{sec:chaplygin} can be repeated also for left principal bundles, under condition that the adjectives "left" and "right" are carefully intertwined). In this situation, again the splitting $d\oplus\g =\h$ defines the splitting of $\T Q$ into its horizontal part $\HH Q= H\times d$ (which is now invariant with respect to the action) and another part $H\times\g$. Yet now, due to the fact that the canonical identification $(Gh,h)\approx (G,e)$ is given by $gh\mapsto g$, the fundamental vector field associated with an element $b\in g$ is $h\mapsto (R_h)_\ast b=(L_h)_\ast \ad_{h^{-1}}b$. It follows that the canonical splitting of $\T_h H$ into its horizontal and vertical part is $(L_h)_\ast d\oplus(R_h)_\ast \g$, and not  $(L_h)_\ast d\oplus(L_h)_\ast \g$. Thus at each point $h\in H$ we have the two different identifications of $\T_h H$ with $d\oplus \g$. The passage between these two splittings is provided by 
\begin{equation}\label{eqn:passage}
(a,b)\longmapsto \left(a'=a+P(\ad_{h^{-1}}b), b'=P'(\ad_{h^{-1}}b)\right).
\end{equation} 

Now given two horizontal vectors $X=(L_h)_\ast\eta$ and $Y=(L_h)_\ast a$ at $h\in H$ (clearly $a,\eta\in d)$ and an element $b\in \g$ one easily shows that
\begin{subequations}
\begin{align}\label{eqn:lie_chapl_1a}
P'\ad_{h^{-1}}\RR(h)(X,Y)=&P'[\eta,a]_\h\\
\B(q)(X,b)=&0\\
\<\FL(\wt X),b>=&\<\frac{\pa L}{\pa \eta}, (R_h)_\ast b>
\intertext{and}
\<\BL(\wt X),b>=&\<\frac{\pa L}{\pa h}, (R_h)_\ast b>.\label{eqn:lie_chapl_4a}
\end{align}
\end{subequations}
We leave as an exercise to check that, under identifications  \eqref{eqn:lie_chapl_1a}--\eqref{eqn:lie_chapl_4a}, formula \eqref{eqn:b_tang_HQ} is equivalent to \eqref{eqn:Lie_vak_var}, condition \eqref{eqn:nh_st} to \eqref{eqn:lie_nh_st},
condition \eqref{eqn:nh_vak} to \eqref{eqn:lie_nh_vak}, and condition \eqref{eqn:vak_nh} to \eqref{eqn:lie_vak_nh}, where the former are derived for pairs $(a',b')$ which are related with pairs $(a,b)$ by \eqref{eqn:passage}. 

We see that in this specific situation we were able to interpret a left-invariant system on a Lie group as a left-Chaplygin system, with invariant horizontal distribution (yet no requirements of the invariance of the Lagrangian were needed). The price that had to be payed for this invariance is the quite complicated passage between the two trivializations of $\T_h H$: $(L_h)_\ast d\oplus (L_h)_\ast\g$ and $(L_h)_\ast d\oplus(R_h)_\ast\g$.  

\end{remark}

\newpage
\section{Examples}\label{sec:examples}
We shall end this paper by considering some well-known examples of nonholonomic systems with linear constraints. All these examples lie in the common setting of Sections~\ref{sec:chaplygin} and \ref{sec:lie_groups} described in Remark~\ref{rem:Chapl_as_Lie}. Thus they can be understood as a practical  demonstration of our results form both Section~\ref{sec:chaplygin} and \ref{sec:lie_groups}.  
\medskip

\begin{example}{The unicycle in a potential field}\label{ex:unicycle}

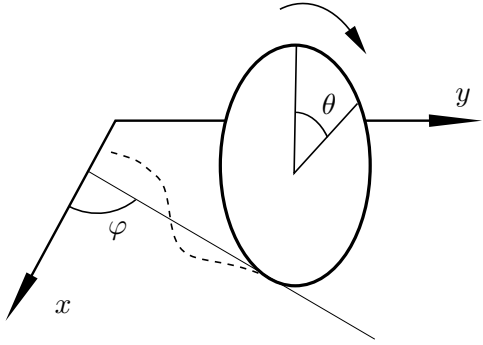
\begin{wrapfigure}[15]{o}{0.4\textwidth}
\begin{center}
\includegraphics[width=0.38\textwidth]{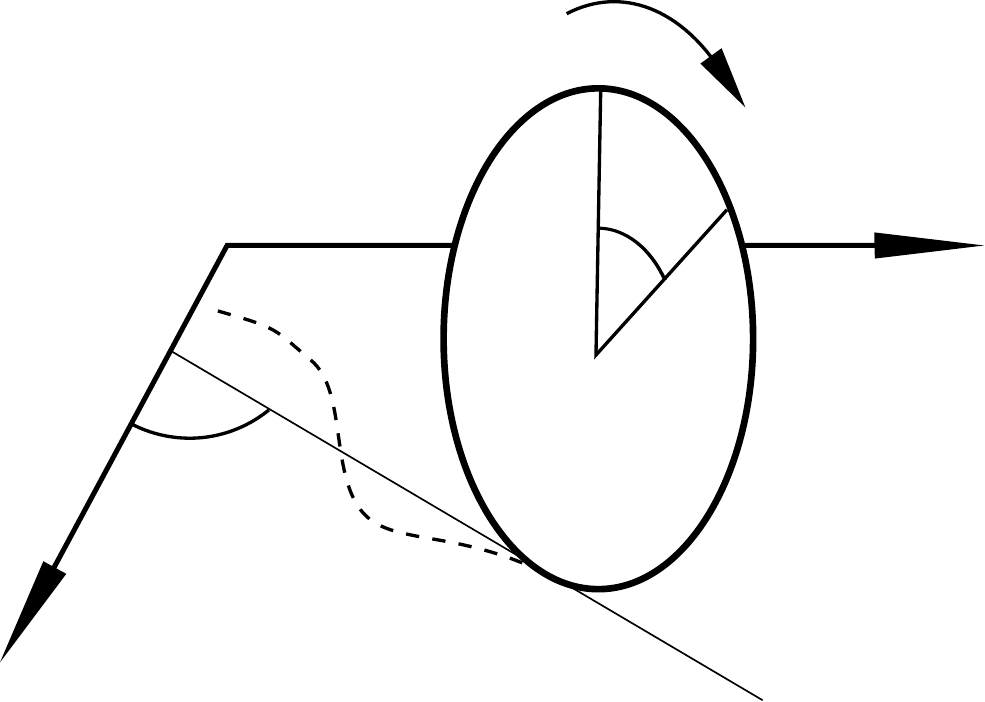}
\put(-160,10){$x$}
\put(-10,90){$y$}
\put(-140,40){$\varphi$}
\put(-60,85){$\theta$}
\end{center}
\caption{The unicycle}\label{fig:unicycle}
\end{wrapfigure} Consider a unicycle moving on the plane without slipping (see figure~\ref{fig:unicycle}). Any configuration of this system is determined by the contact point $(x,y)\in\R^2$, an angle $\varphi\in S^1$, which indicates the direction of movement, and an angle $\theta\in S^1$ describing the rotation of the wheel.
The configuration space is, hence, the Lie group $H=SE(2)\times S^1$ with natural coordinates $(x,y,\varphi,\theta)$. We denote by $m$ and $R$ the mass and the radius of the wheel, and by $I$ and $J$  the inertia of the wheel with respect to $\pa_\theta$- and $\pa_\varphi$-axes, respectively. 

We would like to study motions of the unicycle determined by the purely kinetic Lagrangian 
$$L(\dot x,\dot y,\dot \varphi, \dot\theta):=\frac 12\left[m(\dot x^2+\dot y^2)+J\dot\varphi^2+I\dot\theta^2\right],$$
and nonholonomic constrains (non-slipping condition):
\begin{align*}
\dot x=R\cos(\varphi)\dot\theta,\\
\dot y=R\sin(\varphi)\dot\theta.
\end{align*}

Introduce the following basis of vector fields on $H$
\begin{align*}
&e_1:=\cos\varphi\pa_x+\sin\varphi\pa_y,& \pa_\varphi,\\
&e_2:=\cos\varphi\pa_y-\sin\varphi\pa_x,& \pa_\theta.
\end{align*}
In fact, these are the left-invariant vector fields on $H$ subject to the following commutation relations 
$$[e_1,e_2]_\h=0,\quad[\pa_\varphi,e_1]_\h=e_2,\quad [\pa_\varphi,e_2]_\h=-e_1\quad\text{and}\quad[\pa_\theta,\cdot]_\h=0.$$
Thus $\{e_1,e_2,\pa_\varphi,\pa_\theta\}$ is a basis of the left Lie algebra $\h$ of the group $H$. Let now $(\alpha,\beta,\dot\phi,\dot\theta_1)$ be the fiber-wise coordinates in $\T H$ with respect to this basis. One easily shows that $\alpha=\dot x\cos\varphi+\dot y\sin\varphi$ and $\beta=\dot y\cos\varphi-\dot x\sin\varphi$. Therefore the Lagrangian reads 
$$L(\alpha,\beta,\dot\varphi,\dot\theta)=\frac 12\left[m(\alpha^2+\beta^2)+J\dot\varphi^2+I\dot\theta^2\right].$$
Since it depends on the $\h$-coefficients only, we conclude that $L$ is a left-$H$-invariant function.

The constraints distribution is characterized by equations $\alpha=R\dot\theta$ and $\beta=0$ and thus it is spanned by fields $e_1+\frac 1R\pa_\theta$ and $\pa_\varphi$. We clearly see that we are in the situation described in Section~\ref{sec:lie_groups}, of a system on the Lie group $H$  with the left-$H$-invariant constraints distribution corresponding to the subspace $d=\spann\{e_1+\frac 1R\pa_\theta,\pa_\varphi\}\subset\h$. The natural choice of the completing subspace is $d'=\spann\{e_1,e_2\}\subset\h$, which is not only a subspace, but also an Abelian subalgebra of $\h$ corresponding to a connected Abelian subgroup of translations $G:=\R^2\subset SE(2)$. For this reason we shall denote $d'=\g$. 

 Within this choice of the completing subspace we are, in fact, in the common setting of Sections~\ref{sec:chaplygin} and \ref{sec:lie_groups} described in Remark~\ref{rem:Chapl_as_Lie}. We can now apply the methods developed in these sections  to study the comparison problems for this system. 
\medskip

We shall first check condition \eqref{eqn:lie_nh_vak}. Observe that the left-hand side of \eqref{eqn:lie_nh_vak} vanishes identically, since $\frac{\pa L}{\pa h}=0$ (the Lagrangian is $H$-invariant) and $P[\eta,b]_\h=0$ for every $\eta\in d$ and $b\in\g$ (one can easily check that $[d,\g]_\h\subset\g$). So does the right-hand side of \eqref{eqn:lie_nh_vak}, since $[d,d]_\h\subset\spann\{e_2\}\subset\g$, and thus $\<\frac{\pa L}{\pa\eta},P'[\eta,a]_\h>=0$ for every $\eta,a\in d$, since $\frac{\pa L}{\pa \eta}=m\alpha e_1^\ast+m\beta e_2^\ast+J\dot\varphi\dd\varphi+I\dot\theta\dd\theta$, which on the constraints distribution reduces to  $\frac{\pa L}{\pa \eta}=m\alpha e_1^\ast+J\dot\varphi\dd\varphi+\frac IR\alpha\dd\theta\subset\Ann\{e_2\}$. We conclude that every nonholonomic extremal of the system is a vakonomic one and thus the system answers \eqref{main_question} positively.

We may further ask which nonholonomic extremals are unconstrained ones. A short calculation (using the commutation relations and the form of $\frac{\pa L}{\pa \eta}$) shows that condition \eqref{eqn:lie_nh_st} is equivalent to
$$ \dot\alpha=0 \quad\text{and}\quad \alpha\dot\varphi=0.$$
To check if these equations are satisfied we need to derive the nonholonomic equations of motion \eqref{eqn:lie_nh}. Again skipping some simple calculations we arrive at 
$$\dot\alpha=0\quad\text{and}\quad \ddot\varphi=0.$$
We conclude that both $\alpha$ and $\dot\varphi$ are constants of motion. Moreover, a nonholonomic extremal is an unconstrained one if and only if at least one of these constants vanishes. 

With a little more effort one can determine the vakonomic  multipliers $\lambda(t)$ for the system in question. The general form of any $\lambda(t)\in \Ann(d)$ is $\lambda(t)=f(t)(e_1^\ast-R\dd\theta)+g(t)e_2^\ast$. One can easily show that equations \eqref{eqn:lie_vak_lambda} read as
$$\dot f=g\dot\varphi+m\dot\alpha\quad\text{and}\quad \dot g=\dot\varphi(m\alpha-f).$$ 

Finally note that we can modify the Lagrangian $L$ by adding any  potential term $U(h)$ invariant in $G$-directions without changing the answer to question \eqref{main_question} and without changing the equation for a vakonomic multiplier. Indeed  in this situation $\<h(t)^\ast\left(\frac{\pa L}{\pa h}\right),b>=0$ for any $b\in \g$ and thus the additional term will not alter the left-hand side of formulae  \eqref{eqn:lie_nh_vak} and \eqref{eqn:lie_vak_lambda}. However, the nonholonomic equation of motion could be affected by such an addition.

 Let us summarize our considerations. 
\begin{conclusions}
The unicycle  answers positively question \eqref{main_question}. In such a situation nonholonomic trajectories are characterized by equations $\alpha=const$ and $\dot\varphi=const$. A nonholonomic trajectory is an unconstrained one if and only if at least one of these constants vanishes.

The  positive answer to \eqref{main_question} will not be affected by adding a  $G$-invariant potential term to the Lagrangian.
\end{conclusions}

\end{example}


\begin{example}{The two-wheeled carriage}\label{ex:2wheels}

This example of a nonholonomically constrained system is introduced following \cite{Favretti_1998, Crampin_Mestdag_2010}. It is a natural generalization of the system studied in Example~\ref{ex:unicycle}. Again it is a system on a Lie group with a left-invariant constraints distribution and, by choosing a proper completing distribution, we can consider it in the common setting of Sections~\ref{sec:chaplygin} and \ref{sec:lie_groups} as discussed in Remark~\ref{rem:Chapl_as_Lie}. 

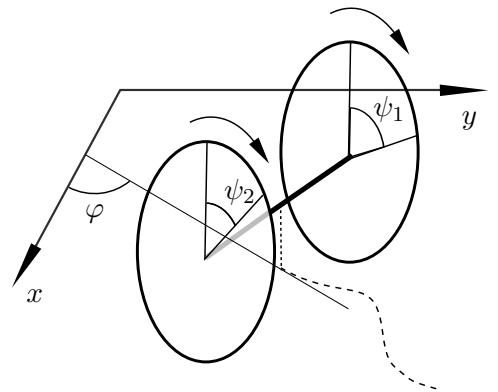
\begin{wrapfigure}[16]{o}{0.4\textwidth}
\begin{center}
\includegraphics[width=0.38\textwidth]{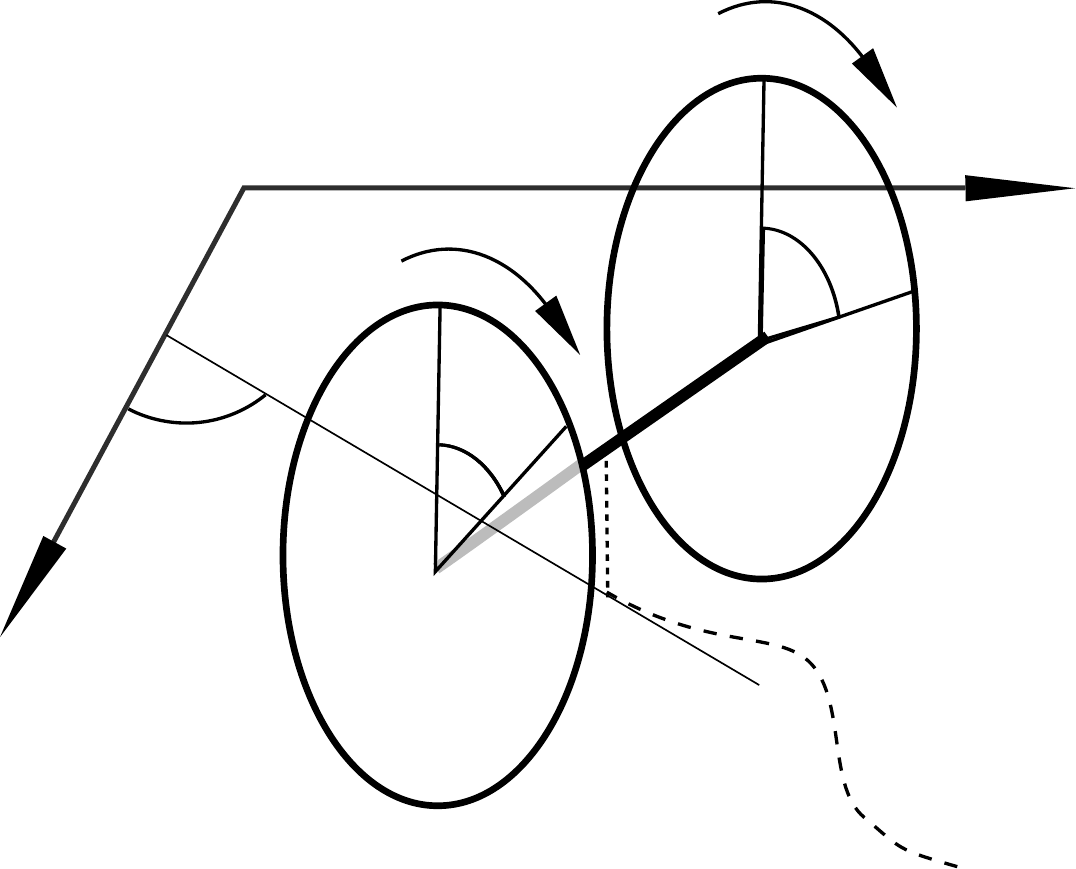}
\put(-173,33){$x$}
\put(-10,100){$y$}
\put(-151,65){$\varphi$}
\put(-43,103){$\psi_1$}
\put(-99,72){$\psi_2$}
\end{center}
\caption{The carriage}\label{fig:carriage}
\end{wrapfigure}

Consider a two-wheeled carriage mowing without slipping on the plane (see figure~\ref{fig:carriage}).  The position of the carriage is determined by an element $\left((x,y,\varphi),\psi_1,\psi_2\right)$ of the Lie group $H=SE(2)\times S^1\times S^1$. Here $(x,y)$ is the position of the center of mass, angle $\varphi$ describes the orientation of the axis, while angles $\psi_1$ and $\psi_2$ specify the rotation of the wheels. 

Parameters of the system include its total mass $m=m_0+2m_1$, being the sum of the mas of the carriage $m_0$ and two masses of the wheels $m_1$; the distance between the center of mass and the axis $l$; the inertia of the whole system  with respect to the $\pa_\varphi$-axis $J$, the inertia of each wheel $I$, the distance between the wheels $2w$ and  the radius of each wheel $R$.

Introduce (following \cite{Favretti_1998}) the variables
$$\theta_1=\frac 12(\psi_1+\psi_2),\quad \theta_2=\frac 12(\psi_1-\psi_2).$$
The kinetic Lagrangian  for this system reads as
$$L(\dot x,\dot y,\dot \varphi, \dot\theta_1,\dot\theta_2):=\frac 12\left[m(\dot x^2+\dot y^2)+J\dot\varphi^2+2I(\dot\theta_1^2+\dot\theta_2^2)\right]+m_0l\dot\varphi(\cos\varphi \dot y-\sin\varphi \dot x).$$

The non-slipping conditions (nonholonomic constrains), which depend on parameters $R$  and $w$, take the form
\begin{align*}
&\dot x\cos(\varphi)+\dot y\sin(\varphi)-R\dot\theta_1=0,\\
&\dot y\cos(\varphi)-\dot x\sin(\varphi)=0,\\
&\dot\varphi-\frac Rw\dot\theta_2=0.
\end{align*}

Introduce now a left-$H$-invariant basis of vector field on $H$
\begin{align*}
&e_1:=\cos\varphi\pa_x+\sin\varphi\pa_y,& \pa_\varphi,\phantom{x\quad\pa_{\theta_2}}\\
&e_2:=\cos\varphi\pa_y-\sin\varphi\pa_x,& \pa_{\theta_1},\quad\pa_{\theta_2};
\end{align*}
which is subject to the following commutation relations
$$[e_1,e_2]_\h=[\pa_{\theta_1},\cdot]_\h=[\pa_{\theta_2},\cdot]_\h=0, \quad [\pa_\varphi,e_1]_\h=e_2\quad\text{and}\quad [\pa_\varphi,e_2]_\h=-e_1.$$ 
Clearly, $\{e_1,e_2,\pa_\varphi,\pa_{\theta_1},\pa_{\theta_2}\}$ form a basis of  the Lie algebra $\h$ of the group $H$.
Let now $(\alpha,\beta,\dot\phi,\dot\theta_1,\dot\theta_2)$ be the fiber-wise coordinates on $\T H$ with respect to this basis. One easily shows that $\alpha=\dot x\cos\varphi+\dot y\sin\varphi$ and $\beta=\dot y\cos\varphi-\dot x\sin\varphi$. Therefore the Lagrangian reads as
$$L(\alpha,\beta,\dot\varphi,\dot\theta_1,\dot\theta_2)=\frac 12\left[m(\alpha^2+\beta^2)+J\dot\varphi^2+2I(\dot\theta_1^2+\dot\theta_2^2)\right]+m_0l\dot\varphi\beta.$$
It is a left-$H$-invariant function, since there is no dependence on the base coordinates. 

The constraints are characterized by the equations $\alpha=R\dot\theta_1$, $\beta=0$ and $w\dot\varphi=R\dot \theta_2$. Thus the constraints distribution is a left-$H$-invariant distribution corresponding to the linear subspace $d=\spann\{e_1+\frac 1R\pa_{\theta_1},\pa_\varphi+\frac wR\pa_{\theta_2}\}\subset \h$. By choosing the completing subspace $d'=\g=\spann\{e_1,e_2,\pa_\phi\}\subset\h$  (which is a subalgebra corresponding to a connected subgroup $G=SE(2)\subset H$), we position ourselves in the situation described in Remark~\ref{rem:Chapl_as_Lie}. Therefore we may apply the methods of Section~\ref{sec:lie_groups} to study the comparison problems for the system. 
\medskip

We shall check condition \eqref{eqn:lie_nh_vak}. By repeating the reasoning from the previous Example~\ref{ex:unicycle} ($\frac{\pa L}{\pa h}=0$ and $[d,\g]_\h\subset\g$) we see that the left-hand side of \eqref{eqn:lie_nh_vak} vanishes identically. Therefore it is enough to check if 
$$\int_I\<\frac{\pa L}{\pa \eta},P'[\eta(t),a(t)]_\h>\dd t=0$$
for the nonholonomic trajectory $\eta(t)$ and $a(t)\in d$ as in part \eqref{part:lie_nh_vak} of Theorem~\ref{thm:Lie}. Now 
$$\frac{\pa L}{\pa \eta}=m\alpha e_1^\ast+m\beta e_2^\ast+J\dot\varphi\dd\varphi+2I\dot\theta_1\dd\theta_1+ 2I\dot\theta_2\dd\theta_2+m_0l\dot\varphi e_2^\ast+m_0l\beta\dd\varphi,$$
which on the constraints distribution reduces to 
$$\frac{\pa L}{\pa \eta}=m\alpha e_1^\ast+J\dot\varphi\dd\varphi+\frac{2I} R\alpha\dd\theta_1+ \frac {2Iw}R\dot\varphi\dd\theta_2+m_0l\dot\varphi e_2^\ast.$$
Take now $\eta(t)=\alpha(t)(e_1+\frac 1R\pa_{\theta_1})+\dot\varphi(\pa_\varphi+\frac wR\pa_{\theta_2})$ and $a(t)=E(t)(e_1+\frac 1R\pa_{\theta_1})+F(t)(\pa_\varphi+\frac wR\pa_{\theta_2})$. Using the commutation relation one easily arrives at 
\begin{equation}\label{eqn:carriage_q2}
\<\frac{\pa L}{\pa \eta},P'[\eta(t),a(t)]_\h>=m_0l\dot\varphi(\alpha F(t)-\dot\varphi E(t)).\end{equation}
We clearly see that if $l=0$ or $\dot\varphi=0$, then the above expression vanishes identically, and thus condition \eqref{eqn:lie_nh_vak} is satisfied. In this way we repeat the observation (made already in \cite{Bloch_Crouch_1993,Crampin_Mestdag_2010,Favretti_1998, Fernandez_Bloch_2008}) 
that for $l=0$ the two-wheeled carriage answers \eqref{main_question} positively. Dealing with the case $l\neq 0$ requires much more attention, namely, we will need the information provided in the nonholonomic equation of motion \eqref{eqn:lie_nh} and in the equation for vakonomic admissible variations \eqref{eqn:Lie_vak_var}.  

A short computation of \eqref{eqn:lie_nh} involving the commutation relations shows that the nonholonomic extremals satisfy 
\begin{subequations}
\begin{align}\label{eqn:alpha}
&\dot\alpha=X(\dot\varphi)^2\\
&\ddot\varphi=-Y\alpha\dot\varphi, \label{eqn:phi}
\end{align}
\end{subequations}
where $X=\frac{m_0lR^2}{mR^2+2I}$ and $Y=\frac{m_0lR^2}{JR^2+2Iw^2}$ depend on the parameters of the system. On the other hand, the pair $a(t)=E(t)(e_1+\frac 1R\pa_{\theta_1})+F(t)(\pa_\varphi+\frac wR\pa_{\theta_2})\in d$ and $b(t)=A(t)e_1+B(t)e_2+C(t)\pa_\varphi\in\g$ satisfies \eqref{eqn:Lie_vak_var} if and only if 
\begin{align*}
\dot A=&\dot\varphi B\\
\dot B=&-\dot\varphi A+\alpha C+(\alpha F-\dot\varphi E)\\
\dot C=&0.
\end{align*} 
In condition \eqref{eqn:lie_nh_vak} we need to restrict our attention to curves vanishing at the end-points, and thus $C(t)\equiv 0$. Now the above system becomes
\begin{subequations}
\begin{align}\label{eqn:a}
\dot A=&\dot\varphi B\\
\dot B=&-\dot\varphi A+(\alpha F-\dot\varphi E).\label{eqn:b}
\end{align} 
\end{subequations}
In the light of condition \eqref{eqn:lie_nh_vak} and equation \eqref{eqn:carriage_q2}, a nonholonomic trajectory (for a carriage with $l\neq 0$) satisfying \eqref{eqn:alpha}--\eqref{eqn:phi} is a vakonomic one if and only if 
$\int_I\dot \varphi(\alpha F-\dot\varphi E)\dd t=0$ for every $(A,B,C,E,F)$ satisfying \eqref{eqn:a}--\eqref{eqn:b} and vanishing at the end-points. Assume that we have such a solution and let us calculate (in the integration by parts we use the fact that $A$ and $B$ vanish at the end-points)
\begin{align*}
\int_I\dot \varphi(\alpha F-\dot\varphi E)\dd t\overset{\eqref{eqn:b}}=&\int_I\dot\varphi\dot B\dd t+\int_I\dot\varphi^2A\dd t\overset{\text{int. by parts}}=
-\int_I\ddot\varphi B\dd t+\int_I \dot\varphi^2 A \dd t\overset{\eqref{eqn:phi}}=\\
&\int Y\alpha\dot\varphi B\dd t+\int_I \dot\varphi^2 A \dd t\overset{\eqref{eqn:a}}=\int_I Y\alpha \dot A+\int_I \dot\varphi^2 A \dd t\overset{\text{int. by parts}}=\\
&-\int_I Y\dot\alpha A+\int_I \dot\varphi^2 A \dd t \overset{\eqref{eqn:alpha}}=\int_I(-XY+1)\dot\varphi^2A\dd t.
\end{align*}
We conclude that if $XY=1$ then every nonholonomic trajectory is a vakonomic one, thus our system answers \eqref{main_question} positively. Condition $XY=1$ describes precisely the special configuration of the system found by a different method by Crampin and Mestdag \cite{Crampin_Mestdag_2010}. 

Let us end the discussion of this example by studying the vakonomic multiplier $\lambda(t)$. The general form of every $\lambda(t)\in \Ann(d)$ is $\lambda(t)=f(t)(e_1^\ast-R\dd\theta_1)+g(t)(\dd\varphi-\frac Rw\dd\theta_2)+h(t)e_2^\ast$. Equation \eqref{eqn:lie_vak_lambda} gives the following set of equations defining the coefficients $f$, $g$ and $h$:
\begin{align*}
\dot f=&m\dot\alpha-m_0l\dot\varphi^2+\dot\varphi h\\
\dot g=&J\ddot\varphi+m_0l\dot\varphi\alpha-\alpha h\\
\dot h=&m_0l \ddot\varphi+m\alpha\dot\varphi-\dot\varphi f.
\end{align*}
Using this equations we may characterize these nonholonomic trajectories which are simultaneously unconstrained ones. Setting $f=g=h=0$ leads to 
\begin{align*}
\dot\alpha=&\frac{m_0l}m\dot\varphi^2\\
\ddot\varphi=&-\frac{m_0l}J\dot\varphi\alpha\\
l\ddot\varphi=&-\frac{m}{m_0}\alpha\dot\varphi.
\end{align*}
For $l=0$ this implies $\dot\alpha=0$, $\ddot\varphi=0$ (which is consistent with the nonholonomic equations \eqref{eqn:alpha}--\eqref{eqn:phi} as for $l=0$ also $X=Y=0$) and $\alpha\dot\varphi=0$. For $l\neq 0$ the comparison of the above equations with the nonholonomic equations of motion \eqref{eqn:alpha}--\eqref{eqn:phi} implies the necessary conditions $\frac{m_0l}m=X$ and $\frac{m_0l}J=\frac{m}{m_0l}=Y$.  One can easily check that these conditions are met if and only if $I=0$ and $XY=1$. Let us summarize our considerations. 

\begin{conclusions}
The two-wheeled carriage with $l=0$ answers positively question \eqref{main_question}. In such a situation nonholonomic trajectories are characterized by equations $\alpha=const$ and $\dot\varphi=const$. Such a trajectory is an unconstrained one if and only if at least one of these constants vanishes.

The two-wheeled carriage with $l\neq 0$ answers positively question \eqref{main_question} only when its parameters are related by the following equation:
$$\frac{m_0^2l^2R^4}{(mR^2+2I)(JR^2+2Iw^2)}=1.$$
In this case a nonholonomic trajectory cannot by an unconstrained one unless $I=0$ or $\dot\varphi=0$. 

Analogously to Example~\ref{ex:unicycle} we could modify the Lagrangian by adding to it a $G$-invariant potential without changing our conclusions about property \eqref{main_question} and the equations defining the vakonomic multiplier.  
\end{conclusions}
\end{example}


\begin{example}{The Heisenberg system}\label{ex:heis}

Consider a well-known system (see eg. \cite{Jurdjevic_1997} p. 424) on $H=\R^3\ni(x,y,z)$ equipped with the structure of the Heisenberg group.  Constraints are introduced by a linear equation 
$$\dot z=y\dot x-x\dot y,$$ 
i.e., the corresponding constraints distribution is spanned by vector fields $e_1=\pa_x+y\pa_z$ and $e_2=\pa_y-x\pa_z$. These fields, together with $\pa_z$ form a basis of the Lie algebra $\h$ of $H$. They are subject to the following commutation relations: 
$$[e_1,e_2]_\h=-2\pa_z,\qquad[\pa_z,\cdot]_\h=0.$$
Consider now coordinates $(\alpha, \beta,\gamma)$ on the fibres of $\T H$ adapted to this basis. One easily  shows that $\alpha=\dot x$, $\beta=\dot y$ and $\gamma=\dot z-y\dot x+x\dot y$ (thus the constraints are characterized by $\gamma=0$). 
We shall study the dynamics defined by the following left-$H$-invariant Lagrangian on $\T H$:
$$L(\alpha,\beta,\gamma)=\frac 12\left(\alpha^2+\beta^2+\gamma^2\right).$$

We clearly are in the setting of Section~\ref{sec:lie_groups} with the constraints distribution corresponding to a subspace $d=\spann\{e_1,e_2\}\subset\h$. For the dimensional reasons any choice of the completing subspace $d'\subset \h$ places us in the setting of Remark~\ref{rem:Chapl_as_Lie}. For simplicity let us choose $d'=\g=\spann\{\pa_z\}$ which is a Lie subalgebra of $H$ corresponding to a 1-dimensional subgroup. 
\medskip

In the considered situation condition \eqref{eqn:lie_nh_vak} is trivially satisfied. Indeed, its left-hand side vanishes as $\frac{\pa L}{\pa h}=0$ and $[d,\g]_\h=0$. The vertical derivative
$\frac{\pa L}{\pa \eta}=\alpha \dd x+\beta\dd y+\gamma(\dd z-y\dd x+x\dd y)$ on the constraints distribution reduces to 
$$\frac{\pa L}{\pa \eta}=\alpha \dd x+\beta\dd y\subset\Ann(\g).$$
Now $[d,d]_\h\subset\g$ and thus also the right-hand side of \eqref{eqn:lie_nh_vak} vanishes identically. 

A similar reasoning shows that also condition \eqref{eqn:lie_nh_st} is trivially satisfied in this situation, thus every nonholonomic trajectory is, in fact, also an extremal of an unconstrained system. This reproduces the result from \cite{Bloch_Crouch_1993}. 

With a little more effort one can derive the nonholonomic equations of motion \eqref{eqn:lie_nh} which read simply as
$$\dot\alpha=0\quad\text{and}\quad\dot\beta=0.$$
Also equation \eqref{eqn:lie_vak_lambda} defining the vakonomic multiplier $\lambda(t)\in \Ann(d)\subset\h^\ast$ (such $\lambda(t)$ has to be of the form $f(t)(\dd z-y\dd x +x\dd y)$) is of particularly simple form. Namely,
$$\dot f=0.$$
We can now summarize our considerations.
\begin{conclusions} The Heisenberg system answers positively question \eqref{main_question} . In fact, every nonholonomic extremal is also an extremal of the unconstrained problem.  
\end{conclusions}
\end{example}


\begin{remark}\label{rem:Bloch_Fernandez}
In Proposition 3(5) in \cite{Fernandez_Bloch_2008} Fernandez and Bloch claimed that every nonholonomic system defined by a 2-distribution $D$ on a 3-manifold $M$ cannot be conditionally variational (i.e., the set of nonholonomic trajectories of this system cannot be a subset of the set of vakonomic trajectories) unless $D$ is integrable. The above Example~\ref{ex:heis} contradicts this claim. This can be also seen  directly at the level of the equations of motion. Indeed, we derived above that these are $\dot\alpha=\dot\beta=0$, i.e., $\ddot x=\ddot y=0$. Now $\ddot z=x\ddot y-y\ddot x=0$ and hence each nonholonomic trajectory is a line respecting the constraints. It is a simple exercise to check that such a curve is also an extremal of the unconstrained problem. By Proposition~\ref{prop:constraints} it also a vakonomic extremal. 

Actually Fernandez and Bloch proved the following fact: if a nonholonomic trajectory $\gamma$ of a system defined by a 2-distribution $D$ on a 3-manifold $M$ is a vakonomic one with a Lagrange multiplier $\mu$, then either $\mu\equiv 0$ or $D$ is integrable along $\gamma$. Hence their Proposition 3(5) should be correctly restated as follows:

\begin{prop}
Consider a system given by a 2-distribution $D$ on a 3-manifold $M$. Then every nonholonomic extremal is a vakonomic one if and only if it is also an unconstrained one. 
\end{prop} 
\end{remark}


\begin{example}{The generalized Heisenberg system}\label{ex:gen_heis}

We can generalize previous Example~\ref{ex:heis} by considering Lagrangians  of the following form
$$
L(x,y,z,\alpha,\beta,\gamma)=f(x,y)\alpha^2+g(x,y)\alpha\beta+h(x,y)\beta^2+ \Phi(x,y,z)\gamma^2-U(x,y).$$ 
In this case $L$ is invariant in the $\g$-direction along the constraints distribution $\{\gamma=0\}$ and thus on every nonholonomic trajectory $\<h^\ast\frac{\pa L}{\pa h},b>=0$ for any $b\in\g$. Moreover, on the constraints distribution 
$$\frac{\pa L}{\pa \eta}=\left(2f(x,y)\alpha+g(x,y)\beta\right)\dd x+(2h(x,y)\beta+g(x,y)\alpha)\dd y\in\Ann(\g),$$  
so this covector annihilates $\g$, $[d,\g]_\h=0$ and $[d,d]_\h\subset\g$. We clearly see that conditions \eqref{eqn:lie_nh_vak} and \eqref{eqn:lie_nh_st} are trivially satisfied in this situation. Note that we came to these conclusions without the necessity of deriving the nonholonomic equation of motion, which in this case are quite complicated.

\begin{conclusions} The generalized Heisenberg system answers  positively question \eqref{main_question}. In fact, every nonholonomic extremal is also an extremal of the unconstrained problem.  
\end{conclusions}

\end{example}  

\section*{Conclusions}

In this paper we studied the comparison problems of nonholonomic and vakonomic constrained Lagrangian dynamics for the same set of constraints. Our approach was based on an observation (made already by several researchers \cite{Cardin_Favretti_1996, GG_2008,Gracia_Martin_Munos_2003,Guo_Liu_Inn_2007,Leon_2012} and rooted in a general philosophy of Tulczyjew \cite{Tul_2004}) that nonholonomically and vakonomically constrained Lagrangian systems can be put into the frames of the same unifying variational formalism (called in this paper a \emph{variational principle}). The differences in these two systems appear at the level of admissible variations. The new idea is to concentrate solely on these differences, completely ignoring the resulting differences at the level of the equations of motion. In fact, we understand the equations of motion as secondary objects: the consequences of the underlying variational principle, not the fundamental description of the system. Such a point of view results in simplicity, since admissible variations have a much simpler description, and clearer geometric nature, than the resulting equations. Moreover, concentrating on the variations we can easily relate the comparison problem of both dynamics with the symmetries of the system. 

As a particular realization of this strategy we studied (generalized) Chaplygin systems and left-invariant systems on Lie groups. Using our approach we were able to substantially generalize many classical results for such systems. 

\section*{Acknowledgments}

The authors are grateful for a referee of the previous version of this work for pointing our attention to the paper of Crampin and Mestdag \cite{Crampin_Mestdag_2010}. We also acknowledge prof. Paweł Urbański for reference suggestions. 

This research was supported by the National Science Center under the grant DEC-2011/02/A/ST1/00208 ``Solvability, chaos and control in quantum systems''.

\newpage
\bibliographystyle{acm}
\bibliography{references}

\begin{thebibliography}{10}

\bibitem{Arnold_Koz_Nies_2010}
{\sc Arnold, V.~I., Khukhro, E., Kozlov, V.~V., and Neishtadt, A.~I.}
\newblock {\em Mathematical Aspects of Classical and Celestial Mechanics},
  vol.~3 of {\em Encyclopaedia of Mathematical Sciences}.
\newblock Springer, 2007.

\bibitem{Bloch_Crouch_1993}
{\sc Bloch, A.~M., and Crouch}.
\newblock {\em Nonholonomic and vakonomic control systems on Riemannian
  manifolds}, vol.~1 of {\em Fields Institute Communications}.
\newblock American Mathematical Soc., 1993, pp.~25--52.

\bibitem{Cantrijn_Cortez_Inn_2002}
{\sc Cantrijn, F., Cortés, J., De~León, M., and De~Diego, D.}
\newblock On the geometry of generalized {C}haplygin systems.
\newblock {\em Math. Proc. Cambridge Philos. Soc. 132\/} (2002), 323--351.

\bibitem{Cardin_Favretti_1996}
{\sc Cardin, F., and Favretti, M.}
\newblock On nonholonomic and vakonomic dynamics of mechanical systems with
  nonintegrable constraints.
\newblock {\em J. Geom. Phys. 18}, 4 (1996), 295--325.

\bibitem{Cortes_Leon_Inn_2003}
{\sc Cortés, J., de~León, M., de~Diego, D., and Martínez, S.}
\newblock Geometric description of vakonomic and nonholonomic dynamics.
  {C}omparison of solutions.
\newblock {\em SIAM J. Control Optim. 41}, 5 (2002), 1389--1412.

\bibitem{Crampin_Mestdag_2010}
{\sc Crampin, M., and Mestdag, T.}
\newblock Anholonomic frames in constrained dynamics.
\newblock {\em Dyn. Syst. 25}, 2 (2010), 159--187.

\bibitem{Favretti_1998}
{\sc Favretti, M.}
\newblock Equivalence of dynamics for nonholonomic systems with transverse
  constraints.
\newblock {\em J. Dyn. Diff. Equations 10}, 4 (1998), 511--536.

\bibitem{Fernandez_Bloch_2008}
{\sc Fernandez, O.~E., and Bloch, A.~M.}
\newblock Equivalence of the dynamics of nonholonomic and variational
  nonholonomic systems for certain initial data.
\newblock {\em J. Phys. A: Math. Theor. 41}, 34 (2008), 344005.

\bibitem{GG_2008}
{\sc Grabowska, K., and Grabowski, J.}
\newblock Variational calculus with constraints on general algebroids.
\newblock {\em J. Phys. A: Math. Theor. 41}, 17 (2008), 175204.
\newblock arXiv: 0712.2766.

\bibitem{GGU_2006}
{\sc Grabowska, K., Grabowski, J., and Urbański, P.}
\newblock Geometrical mechanics on algebroids.
\newblock {\em Int. J. Geom. Meth. Mod. Phys. 3\/} (2006), 559--575.
\newblock arXiv: math-ph/0509063.

\bibitem{GLMM_2009}
{\sc Grabowski, J., de~Leon, M., Marrero, J.~C., and de~Diego, D.~M.}
\newblock Nonholonomic constraints: a new viewpoint.
\newblock {\em J. Math. Phys. 50}, 1 (2009), 013520.

\bibitem{Gracia_Martin_Munos_2003}
{\sc Gr\'{a}cia, X., Marín-Solano, J., and M\~{u}oz Lecanda, M.-C.}
\newblock Some geometric aspects of variational calculus in constrained
  systems.
\newblock {\em Rep. Math. Phys. 51}, 1 (2003), 127--148.

\bibitem{Guo_Liu_Inn_2007}
{\sc Guo, Y.-X., Liu, S.-X., Liu, C., Luo, S.-K., and Wang, Y.}
\newblock Influence of nonholonomic constraints on variations, symplectic
  structure, and dynamics of mechanical systems.
\newblock {\em J. Math. Phys. 48}, 8 (2007), 082901.

\bibitem{Jurdjevic_1997}
{\sc Jurdjevic, V.}
\newblock {\em Geometric {C}ontrol {T}heory}.
\newblock Cambridge University Press, 1997.

\bibitem{Kobayashi_Nomizu_1963}
{\sc Kobayashi, S., and Nomizu, K.}
\newblock {\em Foundations of {D}ifferential {G}eometry}, vol.~1.
\newblock Jonh Willeys \& Sons, 1963.

\bibitem{Koiller_1992}
{\sc Koiller, J.}
\newblock Reduction of some classical non-holonomic systems with symmetry.
\newblock {\em Arch. Rational Mech. Anal. 118\/} (1992), 113--148.

\bibitem{Kolar_Michor_Slovak_1993}
{\sc Kolár, I., Michor, P.~W., and Slovák, J.}
\newblock {\em Natural operations in {D}ifferential {G}eometry}.
\newblock Springer, 1993.

\bibitem{Leon_2012}
{\sc León, M.~d.}
\newblock A historical review on nonholomic mechanics.
\newblock {\em Rev. R. Acad. Cienc. Exactas Fis. Nat. Ser. A Math. RACSAM 106},
  1 (2012), 191--224.

\bibitem{Lewis_Murray_1995}
{\sc Lewis, A.~D., and Murray, R.~M.}
\newblock Variational principles for constrained systems: {T}heory and
  experiment.
\newblock {\em Int. J. Non-Linear Mechanics 30}, 6 (1995), 793--815.

\bibitem{Martinez_2008}
{\sc Martínez, E.}
\newblock Variational calculus on {L}ie algebroids.
\newblock {\em ESAIM Control Optim. Calc. Var. 14}, 2 (2008), 356--380.

\bibitem{Rumianstev_1978}
{\sc Rumiantsev, V.~V.}
\newblock On {H}amilton’s principle for nonholonomic systems.
\newblock {\em P. M. M. USSR 42}, 3 (1978), 407--419.

\bibitem{Tul_2004}
{\sc Tulczyjew, W.~M.}
\newblock {\em The origin of variational principles}, vol.~59 of {\em Banach
  Center Publications}.
\newblock Banach Center Publications, 2003, pp.~41--75.

\end{thebibliography}
\end{document}